\documentclass[eqright]{amsart}
\usepackage{amssymb,amscd,amsthm,latexsym}
\usepackage[all]{xy}
\usepackage[british]{babel}
\CompileMatrices

\newtheorem{theorem}{Theorem}[section]
\newtheorem{lemma}[theorem]{Lemma}
\newtheorem{proposition}[theorem]{Proposition}
\newtheorem{corollary}[theorem]{Corollary}
\newtheorem*{property-P}{Property P}

\theoremstyle{definition}

\newtheorem{definition}[theorem]{Definition}

\newtheorem{conditions}[theorem]{Conditions}
\newtheorem{remark}[theorem]{Remark}

\newtheorem{example}[theorem]{Example}
\newtheorem{examples}[theorem]{Examples}

 \newcommand{\sem}[2][]{\mbox{ $[\![ #2 ]\!]^{#1}$}}

\renewcommand{\to}{\rightarrow}

\newcommand{\supp}{\ensuremath{\mathsf{Supp\,}}}

\newcommand{\coim}{\ensuremath{\mathsf{coim\,}}}
\newcommand{\coker}{\ensuremath{\mathsf{coker\,}}}
\newcommand{\Cok}{\ensuremath{\mathrm{Cok}}}
\renewcommand{\ker}{\ensuremath{\mathsf{ker\,}}}

\newcommand{\Cc}{\ensuremath{\mathcal{C}}}

\newcommand{\Fc}{\ensuremath{\mathcal{F}}}
\newcommand{\Bc}{\ensuremath{\mathcal{B}}}
\newcommand{\Tc}{\ensuremath{\mathcal{T}}}
\newcommand{\Ac}{\ensuremath{\mathcal{A}}}
\newcommand{\Mf}{\mathbb M}
\newcommand{\Ef}{\mathbb E}
\newcommand{\Pc}{\ensuremath{\mathcal{P}}}

\newcommand{\Ab}{\ensuremath{\mathsf{Ab}}}
\newcommand{\ab}{\ensuremath{\mathsf{ab}}}
\newcommand{\Dis}{\ensuremath{\mathsf{Dis}(\Ac)}}
\newcommand{\Disn}{\ensuremath{\mathsf{Dis^n}(\Ac)}}

\newcommand{\Nc}{\ensuremath{\mathcal{N}}}

\newcommand{\Z}{\ensuremath{\mathbb{Z}}}

\newcommand{\RG}{\ensuremath{\mathsf{RG}}}

\newcommand{\Set}{\ensuremath{\mathsf{Set}}}

\newcommand{\Gp}{\ensuremath{\mathsf{Gp}}}

\newcommand{\GpnonA}{\ensuremath{\mathsf{Gpd}}}
\newcommand{\EDM}{\ensuremath{\mathsf{EffDes}}}
\newcommand{\Eq}{\ensuremath{\mathsf{Eq(\Ac)}}}
\newcommand{\Gpd}{\ensuremath{\mathsf{Gpd}}}
\newcommand{\Conn}{\ensuremath{\mathsf{ConnGpd(\Ac)}}}
\newcommand{\Gpdn}{\ensuremath{\mathsf{Gpd}^{n}\!} (\Ac)}
\newcommand{\Gpdd}{\ensuremath{\mathsf{Gpd}^{2}\!} (\Ac)}
\newcommand{\Gpdnn}{\ensuremath{\mathsf{Gpd}^{n-1}\!} (\Ac)}
\newcommand{\Rng}{\ensuremath{\mathsf{Rng}}}
\newcommand{\TConn}{\mathsf{ConnComp}^{\mathbb{T}}}
\newcommand{\TCom}{\mathsf{HComp}^{\mathbb{T}}}
\newcommand{\TPro}{\mathsf{TotDis}^{\mathbb{T}}}

\newcommand{\CM}{\ensuremath{\mathsf{XMod}}}

\newcommand{\Ext}{\ensuremath{\mathsf{Ext}}}

\newcommand{\NExt}{\ensuremath{\mathsf{NExt}}}

\newcommand{\Top}{\ensuremath{\mathsf{Top}}}
\newcommand{\Arr}{\ensuremath{\mathsf{Arr}}}

\newcommand{\Arrn}{\ensuremath{\mathsf{Arr}^{n}\!}}

\newcommand{\CExt}{\ensuremath{\mathsf{CExt}}}
\newcommand{\CExtB}{\ensuremath{\mathsf{CExt_{\Bc}}}}

\newcommand{\DiscFib}{\ensuremath{\mathsf{DiscFib}}}

\newcommand{\Ec}{\ensuremath{\mathcal{E}}}
\newcommand{\MC}{\ensuremath{\mathcal{M}}}

\renewcommand{\hom}{\ensuremath{\mathrm{Hom}}}
\newcommand{\op}{\ensuremath{\mathrm{op}}}

\newbox\skewpullbackbox
\setbox\skewpullbackbox=\hbox{\xy 0;<1mm,0mm>: \POS(4,0)\ar@{-} (-4,0) \ar@{-} (8,4)
\endxy}

\newbox\ksewpullbackbox
\setbox\ksewpullbackbox=\hbox{\xy 0;<1mm,0mm>: \POS(0,-8)\ar@{-} (0,-4) \ar@{-} (4,-4)
\endxy}

\newbox\pullbackbox
\setbox\pullbackbox=\hbox{\xy 0;<1mm,0mm>: \POS(4,0)\ar@{-} (0,0) \ar@{-} (4,4)
\endxy}
\def\pullback{\copy\pullbackbox}

\newbox\pushoutbox
\setbox\pushoutbox=\hbox{\xy 0;<1mm,0mm>: \POS(0,4)\ar@{-} (0,0) \ar@{-} (4,4)
\endxy}
\def\pushout{\copy\pushoutbox}

\begin{document}

\title[Protoadditive functors, derived torsion theories and homology]{Protoadditive functors, derived torsion theories and homology}

\author{ Tomas Everaert and Marino Gran}
 \address{Universit\'e catholique de Louvain,
Institut de Recherche en Math\'ematique et Physique,
Chemin du Cyclotron 2,
1348 Louvain-la-Neuve, Belgium }
 \address{
Vakgroep Wiskunde \\
Vrije Universiteit Brussel \\
 Department of Mathematics  \\
 Pleinlaan 2\\
1050 Brussel \\
 Belgium.
 }

\email{marino.gran@uclouvain.be}
\email{teveraer@vub.ac.be}

 \date{\today}

\maketitle

\begin{abstract}
Protoadditive functors are designed to replace additive functors in a non-abelian setting. Their properties are studied, in particular in relationship with torsion theories, Galois theory, homology and factorisation systems. It is shown how a protoadditive torsion-free reflector induces a chain of derived torsion theories in the categories of higher extensions, similar to the Galois structures of higher central extensions previously considered in semi-abelian homological algebra. Such higher central extensions are also studied, with respect to Birkhoff subcategories whose reflector is protoadditive or, more generally, factors through a protoadditive reflector. In this way we obtain simple descriptions of the non-abelian derived functors of the reflectors via higher Hopf formulae. Various examples are considered in the categories of groups, compact groups, internal groupoids in a semi-abelian category, and other ones.
\\


\end{abstract}

\section*{Introduction}
In recent years, the theory of \emph{semi-abelian categories} \cite{JMT} has become a central subject in categorical algebra. Semi-abelian categories allow for a conceptual and unified treatment of the theories of groups, rings, algebras, and similar non-abelian structures, just like, say, abelian categories are suitable for the study of abelian groups and modules, or toposes for investigating the category of sets and categories of sheaves.

As explained in \cite{JMT}, the formulation of the notion of semi-abelian category can be seen as an appropriate solution to an old problem S.~Mac Lane mentioned in his classical article \cite{Dfg}, which, in fact, led to the introduction of the notion of abelian category a few years later \cite{Buchsbaum}.  

With the introduction of any mathematical structure naturally comes the question of defining a suitable notion of morphism. The meaning of ``suitable'' may of course vary, and depends on the applications one has in mind. For instance, between toposes one usually considers so-called ``geometric morphisms'', but the notion of ``logical morphism'' is of importance too. In asking for an appropriate notion of morphism between semi-abelian categories, we should therefore be more specific. As their name suggests, semi-abelian categories are a weaker notion than that of abelian category. Hence, it seems natural to ask if the classical notion of additive functor can be generalised, in a meaningful way, to the non-additive context of semi-abelian categories. We believe the answer is yes, with the sought-after notion being that of ``protoadditive functor'' we introduced in \cite{EG}, and which we intend to investigate more extensively in the present article. 

Before recalling the definition, it is useful to make some comparative remarks on abelian and semi-abelian categories. By a well-known theorem of M.~Tierney, a category is \emph{abelian} if and only if it is both \emph{exact} (in the sense of Barr \cite{Barr}) and \emph{additive}. Now, if we ignore some natural (co)completeness assumptions, \emph{semi-abelian} categories can be defined as exact categories which are also \emph{pointed} and \emph{protomodular} \cite{Bourn0}. Accordingly, a semi-abelian category can be seen as what remains of the notion of abelian category if one replaces ``additivity'' by the weaker (pointed) ``protomodularity'' condition. 

As observed by D.~Bourn, there is a simple way to express the ``difference'' between an additive and a pointed protomodular category. Classically, any split short exact sequence 
\begin{equation}\label{sses1}
\xymatrix{0 \ar[r] & K \ar[r]^-{\ker (f)} & A\ar[r]<-.8 ex>_f & B \ar[l]<-.8 ex>_s \ar[r] & 0
}
\end{equation}
 in an additive category $\Ac$  determines a canonical isomorphism $A \cong K \oplus B$, showing that any split short exact sequence is given by a biproduct. Since this property is actually equivalent to the additivity condition, we no longer have that it holds in an arbitrary pointed protomodular category: for instance, in the semi-abelian category $\mathsf{Grp}$ of groups, split short exact sequences are well known to correspond to semi-direct products, not to products. Nevertheless, it is still the case in any pointed protomodular category that
$A$ is the \emph{supremum} of $\ker (f) \colon K \rightarrow A$ and $s \colon B \rightarrow A$ as subobjects of $A$: $A \cong K \vee B$. In fact, we have that the following stronger property holds in a pointed category \emph{if and only if} it is protomodular: for every split short exact sequence \eqref{sses1},  $\ker (f)$ and $s$ are jointly \emph{extremal} epic (rather than just jointly epic).

Now recall that a functor between additive categories is additive if and only if it preserves (binary) biproducts. Taking into account the correspondence between biproducts and split short exact sequences in an additive category, as well as the above comparison between additive and pointed protomodular categories, it seems natural to call \emph{protoadditive} \cite{EG} any functor between pointed protomodular categories that preserves split short exact sequences. Then, of course, for a functor between additive categories, being protoadditive is the same thing as being additive, but there are many examples of interest beyond the additive context, as we shall see in this article.

This choice of definition is also motivated by the following reformulation. For a finitely complete category $\Ac$, write $\mathsf{Pt}(\Ac)$ for the category of ``points'' in $\Ac$: split epimorphisms with a given splitting. $\mathsf{Pt}(\Ac)$ is fibred over $\Ac$ via the codomain functor $\mathsf{Pt}(\Ac)\to \Ac$, the so-called ``fibration of points'' \cite{Bourn0, Bourn1996}, the cartesian morphisms being pullbacks along split epimorphisms. This fibration has been intensively studied during the past twenty years, mainly in connection to its strong classification properties in algebra (see \cite{BB}, for instance, and references therein). In particular, a category is protomodular if and only if the change of base functors of the fibration of points reflect isomorphisms. Now, it turns out that if a zero preserving functor $F\colon \Ac\to\Bc$ between pointed protomodular categories is protoadditive then it preserves at once \emph{arbitrary} pullbacks along split epimorphisms (and not only of morphisms from the zero-object). In other words, we have that $F$ is protoadditive if and only if the induced functor $\mathsf{Pt}(\Ac)\to\mathsf{Pt}(\Bc)$ between the categories of points  preserves cartesian morphisms, i.e. if it is a \emph{morphism of fibrations}.

The validity of the classical homological diagram lemmas, such as the five lemma or the snake lemma, make semi-abelian categories particularly suitable for a generalised treatment of non-abelian (co)homology theories. Given, moreover,  that the main domain of application of abelian categories and additive functors is homological algebra, it is then natural to investigate the role of protoadditive functors in semi-abelian homological algebra. We started this investigation in \cite{EG} and will continue it in the present article. 

Recall that, for any Birkhoff subcategory $\Bc$ (= a reflective subcategory closed under subobjects and regular quotients) of a semi-abelian monadic category $\Ac$, the Barr-Beck derived functors of the reflector $I\colon \Ac\to \Bc$ can be described via generalised \emph{Hopf formulae}  \cite{EGV}. These ``formulae'' were defined with respect to a certain chain of ``higher dimensional Galois structures'' naturally induced by the reflection. If $\Ac$ is abelian, then the case where $I$ is additive is of particular importance, as in this case we obtain classical abelian derived functors. Also for a semi-abelian category $\Ac$, the case of a protoadditive $I$ is of interest, since in this case the Hopf formulae take a simplified shape. 

In the present article, among other things, we shall be interested in extending the work of \cite{EG} in two directions: on the one hand, we shall consider reflections $F\colon \Ac\to\Fc$ where $\Ac$ need not be semi-abelian, but only homological and such that every regular epimorphism is an effective descent morphism, and where $\Fc$ is a torsion-free subcategory of $\Ac$ (not necessarily Birkhoff) with protoadditive reflector $F$. We prove that such type of reflections also induce a similar chain of higher dimensional Galois structures along with what we call \emph{derived torsion theories}. This shows, in particular, that protoadditivity is of interest also for functors between homological categories (not necessarily semi-abelian).
 On the other hand, we shall consider Birkhoff subcategories of semi-abelian categories whose reflector may itself not be protoadditive but only factors through a protoadditive reflector. Here, once again we shall obtain a simplified description of the derived functors via the associated Hopf formulae.

To give a simple illustration of this, consider the reflection
\begin{equation}\label{profgroups1}
 \xymatrix@=30pt{
{\mathsf{Grp(HComp)}  \, } \ar@<1ex>[r]_-{^{\perp}}^-{{I}} & {\mathsf{Grp(Prof)}, } \ar@<1ex>[l]^-V}
\end{equation}
where ${\mathsf{Grp(HComp)}}$ is the category of compact (Hausdorff) groups, ${\mathsf{Grp(Prof)}}$ the category of profinite groups, $V$ the inclusion and ${I}$ the functor sending a compact group $G$ to the quotient $I(G)= G/\Gamma_0 (G)$ of $G$ by the connected component $\Gamma_0(G)$ of the neutral element in $G$. It is well known that ${\mathsf{Grp(HComp)}}$ is a semi-abelian category with enough projectives, and the functor $I$ can be shown to be protoadditive (see Example \ref{exproto}.\ref{exdisc}). We can then consider a \emph{double presentation}
$$
\label{doublext}
\xymatrix{F \ar[r]^{} \ar[d] & F/K_1 \ar[d] \\
F/K_2  \ar[r] & G }
$$
of a compact group $G$, in the sense that $K_1$ and $K_2$ are closed normal subgroups of a free compact group $F$ with the property that both $F/K_1$ and $F/K_2$ are also free, and the square is a pushout. Then the third homology group of $G$ corresponding to the reflection (\ref{profgroups1}) (i.e. with coefficients in the functor $I$) is given by the formula
\[
H_3 (G, {\mathsf{Grp(Prof)} } ) = \frac{K_1 \cap K_2 \cap (\Gamma_0(F))} { \Gamma_0(K_1 \cap K_2)},
\]
which is therefore independent of the chosen double presentation.
By choosing a different reflective subcategory of ${\mathsf{Grp(HComp)}}$, for instance the category ${\mathsf{Ab(Prof)} }$ of profinite abelian groups, we get the composite reflection 
 $$
 \xymatrix@=30pt{
{\mathsf{Grp(HComp)}  \, } \ar@<1ex>[r]_-{^{\perp}}^-{\ab} & {\mathsf{Ab(HComp)}}
\ar@<1ex>[l]^-U  \ar@<1ex>[r]_-{^{\perp}}^-{\overline{I}} & {\mathsf{Ab(Prof),} } \ar@<1ex>[l]^-V}
$$
where $\ab \colon {\mathsf{Grp(HComp)}} \rightarrow {\mathsf{Ab(Prof)} }$ is the abelianisation functor, and $\overline{I}$ the (additive) restriction of the functor $I$ in (\ref{profgroups1}).
The results in the present article imply in particular that the corresponding homology group of $G$ is given by 
$$H_3 (G, {\mathsf{Ab(Prof)} } ) = \frac{K_1 \cap K_2 \cap (\overline{[F,F]}\cdot \Gamma_0(F))}{\overline{[K_1,K_2]}\cdot \overline{[K_1 \cap K_2, F]}\cdot \Gamma_0(K_1 \cap K_2)},$$
where the symbol $\cdot$ denotes the product of normal subgroups, and $\overline{[. ,.  ] }$ is the topological closure of the commutator subgroup $[. ,. ] $.
Similar formulas are obtained for the $n$-th homology group $H_n (G, {\mathsf{Ab(Prof)} })$ of $G$, for any $n \ge 2$. The same method applies to many other reflections, some of which are studied in the present article. This provides us with another motivation for studying protoadditive functors: their usefulness to ``compute'' the homology objects explicitely in a variety of situations.

Let us then give a brief overview of the different sections of the article.

\vspace{4mm}

\noindent {\bf Structure of the article.}
The first section is preliminary: we recall some definitions and results---concerning torsion theories, categorical Galois theory and reflective factorisation systems---needed in the text.

Section $2$ is mainly devoted to proving alternative characterisations of the protoadditivity condition in various situations. In particular, we show that the protoadditivity of a torsion-free reflector can be detected from a hereditariness condition of the corresponding torsion subcategory (Theorem \ref{protoM}). Several examples of protoadditive reflectors are examined, and some counter-examples considered, which show the independence from other important types of reflections (such as semi-left-exact, admissible or Barr-exact reflections). 

In Section $3$ we study torsion theories in homological categories whose torsion-free reflector is protoadditive. We prove that an effective descent morphism is a normal extension if and only if its kernel is torsion-free (Proposition \ref{protocentral}). Next, we establish a bijection between torsion theories satisfying a normality condition $(N)$ and stable factorisation systems $(\mathbb E, \mathbb M)$  having the property that every $e \in \mathbb E$ is a normal epimorphism (Proposition \ref{inducedfactorisation}). As a consequence of this, we obtain that every effective descent morphisms $f$ admits a stable ``monotone-light'' factorisation $f=m\circ e$ into a morphism $e$ inverted by the torsion-free reflector followed by a normal extension $m$. We conclude in particular that the category of normal extensions is reflective in the category of effective descent morphisms (Theorem \ref{protofactorisation}).

We continue our study of protoadditive torsion-free reflectors in Section $4$. It turns out that the category of normal extensions is not only reflective in the category of effective descent morphisms, but also torsion-free, and that the reflector is again protoadditive (Proposition \ref{firstderivedtt}). We use this result to construct a chain of derived torsion theories in the categories of so-called higher extensions (Theorem \ref{higherderivedT1}), by adopting the axiomatic approach to higher extensions from \cite{Ev}. 

Next, in Section $5$, we study the normal extensions with respect to a Birkhoff subcategory of a semi-abelian category, in the situation where the reflector is protoadditive. Similar to the case of torsion theories, we have that an effective descent morphism is a normal extension if and only if its kernel lies in the Birkhoff subcategory, but this time the protoadditivity of the reflector is also necessary for this characterisation of normal extensions to hold whenever the normality condition $(N)$ (see page \pageref{conditionN}) is satisfied (Proposition \ref{characterisationbyextensions}). A higher dimensional version of the same result is also proved (Theorem \ref{characterisationbyextensionshigher}). 

In the last section, we generalise results from the previous sections by characterising the normal extensions and higher dimensional normal extensions with respect to a composite reflection
 \[
 \xymatrix@=30pt{
{\Ac \, } \ar@<1ex>[r]_-{^{\perp}}^-{I} & {\, \Bc \, }
\ar@<1ex>[l]^H  \ar@<1ex>[r]_-{^{\perp}}^-{J} & \Cc \ar@<1ex>[l]^G   }
 \]
where $\Ac$ is a semi-abelian category, $\Bc$ a Birkhoff subcategory of $\Ac$, and $\Cc$ an admissible (normal epi)-reflective subcategory of $\Bc$ with protoadditive reflector (Theorem \ref{highercomposite}). The admissibility condition on $\Cc$ is satisfied both in the case where $\Cc$ is a torsion-free subcategory, and where a Birkhoff subcategory of $\Bc$: either case is investigated seperately (Proposition \ref{compositetorsion} and Theorems \ref{compositecommutator} and \ref{compositeintersection}). Finally, we apply the results for the latter case in order to obtain simple descriptions of the non-abelian derived functors of $J\circ I$ via higher Hopf formulae (Corollaries \ref{compositehopf} and \ref{compositehopf2}). We conclude with some new examples in the categories of groups, compact semi-abelian algebras, and internal groupoids in a semi-abelian category.

\tableofcontents

\section{Preliminaries}

\subsection*{Torsion theories}
Torsion theories, although classically defined in abelian categories, have been studied in more general contexts by various authors (see for instance \cite{CHK}, and more recently \cite{BG,CDT,BelReit,JT}). Here we recall the definition from \cite{JT}, which is essentially Dickson's definition from \cite{D}, except that the category $\Ac$ is not asked to be abelian, but only pointed.

Note that by a \emph{pointed} category we mean, as usual, a category $\Ac$ which admits a \emph{zero-object}, i.e.~an object $0\in\Ac$ which is both initial and terminal. For any pair of objects $A, B\in\Ac$, the unique morphism $A\to B$ factorising through the zero-object, will also be denoted by $0$. If $f\colon A\to B$ is a morphism in $\Ac$, we shall write $\ker(f)\colon K[f]\to A$ for its kernel (the pullback along $f$ of the unique morphism $0\to B$) and $\coker (f)\colon B\to\Cok [f]$ for its cokernel (the pushout by $f$ of $A\to 0$), provided they exist. A \emph{short exact sequence} in $\Ac$ is given by a composable pair of morphisms $(k,f)$, as in the diagram
\begin{equation}\label{ses}
\xymatrix{
0 \ar[r] & K \ar[r]^k & A \ar[r]^f \ar[r] & B\ar[r] & 0,}
\end{equation}
such that $k=\ker (f)$ and $f=\coker (k)$. Given such a short exact sequence, we shall sometimes denote the object $B$ by $A/K$.

\begin{definition}
Let $\Ac$ be a pointed category. A pair $(\Tc,\Fc)$ of full and replete subcategories of $\Ac$ is called a \emph{torsion theory} in $\Ac$ if the following two conditions are satisfied: 
\begin{enumerate}
\item
$\hom_{\Ac}(T,F)=\{0\}$ for any $T\in\Tc$ and $F\in\Fc$;
\item
for any object $A\in\Ac$ there exists a short exact sequence
\begin{equation}\label{torsionses}
0\to T \to A \to F \to 0
\end{equation}
such that $T\in\Tc$ and $F\in\Fc$. 
\end{enumerate}
\end{definition}
$\Tc$ is called the \emph{torsion part} and $\Fc$ the \emph{torsion-free part} of the torsion theory $(\Tc,\Fc)$. A full and replete subcategory $\Fc$ of a pointed category $\Ac$ is \emph{torsion-free} if it is the torsion-free part of some torsion theory in $\Ac$. \emph{Torsion} subcategories are defined dually. The terminology comes from the classical example $( \Ab_{t.}, \Ab_{t.f.})$ of torsion theory in the variety $\Ab$ of abelian groups, where 
$\Ab_{t.f.}$ consists of all torsion-free abelian groups in the usual sense (=abelian groups satisfying, for every $n\geq 1$, the implication $nx=0 \Rightarrow x=0$), and $\Ab_{t.}$ consists of all torsion abelian groups.
There are, of course, many more examples of interest, several of which will be considered below.

A torsion-free subcategory is necessarily a reflective subcategory, while a torsion subcategory is always coreflective: the reflection and coreflection of an object $A$ are given by the short exact sequence \eqref{torsionses}, which is uniquely determined, up to isomorphism. Such reflections $\Ac\to\Fc$, for which each unit $\eta_A \colon A \rightarrow F(A)$ is a normal epimorphism (=the cokernel of some morphism) will be called \emph{(normal epi)-reflective}. Given a (normal epi)-reflective subcategory of a pointed category, there are various ways to determine whether or not it is torsion-free. For instance, this happens when the induced radical is idempotent. 

In order to explain what this means, recall that a subfunctor $T \colon \Ac \to \Ac$ of the identity functor is a \emph{radical} if, for any $A\in\Ac$,  the canonical subobject $t_A \colon T(A) \rightarrow A$ is a normal monomorphism (=the kernel of some morphism) and $T(A/T(A)) = 0$ (assuming, in particular, that every $t_A$ admits a cokernel). $T$ is \emph{idempotent} if $T\circ T=T$ or, more precisely, $t_{T(A)}\colon T(T(A))\to T(A)$ is an isomorphism, for every $A\in\Ac$. 

 Any radical $T\colon \Ac\to\Ac$ induces a (normal epi)-reflection $F\colon \Ac\to\Fc$ with units $\eta_A\colon A\to A/T(A)$. Conversely, given any (normal epi)-reflection $F\colon \Ac\to\Fc$ one obtains a radical by considering the kernels $t_A=\ker (\eta_A)\colon K[\eta_A]\to A$, provided they exist. This bijection restricts to a bijection between torsion theories and idempotent radicals, as we shall recall in Theorem \ref{torsiontheorem}.

There are strong connections between torsion theories, admissible Galois structures in the sense of \cite{J} and reflective factorisation systems in the sense of \cite{CHK}. We briefly recall some of these connections in the present section, and refer the reader to the article \cite{CJKP} and to the book \cite{BoJ} for more details.

\subsection*{Admissible Galois structures}
In this subsection, we recall some definitions from Categorical Galois Theory \cite{J1, J}. We shall restrict ourselves to the special case where the basic adjunction in the Galois structure is a reflection (as in \cite{JK4}). 

\begin{definition}
A  \emph{Galois structure} $\Gamma = ( \Ac,\Fc, F ,U,{\mathcal{E}})$ on a category $\Ac$ consists of a full replete reflective subcategory $\Fc$ of $\Ac$, with inclusion $U\colon \Fc\to\Ac$ and reflector $F\colon \Ac\to\Fc$:
\begin{equation} \label{GaloisSTR}
\xymatrix{
{\Ac}\,\, \ar@<1ex>[r]^-{F} & {\Fc},
\ar@<1ex>[l]^-{U}_-{_{\perp}}}
\end{equation}
together with a class  $\Ec$ of morphisms in $\Ac$, such that: 
\begin{enumerate}
\item $\Ac$ admits pullbacks along morphisms in $\Ec$;
\item $\Ec$ contains all isomorphisms, is stable under composition and under pullbacks;
\item $UF(\Ec) \subset \Ec$.
\end{enumerate}
\end{definition}
We shall usually drop the functor $U$ from the notations, since it is a full inclusion. 

Often, $\Ec$ is the class of \emph{all} morphisms in $\Ac$, in which case the Galois structure $\Gamma = ( \Ac,\Fc, F,{\mathcal{E}})$  is called \emph{absolute}. However, in many of the examples we consider, $\Ec$ will be a class of \emph{effective descent morphisms}, whose definition we now recall. (See \cite{JST} for a beautiful introduction to descent theory.)

For an object $B\in\Ac$, we write $(\Ac\downarrow_{\Ec} B)$ for the full subcategory of the comma category ($\Ac \downarrow B$) of objects over $B$, determined by the morphisms in $\Ec$ with codomain $B$. Similarly, we write $(\Fc\downarrow_{\Ec}F(B))$ for the full subcategory of ($\Fc \downarrow F(B)$) whose objects are in $\Ec$. For a morphism $p \colon E \rightarrow B$ in $\Ac$, we denote by 
\[
p^* \colon (\Ac\downarrow_{\Ec} B) \rightarrow (\Ac\downarrow_{\Ec} E)
\]
the ``change of base'' functor sending a morphism $f\colon A \rightarrow B$ in $\Ec$ to its pullback $p^* (f) \colon E\times_B A \rightarrow E$ along $p$.  

\begin{definition}
A morphism $p\colon E \rightarrow B\in \Ec$ is a \emph{monadic extension} when the functor $p^*$ is monadic. When $\Ec$ is the class of all morphisms, a monadic extension will be called an \emph{effective descent morphism}. 
\end{definition}

In a variety of universal algebras, an effective descent morphism is the same as a surjective homomorphism. More generaly, in an exact  \cite{Barr} category, the effective descent morphisms are precisely the regular epimorphisms. However, this need no longer be the case in an arbitrary regular \cite{Barr} category.

Now, let $B$ be an object of $\Ac$. The reflection \eqref{GaloisSTR} induces an adjunction 
\begin{equation} \label{GaloisInduced}
\xymatrix{
 {(\Ac\downarrow_{\Ec} B)}\,\, \ar@<1ex>[r]^-{F^B} & {(\Fc\downarrow_{\Ec}F(B))},
\ar@<1ex>[l]^-{U^B}_-{_{\perp}}}
\end{equation}
where $F^B$ is defined by $F^B (f)= F(f)$ for any $f \in (\Ac\downarrow_{\Ec}B)$, and $U^B (\phi)= \eta_B^* (U(\phi))$ on any $\phi \in (\Fc\downarrow_{\Ec} F(B))$. This adjunction need not, in general, be a full reflection, but those Galois structures for which this \emph{is} the case for every $B\in\Ac$, play a fundamental role:

\begin{definition}
A Galois structure $\Gamma = ( \Ac,\Fc,F ,U,\Ec)$ is \emph{admissible} when the functor $U^B\colon {(\Fc\downarrow_{\Ec} F(B))}  \rightarrow {(\Ac\downarrow_{\Ec} B)}$ is fully faithful for every $B \in \Ac$.
\end{definition}

We shall sometimes say that the reflection $F\colon \Ac\to \Fc$ is ``admissible with respect to $\Ec$'', when we mean that $\Gamma$ is admissible.
 
With respect to a given admissible Galois structure, one studies the following types of morphisms:

\begin{definition}
Let $\Gamma = ( \Ac,\Fc,F ,U,\Ec)$ be an admissible Galois structure. A morphism $f\colon A\to B$ in $\Ec$ is called
\begin{enumerate}
\item
a \emph{trivial extension} (or \emph{trivial covering}) when the canonical commutative square 
\[
\xymatrix{A \ar[r]^-{\eta_A} \ar[d]_{f} & F(A)  \ar[d]^{F(f)} \\
B \ar[r]_-{\eta_B} & F(B)}
\]
is a pullback;
\item
a \emph{central extension} (or \emph{covering}) when it is ``locally trivial'': there exists a monadic extension $p\colon E \rightarrow B$ with the property that $p^* (f)$ is a trivial extension;
\item
a \emph{normal extension} if it is a monadic extension and $f^* (f)$ is a trivial extension. 
\end{enumerate}
\end{definition}
Note that, by admissibility, $f$ is a trivial extension if and only if it lies in the (essential) image of the functor $U^B$. 

If we choose $\Ac=\Gp$ the variety of groups, $\Fc=\Ab$ the subvariety of abelian groups and $F=\ab$ the abelianisation functor, then $(\Gp,\Ab,\ab,\Ec)$ is an admissible Galois structure for $\Ec$ the class of surjective homomorphisms \cite{J}. Here, the trivial extensions are precisely the surjective homomorphisms $f\colon A\to B$ whose restriction $[A,A]\to [B,B]$ to the commutator subgroups is an isomorphism. Central and normal extensions coincide and are precisely the central extensions in the usual sense: surjective homomorphisms $f\colon A\to B$ whose kernel $K[f]$ lies in the centre of $A$. (See \cite{BoJ}, for instance, for more details.)

Note that the admissibility can also be expressed as an exactness property of the reflector: $\Gamma = ( \Ac,\Fc,F ,U,\Ec)$ is admissible if and only if the reflector $F\colon \Ac\to \Fc$ preserves pullbacks of the form
\begin{equation}\label{sle}\vcenter{
\xymatrix{
P \ar[d] \ar[r] \ar@{}[rd]|<<{\pullback} & X \ar[d]^{x}\\
A \ar[r]_-{\eta_A} & F(A)}}
\end{equation}
where  $X \in \Fc$, $x \colon X\rightarrow F(A) $ lies in $\Ec$  and $\eta_A$ is the reflection unit. In particular, in the absolute case (where $\Ec$ is the class of \emph{all} morphisms) this means that an admissible Galois structure is the same as a \emph{semi-left-exact} reflection in the sense of \cite{CHK}: a reflection $F\colon \Ac\to\Fc$ of a category $\Ac$ into a full and replete subcategory $\Fc$ of $\Ac$ preserving all pullbacks \eqref{sle} where $X \in \Fc$.

Semi-left-exact reflections were introduced in the study of reflective factorisation systems. We briefly recall some notions from \cite{CHK}.

\subsection*{Reflective factorisation systems}

For morphisms $e$ and $m$ in a category $\Ac$ we write $e\downarrow m$ if for every pair of morphisms $(a,b)$ such that $b\circ e=m\circ a$,  there exists a unique morphism $d$ such that $d\circ e=a$ and $m\circ d=b$:
\[
\xymatrix{
A \ar[r]^e \ar[d]_a & B \ar@{.>}[ld]|{d}   \ar[d]^b\\
C \ar[r]_m & D.}
\]
For classes $\mathbb{E}$ and $\mathbb{M}$ of morphisms in $\Ac$ we put
\[
\mathbb{E}^{\downarrow}=\{m | e\downarrow m \  \textrm{for all} \ e\in \mathbb{E}\}, \  \ \  \mathbb{M}^{\uparrow}=\{e | e\downarrow m \  \textrm{for all} \ m\in \mathbb{M}\}.
\]  
By a \emph{prefactorisation system} on a category $\Ac$ we mean a pair $(\mathbb{E},\mathbb{M})$ of classes of morphisms in $\Ac$ such that $\mathbb{E}=\mathbb{M}^{\uparrow}$ and $\mathbb{M}=\mathbb{E}^{\downarrow}$. A \emph{factorisation system} is a prefactorisation system $(\mathbb{E},\mathbb{M})$ such that for every morphism $f$ in $\Ac$ there exist morphisms $e\in \mathbb{E}$ and $m\in \mathbb{M}$ such that $f=m\circ e$.

Any full replete reflective subcategory $\Fc$ of a category $\Ac$ determines a prefactorisation system $(\mathbb E, \mathbb M)$ on $\Ac$, where $\Ef$ is the class of morphisms inverted by the reflector $F \colon \Ac \rightarrow \Fc$ and $\Mf=\mathbb{E}^{\downarrow}$. Furthermore, when the reflector $F \colon \Ac \rightarrow \Fc$ is semi-left-exact and $\Ac$ admits pullbacks along every unit $\eta_A\colon A\to F(A)$ ($A\in\Ac$), the prefactorisation system  $(\Ef, \Mf)$ is a factorisation system and $\mathbb M$ consists exactly of the trivial extensions with respect to the corresponding absolute Galois structure. When $\Ac$ admits arbitrary pullbacks, $F \colon \Ac \rightarrow \Fc$ is semi-left-exact if and only if $F$ preserves pullbacks along morphisms in $\Mf$. (See \cite{CHK,CJKP} for more details.)

When the reflector $F\colon \Ac\to \Fc$ preserves pullbacks of the form \eqref{sle} where, however, one no longer assumes that $X$ belongs to $\Fc$, one says that $F$ has \emph{stable units} \cite{CHK}. This property is equivalent to the units $\eta_A\colon A\to F(A)$ ($A\in\Ac$) being \emph{stably in $\Ef$}: the pullback $f^*(\eta_A)$ along any morphism $f\colon B\to F(A)$ lies in $\Ef$. Note that even if $F$ has stable units, the reflective factorisation systems $(\mathbb{E},\mathbb{M})$ need not be \emph{stable} (i.e. the class $\mathbb{E}$ is pullback stable), in general. In fact, it was shown in \cite{CHK} that this only happens when $F$ is a localisation: $F$ preserves arbitrary finite limits. Restricting $\mathbb{E}$ to the class $\mathbb{E}'$ of morphisms $e\in \mathbb{E}$ that are stably in $\mathbb{E}$ and enlarging $\mathbb{M}$ to the class $\mathbb{M}^*$ of central extensions, sometimes (but certainly not always) yields a new factorisation system $(\mathbb{E}',\mathbb{M}^*)$ which is stable by definition. We shall consider examples of where this is ``partially'' true in Section \ref{coveringmorphisms}.

\vspace{3mm} 

We conclude this section by listing several characterisations of torsion-free subcategories in terms of some of the notions recalled above. Most of these are known, but the equivalence between the semi-left-exactness and the stability of units (under the given conditions) is new, as far as we know. 

\begin{theorem}\label{torsiontheorem}
For a full replete subcategory $\Fc$ of a finitely complete pointed category $\Ac$ with pullback-stable normal epimorphisms, the following conditions are equivalent:
\begin{enumerate}
\item
$\Fc$ is a torsion-free subcategory of $\Ac$;
\item
$\Fc$ is a (normal epi)-reflective subcategory of $\Ac$ and the induced radical is idempotent;
\item
$\Fc$ is a (normal epi)-reflective subcategory of $\Ac$ and the reflector $F\colon \Ac\to\Fc$ has stable units;
\item
$\Fc$ is a (normal epi)-reflective subcategory of $\Ac$ and the reflector $F\colon \Ac\to\Fc$ is semi-left-exact (=admissible);
\end{enumerate}
\end{theorem}
\begin{proof}
For the equivalences $(1) \Leftrightarrow (2) \Leftrightarrow (4)$ see \cite{BG,JT}.  $(3)\Rightarrow (4)$ is true by definition.

We prove that $(2)$ implies $(3)$. Consider, for this, a commutative diagram
$$\xymatrix@=35pt{&P \ar[r]^{p_2} \ar@{}[rd]|<<<{\pullback}\ar[d]_{p_1} & X \ar[d]^{x}  \\
T(A) \ar[r]_-{t_A} \ar[ur]^{t} & A \ar[r]_{\eta_A } & F(A)
}
$$
where the square is a pullback and $t= \ker (p_2)$. Observe that $p_2$ is necessarily the cokernel of $t$. Moreover, since $T$ is idempotent, we have that $F(T(A))=0$. Hence, by applying the left adjoint $F$, one gets the commutative diagram
$$\xymatrix@=35pt{&F(P) \ar[r]^{F(p_2)} \ar[d]_(.6){F(p_1)} & F(X) \ar[d]^{F(x)}  \\
0 \ar[r]_-{t_A} \ar[ur]^{F(t)} & F(A) \ar@{=}[r]_{} & F(A)
}
$$
where $F(p_2)= \coker(F(t))$, since obviously $F$ preserves cokernels, so that the square is a pullback, as desired.
\end{proof}

The above theorem asserts, in particular, that any torsion-free reflection gives rise to an admissible Galois structure, and suggests to study the central extensions with respect to a torsion theory. We shall do this in sections \ref{coveringmorphisms} and \ref{sectionderived}  (see also \cite{CJKP,GR,GJ}), and we shall be particularly interested in torsion theories with a \emph{protoadditive} reflector.

\section{Protoadditive functors}\label{protoadditivesection}
Let $\Ac$ be a pointed category with pullbacks along split epimorphisms. By a \emph{split short exact sequence} in $\Ac$ we mean a triple $(k,f,s)$ of morphisms in $\Ac$, as in the diagram
\begin{equation}\label{sses}
\xymatrix{0 \ar[r]& K \ar[r]^k & A \ar@<-.8 ex>[r]_f & B \ar@<-.8ex>[l]_s \ar[r] &0, }
\end{equation}
such that $k=\ker(f)$ and $f\circ s=1_B$ (i.e. $f$ is a split epimorphism with splitting~$s$). $\Ac$ is a \emph{protomodular} category in the sense of Bourn \cite{Bourn0} precisely when the split short five lemma holds true in $\Ac$: given a morphism 
\[
\xymatrix{
0 \ar[r] & K  \ar[r] \ar[d]_-{\kappa} & A \ar[d]_{\alpha} \ar@<-.8 ex>[r]_f & B \ar@<-.8ex>[l]_{s} \ar[d]^{\beta} \ar[r] & 0\\
0 \ar[r] &  K'  \ar[r]  & A'  \ar@<-.8 ex>[r]_{f'} & B' \ar@<-.8ex>[l]_{s' } \ar[r] & 0}
\]
of split short exact sequences, if both $\kappa$ and $\beta$ are isomorphisms, then so is $\alpha$. Note that the protomodularity can be equivalently expressed as the property that the right-hand square $\beta \circ f = f' \circ \alpha$ is a pullback if and only if $\kappa$ is an isomorphism, for any morphism of split short exact sequences as above.

The prototypical example of a pointed protomodular category is the variety of groups. In fact, \emph{any} pointed variety whose theory contains the group operations and identities (such as the varieties of rings and of Lie algebras) is protomodular, and more examples will be considered in what follows.

If a pointed protomodular category $\Ac$ is, moreover, finitely complete, then any regular epimorphism (=the coequaliser of some pair of morphisms), and in particular any split epimorphism, is normal \cite{Bourn0}. Thus, in particular, any split short exact sequence is a short exact sequence. Of course, if $\Ac$ is an additive category, then any split short exact sequence in $\Ac$ is, up to isomorphism, of the form
\[
\xymatrix{0 \ar[r]& K \ar[r]^-{i_K} & K\oplus B \ar@<-.8 ex>[r]_-{\pi_B} & B  \ar@<-.8ex>[l]_-{i_B} \ar[r] &0 }
\]
where $K\oplus B$ is the biproduct of $K$ and $B$, $i_K$ and $i_B$ are the canonical injections and $\pi_B$ the canonical projection, and the split short five lemma becomes a triviality. Hence, any additive category is pointed protomodular. Moreover, a functor between additive categories is additive (that is, it preserves binary biproducts) if and only if it preserves split short exact sequences. We claim that the latter property is still meaningful in a non-additive context. This brings us to the central notion of this article:

\begin{definition} \cite{EG}
A functor $F\colon \Ac\to\Fc$ between pointed protomodular categories $\Ac$ and $\Fc$ is \emph{protoadditive} if it preserves split short exact sequences:  for any split short exact sequence \eqref{sses} in $\Ac$, the image 
$$\xymatrix{0 \ar[r]& F(K) \ar[r]^{F(k)} & F(A) \ar@<-.8 ex>[r]_{F(f)} & F(B)  \ar@<-.8ex>[l]_{F(s)} \ar[r] &0 }$$
by $F$ is a split short exact sequence in $\Fc$.
\end{definition}

Note that a protoadditive functor necessarily preserves the zero object. Moreover, the preservation of split short exact sequences implies at once the preservation of arbitrary pullbacks along split epimorphisms:

\begin{proposition}\label{protoadditive-pullback}
A zero-preserving functor $F\colon \Ac\to\Fc$ between pointed protomodular categories $\Ac$ and $\Fc$ is protoadditive if and only if it preserves pullbacks along split epimorphisms.
\end{proposition}
\begin{proof}
Clearly, $F$ is protoadditive as soon as it preserves pullbacks along split epimorphisms as well as the zero object. 

Now assume that $F$ is protoadditive. Given a pullback 
\[
\xymatrix{A \times_B E \ar@{}[rd]|<<{\pullback} \ar@<-.8ex>[r]_-{\pi_E} \ar[d]_{\pi_A} & E \ar[d]^p  \ar@<-.8ex>[l]  \\
A  \ar@<-.8 ex>[r]_f & B \ar@<-.8ex>[l] }
\]
along a split epimorphism $f$, the restriction $\overline{\pi}_A \colon K[\pi_E] \rightarrow K[f]$ of $\pi_A$ is an isomorphism. By applying the functor $F$,  one gets a morphism
$$
\xymatrix@=40pt{
0 \ar[r] & F(K[\pi_E])   \ar[d]_{F(\overline{\pi}_A)} \ar[r]^-{F(\ker (\pi_E))} & F(E \times_B A)   \ar@<-.8ex>[r]_-{F(\pi_E)}  \ar[d]_-{F(\pi_A)} &F( E ) \ar@<-.8ex>[l] \ar[d]^{F(p)} \ar[r] & 0 \\
0 \ar[r] & F(K[f]) \ar[r]_-{F(\ker (f))}  & F(A) \ar@<-.8ex>[r]_{F(f)}  & F(B) \ar@<-.8ex>[l]  \ar[r] & 0}
$$
of split short exact sequences in $\Fc$, where the left hand vertical arrow $F(\overline{\pi}_A) $ is an isomorphism. The right hand square is then a pullback by protomodularity.
\end{proof}

A pointed protomodular category is called \emph{homological} \cite{BB} if it is also regular \cite{Barr}: finitely complete with stable regular epi-mono factorisations. A \emph{semi-abelian} category \cite{JMT} is a pointed protomodular category $\Ac$ with binary coproducts which is, moreover, exact \cite{Barr}: regular and every equivalence relation in $\Ac$ is effective (=the kernel pair of some morphism).

Since any variety of universal algebras is exact, any pointed protomodular variety is semi-abelian. The category of topological groups provides an example of a homological category which is not semi-abelian, as opposed to its full subcategory of compact Hausdorff groups, which  \emph{is} semi-abelian. In fact, in the latter two examples, we could replace the theory of groups with any semi-abelian algebraic theory, i.e.~a Lawvere theory $\mathbb{T}$ such that the category $\Set^{\mathbb{T}}$ of $\mathbb{T}$-models in the category $\Set$ of sets is a semi-abelian category. It was shown in \cite{BouJ} that a theory is semi-abelian precisely when it contains a unique constant, written $0$, binary terms $\alpha_i (x,y)$ (for $i\in \{1, \dots, n\}$ and a natural number $n\geq 1$) and an $(n+1)$-ary term $\beta$ subject to the identities $$\alpha_i(x,x)=0 \quad {\rm and}\quad \beta(\alpha_1(x,y), \dots , \alpha_n (x,y),y)=x.$$  
In \cite{BC} it was proved, for any semi-abelian theory $\mathbb{T}$, that the category $\Top^{\mathbb{T}}$ of topological $\mathbb{T}$-algebras (=$\mathbb{T}$-models in the category $\Top$ of topological spaces) is homological, and that the full subcategory $\TCom$  of compact Hausdorff topological $\mathbb{T}$-algebras is semi-abelian (in fact, a semi-abelian category monadic over $\Set$).

Diagram lemmas such as the (short) five lemma, the $3\times 3$ lemma and the snake lemma, which are well known to hold in the abelian context, are also valid in any homological category \cite{B2, BB}. The $3\times 3$ lemma  immediately gives us the following:

\begin{proposition}\label{reflector=radical}
Let $\Ac$ be a homological category and $F\colon \Ac\to \Fc$ the reflector into a full replete (normal epi)-reflective subcategory $\Fc$ of $\Ac$. Then $F$ is protoadditive if and only if the corresponding radical $T\colon \Ac\to\Ac$ is protoadditive.
\end{proposition}

For a homological category $\Ac$, and a full replete (normal epi)-reflective subcategory $\Fc$ of $\Ac$, the protoadditivity of the reflector $F\colon \Ac\to\Fc$ can also be formulated as the preservation of a certain class of monomorphisms. 

\begin{definition}\cite{BJK}
A \emph{protosplit monomorphism} in a pointed protomodular category $\Ac$ is a normal monomorphism $k \colon K \rightarrow A$ that is the kernel of a split epimorphism. \end{definition}
In other words, protosplit monomorphisms are the monomorphisms $k$ appearing in split short exact sequences of the form \eqref{sses}.

\begin{proposition}\label{caracterisationproto}
Let $\Ac$ be a homological category and $F \colon \Ac \rightarrow \Fc$ the reflector into a full replete (normal epi)-reflective subcategory $\Fc$ of $\Ac$. Then the following conditions are equivalent:
\begin{enumerate}
\item $F \colon \Ac \rightarrow \Fc$ is a protoadditive functor;
\item $F \colon \Ac \rightarrow \Fc$ sends protosplit monomorphisms to normal monomorphisms;
\item $F \colon \Ac \rightarrow \Fc$ sends protosplit monomorphisms to monomorphisms.
\end{enumerate}
\end{proposition}
\begin{proof}
The implications $(1) \Rightarrow (2) \Rightarrow (3)$ are trivial, and we have that $(2)$ implies $(1)$ since, for a split short exact sequence \eqref{sses}, if $F(k)$ is a normal monomorphism, it is necessarily the kernel of its cokernel, and the latter is $F(f)$, since $F$ preserves colimits. Note that for these implications the regularity of $\Ac$ is irrelevant, as is the assumption that the reflection units $\eta_A\colon A\to F(A)$ are normal epimorphisms.

Let us then prove the implication $(3) \Rightarrow (1)$. For this, we consider a split short exact sequence \eqref{sses}. It induces the diagram
\[
\xymatrix{0 \ar[r] &K \ar[r]^k \ar[d]_{\eta_K} & A \ar[d]^{\eta_A} \ar@<-.8 ex> [r]_-f & B \ar[d]^{\eta_{B}} \ar@<-.8ex>[l]_-s \ar[r] \ar[d]
&0  \\
 &F(K) \ar[r]_{F(k)} & F(A) \ar@<-.8 ex> [r]_-{F(f)} & F(B) \ar@<-.8ex>[l]_-{F(s)} \ar[r] 
&0 
}
\]
in $\Ac$ where the vertical morphisms are the reflection units. By assumption, $F(k)$ is a monomorphism. Moreover, since any homological category is regular Mal'tsev (see \cite{Bourn1996}), the right hand square of regular epimorphisms is a \emph{regular pushout} or \emph{double extension} (by Proposition $3.2$ in \cite{B3}), which means that also the induced morphism $(\eta_A,f)\colon A\to F(A)\times_{F(B)}B$ to the pullback of $F(f)$ along $\eta_B$ is a regular epimorphism. Hence, by regularity of $\Ac$, so is the restriction of $\eta_A$ to the kernels $K\to K[F(f)]$, since this is a pullback of $(\eta_A,f)$. It follows that also the induced morphism $F(K)\to K[F(f)]$ is a regular epimorphism. As it is also a monomorphism---since $F(k)$ is a monomorphism---it is then an isomorphism. Hence $F(k)$ is the kernel of $F(f)$ in the category $\Ac$. Since the inclusion $\Fc \rightarrow \Ac$ reflects limits, $F(k)$ is the kernel of $F(f)$ in $\Fc$, and we can conclude that the functor $F \colon \Ac \rightarrow \Fc$ is indeed protoadditive.
\end{proof}

Let us now consider torsion theories $(\Tc,\Fc)$ whose reflector $F\colon \Ac\to\Fc$ is protoadditive.  First of all, we show how the protoadditivity of $F$ can be detected from the torsion subcategory $\Tc$, when $\Ac$ is homological.

A subcategory $\Tc$ of a category $\Ac$ is called \emph{$\MC$-hereditary}, for $\MC$ a class of monomorphisms in $\Ac$, if for any $m\colon A\to B$ in $\MC$, $B\in\Tc$ implies that $A\in\Tc$. When $\MC$ is the class of all monomorphisms,  $\Tc$ is simply  called \emph{hereditary}. A torsion theory $(\Tc,\Fc)$ is ($\MC$-)hereditary if its torsion part $\Tc$ is so.


\begin{theorem}\label{protoM}
For a torsion theory $(\Tc,\Fc)$ in a homological category $\Ac$, the following conditions are equivalent:
\begin{enumerate}
\item
the torsion subcategory $\Tc$ is $\MC$-hereditary, for $\MC$ the class of protosplit monomorphisms;
\item
the reflector $F\colon \Ac\to\Fc$ is protoadditive.
\end{enumerate}
\end{theorem}
\begin{proof}
The implication $(1)\Rightarrow (2)$ follows from Proposition \ref{reflector=radical}, since for any $\MC$-hereditary torsion theory $(\Tc,\Fc)$ in $\Ac$, the image 
$$
\xymatrix{0 \ar[r]& T(K) \ar[r]^{T(k)} & T(A) \ar@<-.8 ex> [r]_-{T(f)} & T(B) \ar@<-.8ex>[l]_-{T(s)} \ar[r] &0.}
$$
by the corresponding radical $T\colon \Ac\to\Ac$ of any split short exact sequence \eqref{sses} in $\Ac$ is again a split short exact sequence. Indeed, the coreflector $T \colon \Ac \rightarrow \mathcal T$ certainly preserves kernels (as any right adjoint), and the fact that $\mathcal T$ is closed in $\Ac$ under protosplit monomorphisms implies that $T(k) \colon T(K) \rightarrow T(A)$ is still the kernel of $T(f)$ in the category $\Ac$.

For the implication $(2)\Rightarrow (1)$, assume that $F \colon \Ac \rightarrow \Fc$ is protoadditive. Then, if $k \colon K \rightarrow A$ is a protosplit monomorphism such that $A$ lies in the corresponding torsion subcategory $\mathcal T$, then $K$ lies in $\Tc$ as well: indeed, by applying the functor $F$ to $k$ we obtain the morphism $F(k) \colon F(K) \rightarrow F(A)$ which is a monomorphism since $F$ is protoadditive, so that $F(A)=0$ implies that  $F(K)=0$, hence $K\in\mathcal T$.
\end{proof}

Recall that a subcategory $\Fc$ of a pointed category $\Ac$ is \emph{closed under extensions} if for any short exact sequence \eqref{ses} in $\Ac$, the object $A\in\Fc$ as soon as both $K\in\Fc$ and $B\in\Fc$. It is well known that a full replete (normal epi)-reflective subcategory $\Fc$ of an abelian category $\Ac$ is torsion-free if and only if it is closed under extensions. While the ``only if" part is still valid in arbitrary pointed categories $\Ac$, this is no longer the case for the ``if" part (see \cite{JT}).     However, it turns out that both implications hold when the reflector $F\colon \Ac\to\Fc$ is protoadditive and $\Ac$ is either a semi-abelian category or a category of topological semi-abelian algebras, as we shall see below. Moreover, a full replete reflective subcategory $\Fc$ of a pointed protomodular category $\Ac$ with protoadditive reflector $F\colon \Ac\to\Fc$ is always  \emph{closed under split extensions}, even if it is not torsion-free. This means that  for any split short exact sequence \eqref{sses} in $\Ac$, the object $A$ lies in $\Fc$ as soon as $K\in\Fc$ and $B\in\Fc$:

\begin{proposition}\label{splitext}
Any full replete reflective subcategory $\Fc$ of a pointed protomodular category $\Ac$ with protoadditive reflector $F \colon \Ac \rightarrow \Fc$ is closed under split extensions in $\Ac$.
\end{proposition}
\begin{proof}
If \eqref{sses} is a split extension in $\Ac$ with $K, B \in \Fc$, then the split short five lemma applied to the commutative diagram 
\[
\xymatrix@=35pt{0   \ar[r]& K \ar@{=}[d]_{\eta_K}   \ar[r]^{k} & A \ar[d]_{\eta_A}  \ar@<-.8 ex> [r]_-{f} & B \ar@{=}[d]^{\eta_B} \ar@<-.8ex>[l]_-{s} \ar[r] &0 \\
0 \ar[r]& F(K) \ar[r]^{F(k)} & F(A) \ar@<-.8 ex> [r]_-{F(f)} & F(B) \ar@<-.8ex>[l]_-{F(s)} \ar[r] &0}
\]
of exact sequences in $\mathcal A$ shows that the reflection unit $\eta_A$ is an isomorphism. Hence, $A$ belongs to $\Fc$. 
\end{proof}

\begin{remark}
Closedness under split extensions is not a sufficient condition for $F\colon \Ac\to \Fc$ to be protoadditive. For instance, the quasivariety $\Gp_{t.f.}$ of torsion-free groups (=groups satisfying, for every $n\geq 1$, the implication $x^n=1\Rightarrow x=1$) is closed under (split) extensions in the variety $\Gp$ of groups, since it is a torsion-free subcategory of $\Gp$, but the reflector $\Gp\to\Gp_{t.f.}$ is not protoadditive  (see Example \ref{counterproto}.\ref{extfree}).
\end{remark}


Now let $\Ac$ be a pointed category and $\Fc$ a full replete (normal epi)-reflective subcategory of $\Ac$. As remarked above, closedness under extensions is necessary but not sufficient for $\Fc$ to be torsion-free. However, by the Corollary in \cite{JT}, when $\Ac$ is homological, the two conditions are equivalent as soon as the composite $t_A\circ t_{T(A)}\colon T(T(A))\to A$ is a normal monomorphism, for any $A\in\Ac$ (here, as before, $t_A\colon T(A)\to A$ denotes the coreflection counit). It turns out that the latter property is always satisfied if $\Ac$ is semi-abelian and $F\colon \Ac\to\Fc$ is protoadditive:

\begin{lemma} \label{compositeisnormal } 
Let  $\Fc$ be a full replete (normal epi)-reflective subcategory  of a semi-abelian category $\Ac$ with protoadditive reflector $F \colon \Ac \rightarrow \Fc$. Then, for any normal monomorphism $k \colon K \rightarrow A$, the monomorphism $k \circ t_K \colon T(K) \rightarrow A$ is normal.
\end{lemma}
\begin{proof}
Let $(R, \pi_1, \pi_2)$ be the equivalence relation on $A$ corresponding to the normal subobject $k \colon K \rightarrow A$, so that $k = \pi_2 \circ \ker (\pi_1)$. One then forms the diagram
$$\xymatrix@=35pt{T(K )  \ar[d]_{t_K} \ar[r]^{T ( \ker (\pi_1))} & T(R) \ar[d]_{t_R}  \ar@<.8 ex>[r]^{T(\pi_1)} \ar@<-.8 ex>[r]_{T (\pi_2)} & T(A)\ar[d]^{t_A} \\
K \ar[r]_{\ker (\pi_1)} & R \ar@<.8 ex>[r]^{\pi_1} \ar@<-.8 ex>[r]_{\pi_2} & A}$$
which is obtained by applying the radical $T$ corresponding to the reflector $F$ to the lower row. One observes that the left-hand square is a pullback, due to the fact that $T(\ker (\pi_1))$ is the kernel of $T(\pi_1)$ (by Proposition \ref{reflector=radical})
and $t_A$ is a monomorphism. 
It follows that the composite $\ker (\pi_1) \circ t_K$ is a normal monomorphism, as an intersection of normal monomorphisms. Finally, the arrow  $\pi_2 \circ \ker (\pi_1) \circ t_K = k \circ t_K$ is a normal monomorphism, as it is the regular image along the regular epimorphism $\pi_2$ of the normal monomorphism $\ker (\pi_1) \circ t_K$.
\end{proof}

Hence, the Corollary in \cite{JT} gives us:

\begin{proposition}\label{torsion=closed}
Let $\Ac$ be a semi-abelian category and $F \colon \Ac \rightarrow \Fc$ a protoadditive reflector into a full replete (normal epi)-reflective subcategory $\Fc$ of $\Ac$. Then $\Fc$ is a torsion-free subcategory of $\Ac$ if and only if $\Fc$ is closed in $\Ac$ under extensions.
\end{proposition}

\begin{remark}\label{idealtopological}
Lemma \ref{compositeisnormal }, and, consequently, also Proposition \ref{torsion=closed} remain valid if $\Ac=\Top^{\mathbb{T}}$ is a category of topological semi-abelian algebras. Indeed, to adapt the proof of Lemma \ref{compositeisnormal } to this situation, it suffices to verify that the monomorphism $\pi_2\circ \ker (\pi_1) \circ t_K$ is normal. For this, first notice that the underlying morphism of semi-abelian algebras is normal. To check that it is also normal in the category of \emph{topological} semi-abelian algebras, we observe that $T(K)$ carries the induced topology for the inclusion $\pi_2\circ \ker (\pi_1) \circ t_K$ into $A$. This follows from the fact that $0\times T(K)$ has the topology induced by $R$, while $R$ has the topology induced by $A\times A$.   
\end{remark}
 
Before considering some examples, we investigate the influence of a protoadditive reflector on the associated (pre)factorisation system. As recalled above, a full replete reflection $F\colon \Ac\to \Fc$ is a localisation if and only if the class $\Ef$ of morphisms inverted by $F$ is stable under pullback. In the case of a protoadditive $F$, we still have that $\Ef$ is stable under pullback along split epimorphisms, and also the converse is true if $F$ is semi-left-exact:

\begin{theorem}
Let $\Ac$ be a finitely complete pointed protomodular category, $\Fc$ a full replete reflective subcategory of $\Ac$ and $(\Ef,\Mf)$ the induced (pre)factorisation system. Then $(1)$ implies $(2)$:
\begin{enumerate}
\item $F : \Ac \rightarrow \Fc$ is protoadditive;
\item the class $\Ef$ is stable under pullback along split epimorphisms.
\end{enumerate}
If, moreover, the reflector $F\colon \Ac\to \Fc$ is semi-left-exact, then the two conditions are equivalent.
\end{theorem}
\begin{proof}
First of all note that $\Fc$ is pointed and protomodular as a full reflective subcategory of $\Ac$. Since a protoadditive functor between pointed protomodular categories preserves pullbacks along split epimorphisms (by Proposition \ref{protoadditive-pullback}), $(1)$ implies $(2)$. 

Conversely, assume that $\Ef$ is stable under pullback along split epimorphisms and $F$ is semi-left-exact. Consider morphisms $f\colon A\to B$ and $p\colon E\to B$ in $\Ac$, with $f$ a split epimorphism. Let $e\in \Ef$ and $m\in \Mf$ be morphisms such that $p=m\circ e$. Then in the diagram
\[
\xymatrix{
E\times_BA \ar[r] \ar@<-.8 ex>[d] & I\times_BA \ar[r] \ar@<-.8 ex>[d] & A \ar@<-.8 ex>[d]_f\\
E \ar@<-.8 ex>[u] \ar[r]_e & I \ar@<-.8 ex>[u] \ar[r]_m & B \ar@<-.8 ex>[u]}
\]
the left hand pullback is preserved by $F$ by assumption, and the right hand pullback because $F$ is semi-left-exact (which implies that $F$ preserves pullbacks along morphisms in $\Mf$). Thus $F$ is protoadditive.
\end{proof}

\begin{remark}
Notice that we could have taken the object $E$ in the above proof to be zero. Hence, one could replace the condition $(2)$ with the apparently weaker condition: $(2')$ if $0 \rightarrow B$ is in $\Ef$, any kernel of a split epimorphism with codomain $B$ is in $\Ef$.
\end{remark} 
 
 \begin{examples}\label{exproto}
\begin{enumerate}

\item
Any reflector into a full reflective subcategory of an additive category is additive, hence protoadditive.

\item\label{rings}
Let $\mathsf{CRng}$ be the semi-abelian variety of commutative but not necessarily unitary rings. Write $\mathsf{RedCRng}$ for the quasivariety of reduced commutative rings (namely those ones satisfying, for every $n\geq 1$, the implication $x^n=0 \Rightarrow x=0$) and 
$\mathsf{NilCRng}$ for the full subcategory of $\mathsf{CRng}$ consisting of nilpotent commutative rings. Then ($\mathsf{NilCRng},\mathsf{RedCRng}$) is a hereditary torsion-theory in $\mathsf{CRng}$, so that, by Theorem \ref{protoM}, the reflector $F$
\[
\xymatrix{
{ \mathsf{CRng} }\,\, \ar@<1ex>[r]^-{F} & {\mathsf{RedCRng} }
\ar@<1ex>[l]^-{U}_-{_{\perp}}}
\]
is protoadditive.

\item\label{exgroupoids}
Let $\Ac$ be an arbitrary semi-abelian category and $\Gpd(\Ac)$ the category of (internal) groupoids in $\Ac$, which is again semi-abelian.  Recall (e.g. from \cite{ML}) that an (internal) groupoid $A=(A_1,A_0,m, d,c,i)$ in $\Ac$ is a diagram in $\Ac$ of the form
\[
\xymatrix{
A_1\times_{A_0}A_1 \ar[r]^-{m} & A_1 \ar@<1.8 ex>[rr]^{d} \ar@<-1.8 ex>[rr]_{c} && A_0, \ar@<0.7 ex>[ll]_{i}}
\]
where $A_0$ represents the ``object of objects'', $A_1$ the ``object of arrows'', $A_1\times_{A_0}A_1$ the ``object of composable arrows'', $d$ the ``domain'', $c$ the ``codomain'', $i$ the ``identity'', and $m$ the ``composition''. Of course, these morphisms have to satisfy the usual commutativity conditions expressing, internally, the fact that $A$ is a groupoid. 

There is an adjunction 
\[
\xymatrix{
{\Gpd(\Ac) }\,\, \ar@<1ex>[r]^-{\pi_0} & {\Ac}
\ar@<1ex>[l]^-{D}_-{_{\perp}}}
\]
where $D$ is the functor associating, with any object $A_0\in\Ac$, the discrete equivalence relation on $A_0$, 
and $\pi_0$ is the connected component functor. This functor $\pi_0$ sends a groupoid $A$ as above to the object $\pi_0 (A)$ in $\Ac$ given by the coequalizer of $d$ and $c$ or, equivalently, by the quotient $A_0/\Gamma_0(A)$, where $\Gamma_0(A)$ is the connected component of $0$ in $A$. It was proved in \cite{EG} that the functor $\pi_0$ is protoadditive. 

\item\label{exdisc} 
Let $\mathbb T$ be a semi-abelian algebraic theory. As mentioned above, $\mathbb T$ contains a unique constant $0$, binary terms $\alpha_i (x,y)$ (for $i\in \{1, \dots, n\}$ and some natural number $n\geq 1$) and an $(n+1)$-ary term $\beta$ subject to the identities 
\[
\alpha_i(x,x)=0 \quad {\rm and}\quad \beta(\alpha_1(x,y), \dots , \alpha_n (x,y),y)=x.
\]
Consider the semi-abelian category $\TCom$ of compact Hausdorff topolo-~gical $\mathbb{T}$-algebras (\emph{compact $\mathbb{T}$-algebras} for short) and $\TPro$ its full subcategory of compact and totally disconnected $\mathbb{T}$-algebras. It was shown in \cite{BC} that $\TPro$ is a (normal epi)-reflective subcategory of $\TCom$, where the reflector $I$
\[
\xymatrix{
{\TCom }\,\, \ar@<1ex>[r]^-{I} & {\TPro}
\ar@<1ex>[l]^-{}_-{_{\perp}}}
\]
sends a compact algebra $A$ to the quotient $A/\Gamma_0(A)$ of $A$ by the connected component $\Gamma_0(A)$ of $0$ in $A$. From  \cite{BG}  we know that $\TPro$ is, moreover, a torsion-free subcategory of $\TCom$ with corresponding torsion subcategory the category $\TConn$ of connected compact $\mathbb{T}$-algebras. 

We claim that $I$ is protoadditive. By Theorem \ref{protoM}, it suffices to prove that $\TConn$ is $\MC$-hereditary, for $\MC$ the class of protosplit monomorphisms. For this purpose we consider a split short exact sequence
\[
\xymatrix{0 \ar[r]& K \ar[r]^k & A \ar@<-.8 ex> [r]_f & B \ar@<-.8ex>[l]_s \ar[r] &0} 
\]
in $\TCom$ and suppose that $A$ is connected. Notice that the binary term
\[
\sigma(x,y)= \beta(\alpha_1(x,y), \dots , \alpha_n (x,y),0)
\]
is a \emph{subtraction}  \cite{U}, i.e.~ we have that $\sigma(x,x)=0$ and $\sigma(x,0)=x$. It follows that sending an element $a\in A$ to the element $\sigma(a,s(f(a)))$ defines a continuous map $g\colon A\to K$ such that $g\circ k=1_{K}$. In particular, $g$ is surjective, and $K$ is connected as a continuous image of the connected space $A$. 

\item 
Now consider the category $\mathsf{Top}^{\mathbb T}$ of topological $\mathbb{T}$-algebras and its full subcategory $\mathsf{Haus}^{\mathbb T}$ of Hausdorff $\mathbb{T}$-algebras, still for a semi-abelian theory $\mathbb{T}$. It was shown in \cite{BC} that $\mathsf{Haus}^{\mathbb T}$ is a (normal epi)-reflective subcategory of $\mathsf{Top}^{\mathbb T}$ where the reflector $I$ 
\[
\xymatrix{
{\mathsf{Top}^{\mathbb T} }\,\, \ar@<1ex>[r]^-{I} & {\mathsf{Haus}^{\mathbb T}}
\ar@<1ex>[l]_-{_{\perp}}}
\]
sends a topological semi-abelian algebra $A$ to the quotient $A /\overline{\{0\} }$ of $A$ by the closure $\overline{\{0\}}$ in $A$ of the trivial subalgebra $\{ 0\}$. From  \cite{BG}  we know that $\mathsf{Haus}^{\mathbb T}$ is, moreover, a torsion-free subcategory of $\mathsf{Top}^{\mathbb T}$, with corresponding torsion subcategory the category $\mathsf{Ind}^{\mathbb{T}}$  of indiscrete $\mathbb{T}$-algebras, and that $(\mathsf{Ind}^{\mathbb{T}},\mathsf{Haus}^{\mathbb T})$ is \emph{quasi-hereditary} \cite{GR}: $\MC$-hereditary for $\MC$ the class of regular monomorphisms. Since any protosplit monomorphism is regular, we conclude via Theorem \ref{protoM} that $I$ is protoadditive.

\item 
Abelianisation functors are usually \emph{not} protoadditive (see the last paragraph in Example \ref{counterproto}.\ref{extfree}, for instance). However, here is a nontrivial example of one that is protoadditive. Let $\mathsf{Rng^*}$ be the semi-abelian variety of (not necessarily unital) rings satisfying the identity $xyxy=xy$. The abelian objects in $\mathsf{Rng^*}$ are the $0$-rings: rings satisfying the identity $xy=0$. The reflector $\mathsf{ab} \colon \mathsf{Rng^*} \rightarrow{\mathsf{0\textrm{-}Rng} }$ in the adjunction
\[
\xymatrix{
{ \mathsf{Rng^*} }\,\, \ar@<1ex>[r]^-{\mathsf{ab}} & {\mathsf{0\textrm{-}Rng} }
\ar@<1ex>[l]^-{}_-{_{\perp}}}
\]
sends a ring $A$ in $\mathsf{Rng^*}$ to the quotient $\mathsf{ab}(A) = A/[A,A]$ of $A$ by the ideal $[A,A] = \{\sum_i a_i{a}_i' \mid a _i \in A, {a}_i' \in A  \} $ consisting of all (finite) sums of products of elements in $A$. We claim that the functor $\mathsf{ab}$ is protoadditive. By Proposition \ref{reflector=radical}, it suffices to prove that the corresponding radical $T=[\cdot, \cdot ]\colon \mathsf{Rng^*}\to\mathsf{Rng^*}$ is protoadditive. To this end, we consider a  split short exact sequence
\[
\xymatrix{0 \ar[r]& K \ar[r] & A \ar@<-.8 ex> [r]_-{f} & B \ar@<-.8 ex>[l]_-{s} \ar[r] &0}
\]
in $\mathsf{Rng^*}$ and the restriction induced by the radical $T$ in $\mathsf{Rng^*}$:
\[
\xymatrix{ T(K) \ar[r] & T(A) \ar@<-.8 ex> [r]_-{T(f)} &T(B) \ar@<-.8 ex>[l]_-{T(s)}.}
\]
We shall prove that $T(K) = K[T(f)]$, which will imply that the lower sequence is exact. Let $a=\sum_i a_i{a}_i' $ be an element of $T(A)$ such that $T(f(a)) = f(a)=0$. We have to prove that  $a\in T(K)$. But any element $a_i \in A$ can be written as $a_i=k_i + s(b_i)$ for some $k_i \in K$ and $b_i \in B$ and, similarly, ${a}_i' = k_i' + s(b_i')$. Notice that $f(a)=0$ implies that $\sum_i b_i b_i' =0$. Hence, using the identity $xyxy=xy$ we find that  
\begin{eqnarray*}
a &=& \sum_i k_ik_i' + s(b_i)k_i' + s(b_i')k_i + b_i b_i' \\
 &=& \sum_i k_ik_i' + s(b_i)k_i' + s(b_i')k_i \\
&=&  \sum_i k_ik_i' + (s(b_i)k_i')(s(b_i)k_i') + (s(b_i')k_i)(s(b_i')k_i). 
\end{eqnarray*}
Since $K$ is a two-sided ideal of $A$, this shows that $a\in T(K)$.

Notice also that the identity $xyxy=xy$ implies that the radical $T$ is idempotent so that $0\textrm{-}\Rng$ is a torsion-free subcategory of $\mathsf{Rng^*}$ by Theorem~\ref{torsiontheorem}. 
\end{enumerate}
 \end{examples}

We conclude this section with some (counter)examples, to show the independence of the notions of protoadditivity, admissibility, semi-left-exactness and Barr-exactness (=the preservation of kernel pairs of regular epimorphisms), for a (normal epi)-reflector to a full subcategory.

\begin{examples}\label{counterproto}
\begin{enumerate}
\item \emph{A torsion-free reflector which is not protoadditive.}\label{extfree} Consider the category $\Gp$ of groups and the subquasivariety ${\Gp}_{t.f.}$ of torsion-free groups (=groups satisfying, for all $n\geq 1$, the implication $x^n=1\Rightarrow x=1$). ${\Gp}_{t.f.}$ is easily seen to be a torsion-free subcategory of $\Gp$ with corresponding torsion subcategory consisting of all groups generated by elements of finite order. However, the reflector $F \colon \Gp \rightarrow  {\Gp}_{t.f.} $ is not protoadditive. To see this, we shall give an example already considered in \cite{GJ} for a different purpose. Consider the infinite dihedral group $C_2  \ltimes  \mathbb Z$, where the action of $C_2=\{1, c\}$ on the group of integers $\mathbb Z$ is given by
$c \cdot z = -z$ and $1\cdot z=z, \, \forall z \in \mathbb Z$. The canonical injections of $\mathbb{Z}$ and $C_2$ and the projection on $C_2$ determine a split short exact sequence 
$$
\xymatrix{0 \ar[r] & {\mathbb Z} \ar[r] & C_2  \ltimes  \mathbb Z  \ar@<-.8 ex>[r]_-{} & C_2\ar@<-.8 ex>[l]_-{}  \ar[r] & 0}
$$
which is not preserved by $F$, since its image by $F$ is
\[
 \xymatrix{  {\mathbb Z} \ar[r] & 0   \ar@<-.8 ex>[r]_-{} & 0 \ar@<-.8 ex>[l]_-{}  \ar[r] & 0.}
\]
Observe that this same split short exact sequence can be used to show that the abelianisation functor $\mathsf{ab}\colon \Gp \rightarrow \Ab$ is not protoadditive: while both $\mathbb Z$ and $C_2$ are abelian groups, $C_2  \ltimes  \mathbb Z$ is not, and one concludes via Proposition~\ref{splitext}.

\item \emph{A protoadditive reflector which is not admissible.}
Consider the variety $\Ab$ of abelian groups and the quasivariety $\Fc$ of abelian groups determined by the implication ($4x=0 \Rightarrow 2x=0$). The reflector $F \colon \Ab \rightarrow \Fc$ is additive, thus in particular protoadditive. However,  $F$ is not admissible (with respect to surjective homomorphisms). Let $C_n$ denote the cyclic group of order $n$ ($n\geq 1$) and $\Z$ the group of integers. Then consider the reflection unit $\eta_{C_4}\colon C_4\to F(C_4)=C_2$ and the surjective homomorphism $\Z\to C_2$ in $\Fc$, and note that their pullback (the left hand square below) is sent to the right hand square below, which is not a pullback:
\[
\xymatrix{
C_4\times_{C_2}\Z \ar[r] \ar[d] & C_4 \ar[d] & C_4\times_{C_2}\Z \ar[r] \ar[d] & C_2 \ar@{=}[d] \\
\Z \ar[r] & C_2 & \Z \ar[r] & C_2}
\]

\item\label{Boolean} \emph{A Barr-exact admissible reflector which is not protoadditive.} 
Consider the variety $\mathsf{Rng}$ of nonassociative nonunital rings and its subvariety $\mathsf{Boole}$ of nonassociative Boolean rings, determined by the identity $x^2=x$. Since the reflector $I \colon \mathsf{Rng} \rightarrow \mathsf{Boole}$ sends groupoids to groupoids (by Lemma \ref{Marino}) and $\mathsf{Boole}$ is an arithmetical category (which means that every internal groupoid in $\mathsf{Boole}$ is an equivalence relation---see Example $2.9.13$ in \cite{BB}), $I$ is Barr-exact. However, $I$ is not protoadditive. To see this, consider the split short exact sequence
\[
\xymatrix{
0 \ar[r] &  C_2 \ar[r]^-{i_2} & C_2 \ltimes C_2  \ar@<-.8 ex>[r]_-{p_1} & C_2 \ar[r] \ar@<-.8 ex>[l]_-{i_1}& 0
}
\]
in $\mathsf{Rng}$, where $C_2= {\mathbb Z} / 2 {\mathbb Z}$, the addition in the ring $C_2 \ltimes C_2$ is defined by $(a,b) +(c,d)= (a+c, b+d)$, the multiplication by $(a,b) \cdot (c,d) = (ac, bc+bd)$, and the morphisms $i_1$ and $i_2$ are the canonical injections and $p_1$ the canonical projection on the first component. While $C_2$ is Boolean,  $C_2 \ltimes C_2$ is not, since $(1,1)\cdot (1,1)= (1,1+1)= (1,0)$. Hence, $\mathsf{Boole}$ is not closed in $\mathsf{Rng}$ under split extensions, and we conclude via Proposition \ref{splitext} that $I$ is not protoadditive.

\item\emph{A protoadditive torsion-free reflection which is not Barr-exact.} 
The reflector $\pi_0 \colon \Gpd(\Ac) \rightarrow \Ac$ from Example \ref{exproto}.\ref{exgroupoids} is not Barr-exact, in general.  For instance, if $\Ac=\Gp$ and $A$ is an abelian group, then the diagram
\[
\xymatrix{
A\times A \ar[r]^-{-} \ar@<1.2 ex>[d]^{\pi_2}\ar@<-1.2 ex>[d]_{\pi_1} & A  \ar@<1.2 ex>[d] \ar@<-1.2 ex>[d] \\
A \ar[u]|{\delta} \ar[r] & 0,\ar[u] }
\]
where the the upper horizontal morphism $- \colon A \times A \rightarrow A$ sends a pair $(a_1,a_2)$ to the difference $a_1-a_2$, $\pi_1$ and $\pi_2$ are the product projections and $\delta$ the diagonal morphism, is a regular epimorphism of internal groupoids in $\Gp$ whose kernel pair is not preserved by the functor $\pi_0$.

\item
Another example of the same kind is the reflector $I \colon {\TCom }\to {\TPro}$ from Example \ref{exproto}.\ref{exdisc}: $\TPro$ is a torsion-free subcategory of $\TCom$ and $I$ is protoadditive, but not Barr-exact.  Indeed, if $I$ were Barr-exact, then it would preserve short exact sequences (i.e. it would be a protolocalisation in the sense of \cite{BCGS}) since the kernel of any morphism can be obtained via the kernel of one of the kernel pair projections (as in the proof of Lemma \ref{compositeisnormal }), and this latter kernel is preserved because $I$ is protoadditive. However, the short exact sequence 
\[
\xymatrix{
0 \ar[r] & \{-1,1\}  \ar[r] & S^1  \ar[r] & S^1/\{-1,1\} \ar[r] & 0
}
\]
in the category of compact Hausdorff groups,
where $S^1$ is the unit circle group equipped with the topology induced by the Euclidean topology on ${\mathbb R}^2$, is not preserved, since both $S^1$ and $S^1/\{-1,1\}$ are connected while $\{-1,1\}$ is not.

\item
Other examples of this kind are provided by cohereditary (=the torsion-free part is closed under quotients) torsion theories in the abelian context whose corresponding reflector is not a localisation.

\item\emph{An additive admissible reflector which is not torsion-free.}
Consider the variety $\Ab$ of abelian groups, and the Burnside variety $\mathsf{B}_2$ of exponent $2$: $\mathsf{B}_2$ consists of all abelian groups $A$ such that $a+a=0$ for any $a\in A$. Then the reflector $\Ab\to\mathsf{B}_2$ is additive and admissible (with respect to regular epimorphisms) \cite{JK}, but $\mathsf{B}_2$ is not a torsion-free subcategory of $\Ab$, since the induced radical $T\colon \Ab\to\Ab$ is not idempotent: for instance, by considering the cyclic group $C_4$ we see that $T(C_4)=C_2$ while $T(T(C_4))=T(C_2)=0$.
\end{enumerate}
\end{examples}

\section{Torsion-free subcategories with a protoadditive reflector}\label{coveringmorphisms}
In \cite{CJKP} Carboni, Janelidze, Kelly and Par\'e considered for every factorisation system $(\Ef,\Mf)$ on a finitely complete category $\Ac$ classes $\Ef'$ and $\Mf^*$ of morphisms in $\Ac$ defined as follows: $\Ef'$ consists of all morphisms $f$ that are \emph{stably in $\Ef$}, i.e.~ every pullback of $f$ is in $\Ef$; while $\Mf^*$ consists of all morphisms $f\colon A\to B$ that are \emph{locally in $\Mf$}, this meaning that there exists an effective descent morphism $p\colon E\to B$ for which the pullback $p^*(f)$ is in $\Mf$. Thus, if $(\Ef,\Mf)$ is the factorisation system associated with an admissible (semi-left-exact) reflection, then $\Mf^*$ consists of all central extensions with respect to the corresponding absolute Galois structure. While it is always true that $\Ef'\subseteq (\Mf^*)^{\uparrow}$, one does not necessarily have that $(\Ef',\Mf^*)$ is a factorisation system. However, this does happen to be the case in a number of important examples. For instance, when $(\Ef,\Mf)$ is the factorisation system on the category of compact Hausdorff spaces associated with the reflective subcategory of totally disconnected spaces: in this case, stabilising $\Ef$ and localising $\Mf$ yields the Eilenberg and Whyburn monotone-light factorisation for maps of compact Hausdorff spaces \cite{Eilenberg,Whyburn}. Another example given in \cite{CJKP} is that of a hereditary torsion theory $(\Tc,\Fc)$ in an abelian category $\Ac$ with the property that every object of $\Ac$ is the quotient of an object of $\Fc$. Also in this case the associated factorisation system $(\Ef,\Mf)$ induces a factorisation system  $(\Ef',\Mf^*)$. In fact, as we shall explain in \cite{EG9}, this remains true when the category $\Ac$ is merely homological and  $\Tc$ is not asked to be hereditary. 

Even if $(\Ef',\Mf^*)$ fails to be a factorisation system, it might still be ``partially'' so, in the sense that $\Ac$ admits ``monotone-light'' factorisations, but only for morphisms of a particular class. We shall prove in this section that this is the case for the class of effective descent morphisms in a homological category $\Ac$, if $(\Ef,\Mf)$ is the factorisation system associated with a torsion-free subcategory $\Fc$ whose reflector $F\colon \Ac\to \Fc$ is protoadditive (and such that condition $(N)$ below is satisfied). 

Let $\Ac$ be a homological category, $(\Tc,\Fc)$ a torsion theory in $\Ac$ with reflector $F\colon \Ac\to\Fc$ and coreflector $T\colon \Ac\to\Tc$, $(\Ef,\Mf)$ the associated reflective factorisation system and $(\Ef',\Mf^*)$ as defined above. We shall also consider the classes $\overline{\Ef}$ and $\overline{\Mf}$ defined as follows: $\overline{\Ef}$ is the class of all normal epimorphisms $f$ in $\Ac$ such that $K[f]\in\Tc$; $\overline{\Mf}$ is the class of all morphisms $f$ in $\Ac$ such that $K[f]\in\Fc$. As we shall see, it is always true that $\overline{\Ef}\subseteq \overline{\Mf}^{\uparrow}$, and it is ``often'' the case that $(\overline{\Ef},\overline{\Mf})$ is a factorisation system. In the present section, we shall mostly be concerned with comparing $(\Ef',\Mf^*)$ with $(\overline{\Ef},\overline{\Mf})$. In particular, we shall be interested in conditions on the torsion theory $(\Tc,\Fc)$ under which one has for any effective descent morphism $f$ in $\Ac$ that $f\in\Mf^*$ if and only if $f\in\overline{\Mf}$.

We shall consider the conditions
\begin{enumerate}
\item[(P)] the reflection $F\colon \Ac\to\Fc$ is protoadditive;
\item[(N)]\label{conditionPageN} for any morphism $f\colon A\to B$ in $\Ac$, $\ker (f)\circ t_{K[f]}\colon T(K[f])\to A$ is a normal monomorphism.
\end{enumerate}

\begin{remark}\label{conditionN} 
Recall from Section \ref{protoadditivesection} that condition $(P)$ is satisfied if and only if $\Tc$ is $\MC$-hereditary, for $\MC$ the class of protosplit monomorphisms in $\Ac$.

Condition $(N)$ is trivially satisfied if $\Ac$ is an abelian category, and this is also the case if $\Ac$ is the category of groups: it suffices to observe that the inner automorphisms on $A$ restrict to $T(K[f])$, for any $f\colon A\to B$. 

If $\Ac$ is homological and $\Tc$ is  \emph{quasi-hereditary} in the sense of \cite{GR}, then $(N)$ is satisfied, since in this case we have that $T(K[f])=K[f]\cap T(A)$ for any morphism $f\colon A\to B$. From Lemma \ref{compositeisnormal } and Remark \ref{idealtopological} we know that if $\Ac$ is either semi-abelian or a category of topological semi-abelian algebras then $(P)$ implies $(N)$. 
\end{remark}

\begin{remark}
When studying the factorisation system associated with an absolute Galois structure it is natural to consider also the condition

(C) ``$\Fc$ covers $\Ac$'': for any object $B\in\Ac$, there exists an effective descent morphism $E\to B$ such that $E\in\Fc$.

This condition has been considered before by several authors, for instance, in \cite{CJKP,JMT1, GJ,GR}. 
Several of the results in this section established under conditions $(N)$ and $(P)$ have a corresponding ``absolute'' formulation where condition $(C)$ replaces $(P)$: since this article is mainly concerned with condition $(P)$ we decided to leave these developments for another article \cite{EG9}.
\end{remark}

We begin with a characterisation of the normal extensions associated with a torsion theory $(\Tc,\Fc)$ in a homological category $\Ac$ satisfying condition $(P)$ (see also \cite{GJ, GR}). Let us write $\Gamma_{\Fc}$ for the induced absolute Galois structure on $\Ac$.
\begin{proposition}\label{protocentral}
Assume that $(\Tc,\Fc)$ satisfies condition $(P)$. Then for any effective descent morphism $f\colon A\to B$ in $\Ac$ the following conditions are equivalent:
\begin{enumerate}
\item $f \colon A \rightarrow B$ is a normal extension with respect to $\Gamma_{\Fc}$;
\item $f \colon A \rightarrow B$ is a central extension with respect to $\Gamma_{\Fc}$;
\item $K[f] \in \Fc$.
\end{enumerate}
\end{proposition}
\begin{proof}
$(1)\Rightarrow (2)$ is true by definition.

$(2)\Rightarrow (3)$ is well known, but let us recall the argument. Consider the following diagram \eqref{coveringdiagram} in $\Ac$, where the right hand square is a pullback and $p$ is an effective descent morphism:
\[\vcenter{\begin{equation}\label{coveringdiagram}
\xymatrix{
F(P) \ar[d]_{F(p^*(f))} & P \ar@{}[rd]|<<{\pullback} \ar[d]_{p^*(f)}\ar[l]_-{\eta_P} \ar[r] & A \ar[d]^f\\
F(E) & E \ar[r]_p \ar[l]^-{\eta_E} & B}
\end{equation}}
\]
If the left hand square is a pullback as well, then
\[
K[F(p^*(f))] \cong K[p^*(f)] \cong K[f]
\]
and it follows that $K[f]\in\Fc$, since $K[F(p^*(f))]\in\Fc$ as the kernel of the morphism $F(p^*(f))$ in $\Fc$. 

$(3)\Rightarrow (1)$ Consider an effective descent morphism $f\colon A\to B$ with $K[f]\in \Fc$ and its kernel pair $(\pi_1,\pi_2)\colon R[f]\to A$. By applying the reflector $F\colon \Ac\to\Fc$, we obtain the following commutative diagram of split short exact sequences in $\Ac$:
\[
\xymatrix@=35pt{0   \ar[r]& K[f] \ar@{=}[d]_{\eta_{K[f]}}   \ar[r] & R[f] \ar[d]_{\eta_{R[f]}}  \ar@<-.8 ex> [r]_-{\pi_1} & A \ar[d]^{\eta_A} \ar[r]  \ar@<-.8ex>[l] &0 \\
0 \ar[r]& F(K[f]) \ar[r]_{F(k)} & F(R[f]) \ar@<-.8 ex> [r]_-{F(\pi_1)} & F(A) \ar@<-.8ex>[l] \ar[r] &0.}
\]
Since  $\eta_{K[f]}$ is an isomorphism, the right hand square is a pullback, and we conclude that $f$ is a normal extension. 
\end{proof}

Next we show that, for \emph{every} torsion theory $(\Tc,\Fc)$, one has that $\overline{\Ef}\subseteq\overline{\Mf}^{\uparrow}$ -- or equivalently $\overline{\Mf}\subseteq \overline{\Ef}^{\downarrow}$:
\begin{lemma}\label{orthogonal}
For any pair of morphisms $e\in\overline{\Ef}$ and $m\in\overline{\Mf}$ we have that $e\downarrow m$.
\end{lemma}
\begin{proof}
Consider a commutative square as in the right hand side of the diagram
\[
\xymatrix{
K[e] \ar[r]^-{\ker (e)} \ar[d]_k & A \ar[r]^e \ar[d]_a & B \ar@{.>}[ld]\ar[d]^b\\
K[m] \ar[r]_-{\ker (m)} & C \ar[r]_m & D}
\]
and assume that $e\in\overline{\Ef}$ and $m\in\overline{\Mf}$. Since, by assumption, $K[e]\in\Tc$ and $K[m]\in\Fc$, we see that $k$ is the zero morphism. It follows that also $a\circ \ker (e)=\ker (m)\circ k$ is zero. Since $e$ was assumed to be a normal epimorphism, it is the cokernel of its kernel $\ker (e)$, and there exists a unique dotted arrow making the diagram commute. This shows that $e\downarrow m$.
\end{proof}
By further assuming that condition $(N)$ is satisfied, we get a stable factorisation system:
\begin{proposition}\label{inducedfactorisation}
If $(\Tc,\Fc)$ satisfies condition $(N)$, then $(\overline{\Ef},\overline{\Mf})$ is a stable factorisation system on $\Ac$, and $(\Tc,\Fc)$ is uniquely determined by $(\overline{\Ef},\overline{\Mf})$. In fact, there is a bijective correspondence between
\begin{enumerate}
\item
the torsion theories in $\Ac$ satisfying condition $(N)$;
\item
the stable factorisation systems $(\Ef,\Mf)$ on $\Ac$ such that every $e\in\Ef$ is a normal epimorphism.
\end{enumerate}
\end{proposition}
\begin{proof}
Suppose that $(\Tc,\Fc)$ is a torsion theory in $\Ac$ satisfying condition $(N)$. Then any morphism $f\colon A\to B$ in $\Ac$ with kernel $K$ factorises as
\[
\xymatrix{
A \ar[r]^-{q_{T(K)}} & A/T(K) \ar[r]^-m &B,} 
\]
where $q_{T(K)}$ is the cokernel of the composite $\ker(f)\circ t_K\colon T(K)\to A$, and $t_K\colon T(K)\to K$ is the coreflection counit. Clearly, $q_{T(K)}$ lies in $\overline{\Ef}$, and we also have that $m\in\overline{\Mf}$ since from the ``double quotient'' isomorphism theorem (see Theorem $4.3.10$ in \cite{BB}) it follows that
\[
K[m]=K[\coim(m)\colon A/T(K)\to A/K]=K/T(K)=F(K)\in\Fc,
\]
where $\coim(m)$ denotes the normal epi part of the (normal epi)-mono factorisation of $m$. Here we used that $I[m]=I[f]=A/K$ by the uniqueness of the (normal epi)-mono factorisation of $f$. Thus we see that $(\overline{\Ef},\overline{\Mf})$ is a factorisation system since any morphism of $\Ac$ admits an $(\overline{\Ef},\overline{\Mf})$-factorisation and  $\overline{\Ef}\subseteq\overline{\Mf}^{\uparrow}$ by  Lemma~\ref{orthogonal}.

To see that  the class $\overline{\Ef}$ is pullback-stable, it suffices to observe that normal epimorphisms are pullback-stable in the homological category $\Ac$, and that pulling back induces an isomorphism between kernels.

Conversely, given a stable factorisation system $(\Ef,\Mf)$ on $\Ac$ such that every $e\in\Ef$ is a normal epimorphism, we consider the full subcategories $\Tc$ and $\Fc$ of $\Ac$ defined on objects via
\[
\Tc=\{T\in\Ac \ | \ T\to 0\in\Ef \}; \ \ \Fc=\{F\in\Ac \ | \ F\to 0\in\Mf \}.
\]
Then $\hom_{\Ac}(T,F)=\{0\}$ for any $T\in\Tc$ and $F\in\Fc$ since the assumption that $(\Ef,\Mf)$ is factorisation system implies, for any morphism $T\to F$, the existence of the dotted arrow making the following diagram commute:
\[
\xymatrix{
T \ar[r] \ar[d] & 0\ar[d] \ar@{.>}[ld]\\
F \ar[r] & 0.}
\]
Moreover, if for an object $A\in\Ac$,  $m\circ e\colon A\to I\to 0$ is the $(\Ef,\Mf)$-factorisation of the unique morphism $A\to 0$, then 
\[
\xymatrix{
0 \ar[r] & K[e] \ar[r] & A \ar[r]^e & I\ar[r] & 0}
\]
is a short exact sequence with $I\in\Fc$ and also $K[e]\in\Tc$, since $K[e]\to 0\in\overline{\Ef}$ as the pullback of $e$ along the unique morphism $0\to I$. We conclude that $(\Tc,\Fc)$ is a torsion theory. Note that the radical $T\colon \Ac\to\Ac$ is defined on objects $A\in\Ac$ as $T(A)=K[e]$, where $e$ is the ``$\Ef$-part'' of the $(\Ef,\Mf)$-factorisation of the morphism $A\to 0$.

To see that  $(\Tc,\Fc)$ satisfies condition $(N)$, consider a morphism $f\colon A\to B$ with $(\Ef,\Mf)$-factorisation $f=m\circ e$ and kernel $K$. Let $m'\circ e'$ be the $(\Ef,\Mf)$-factorisation of the unique morphism $K\to 0$. Then there is a unique morphism $I'\to I$ such that the diagram below---in which the outer rectangle is a pullback---commutes:
\[
\xymatrix{
K \ar[d] \ar[r]^{e'} & I' \ar[r]^{m'} \ar@{.>}[d] & 0 \ar[d] \\
A \ar[r]_e & I \ar[r]_m & B}
\]
The uniqueness of the $(\Ef,\Mf)$-factorisation of $K\to 0$, together with the pullback-stability of both classes $\Ef$ and $\Mf$, imply that the two squares are pullbacks. Consequently, we have that $\ker (f)\circ \ker (e')=\ker (e)$ and this is a normal monomorphism, as desired.

Clearly, if $(\Tc,\Fc)$ is a torsion theory satisfying $(N)$, then $(\Tc,\Fc)$ coincides with the torsion theory induced by $(\overline{\Ef},\overline{\Mf})$. On the other hand, consider a stable factorisation system $(\Ef,\Mf)$ on $\Ac$ such that every $e\in\Ef$ is a normal epimorphism. Let $(\Tc,\Fc)$ be the induced torsion theory and $(\overline{\Ef},\overline{\Mf})$ the associated factorisation system, and consider a normal epimorphism $e$. Then $e$ is the cokernel of its kernel, i.e.~the following square is both a pullback and a pushout: 
\[
\xymatrix{
K[e] \ar[r] \ar[d] & 0 \ar[d]\\
A \ar[r]_e & B}
\]
Using that the class $\Ef$ is pullback-stable (by assumption) as well as pushout-stable (since $(\Ef,\Mf)$ is a factorisation system), we see that 
\[
e\in\Ef \Leftrightarrow K[e]\to 0\in\Ef \Leftrightarrow K[e]\in\Tc\Leftrightarrow e\in\overline{\Ef}
\]
and it follows that $\Ef=\overline{\Ef}$. Since both $(\Ef,\Mf)$ and $(\overline{\Ef},\overline{\Mf})$ are (pre)factorisation systems, this implies that $(\Ef,\Mf)=(\overline{\Ef},\overline{\Mf})$.
\end{proof}
The following observation will be needed:

\begin{lemma}\label{stablekernel}
Given a torsion theory $(\Tc,\Fc)$ one always has that $\overline{\Ef}\subseteq \Ef'$. 
\end{lemma}
\begin{proof}
It has already been explained in the proof of Proposition \ref{inducedfactorisation} that the class $\overline{\Ef}$ is pullback-stable. Hence, to prove that $\overline{\Ef}\subseteq \Ef'$, it suffices to show that $\overline{\Ef}\subseteq \Ef$. Consider, therefore, a normal epimorphism $f$ in $\Ac$ with $K[f]\in\Tc$. Then $f$ is the cokernel of its kernel $\ker (f)$ and, consequently, $F(f)$ the cokernel (in $\Fc$) of $F(\ker (f))$. Since $F(K[f])=0$ by assumption, this implies that $F(f)$ is an isomorphism.

\end{proof}


We are now ready to prove the what we announced at the beginning of this section concerning the existence of ``monotone-light'' factorisations. We write $\EDM(\Ac)$ (resp. $\NExt_{\Fc}(\Ac)$) for the full subcategory of the arrow category $\Arr(\Ac)$ determined by all effective descent morphisms in $\Ac$ (resp.~all normal extensions with respect to $\Gamma_{\Fc}$).

\begin{theorem}\label{protofactorisation}
If $(\Tc,\Fc)$ is a torsion theory in $\Ac$ satisfying conditions $(P)$ and $(N)$, then the following properties hold:
\begin{enumerate}
\item
$\NExt_{\Fc}(\Ac)$ is a reflective subcategory of $\EDM(\Ac)$;
\item
normal extensions are stable under composition;
\item
any effective descent morphism $f\colon A\to B$ factors uniquely (up to isomorphism) as a composite $f=m\circ e$, where $e$ is stably in $\Ef$ and $m$ is a normal extension; moreover, this factorisation coincides with the $(\overline{\Ef},\overline{\Mf})$-factorisation of $f$.
\end{enumerate}
\end{theorem}
\begin{proof}
$(1)$ By Propostion \ref{inducedfactorisation}, the full subcategory of $\Arr(\Ac)$ determined by the class $\overline{\Ef}$ is reflective in $\Arr(\Ac)$: the reflection of a morphism $f$ with $(\overline{\Ef},\overline{\Mf})$-factorisation $f=m\circ e$ is given by $m$, with unit $e$. To see that this reflection restricts to a reflection $\EDM(\Ac)\to\NExt_{\Fc}(\Ac)$, it suffices to consider Proposition \ref{protocentral} and to note that effective descent morphisms satisfy the strong right cancellation property.

$(2)$ follows from condition $(P)$ only: if $f\colon A\to B$ and $g\colon B\to C$ are normal extensions then $g\circ f$ is still an effective descent morphism, and there is a short exact sequence 
\[
\xymatrix{
0\ar[r] & K[f] \ar[r] & K[g\circ f] \ar[r] & K[g] \ar[r] & 0}
\]
with $K[f]$ and $K[g]$ torsion-free by Proposition \ref{protocentral}. Since $\Fc$ is closed under extensions, $K[g\circ f]$ is torsion-free as well, so that $g\circ f$ is a normal extension, again by Proposition \ref{protocentral}.


$(3)$ Let $f\colon A\to B$ be an effective descent morphism with $(\overline{\Ef},\overline{\Mf})$-factorisation $f=m\circ e$. Then $e$ is stably in $\Ef$ by Lemma \ref{stablekernel}, and it has already been remarked above in $(1)$ that $m$ is a normal extension. The uniqueness of this factorisation follows from the fact that $\Ef'\subseteq (\Mf^*)^{\uparrow}$.
\end{proof}

Before considering some examples, let us show that the assumption that the (normal epi)-reflection $F\colon \Ac\to \Fc$ is torsion-free in the above was crucial. Note that any split epimorphism, and in particular, any morphism $A\to 0$ is an effective descent morphism.

\begin{proposition}
Let $\Ac$ be a pointed category and $(\Ef,\Mf)$ the (pre)factorisation system on $\Ac$ associated with a given reflection $F\colon \Ac\to \Fc$ to a full replete subcategory $\Fc$. If, for every object $A\in\Ac$, the morphism $\tau\colon A\to 0$ admits a factorisation $\tau=m\circ e$ with $e\in \Ef'$ and $m\in \Mf^*$, then $F$ has stable units. 

Thus, in particular, $F$ has stable units whenever $(\Ef',\Mf^*)$ is a factorisation system. 
\end{proposition}
\begin{proof}
Let $A$ be an object of $\Ac$ and $\tau=m\circ e$ the $(\Ef',\Mf^*)$-factorisation of the morphism $\tau\colon A\to 0$, i.e.~$e\in \Ef'$ and $m\in \Mf^*$. Remark that $\tau$ factorises, alternatively, as in the right hand triangle
\[
\xymatrix@=10pt{
 A \ar[rr]^{\tau} \ar[rd]_e && 0 &&&  A \ar[rr]^{\tau} \ar[rd]_{\eta_A} && 0  \\
& I \ar[ru]_{m}&&&&& F(A) \ar[ru] &}
\]
 (where $\eta_A\colon A\to F(A)$ is the reflection unit) and notice that $\eta_A\in \Ef$ and $F(A)\to 0\in \Mf$. If we can prove $\tau=e\circ m$ too is an $(\Ef,\Mf)$-factorisation, then both factorisations coincide (up to isomorphism) and it will follow that $\eta_A\in\Ef'$, as desired. Since $\Ef'\subseteq \Ef$ by definition, it will suffice to show that $m\in\Mf$. 
 
 Since $m\in \Mf^*$, there exists in $\Ac$ an effective descent morphism $E\to 0$ such that the product projection $\pi_E\colon E\times I\to E$ lies in $\Mf$. But this implies that also $m\in \Mf$, since $m$ appears as a pullback of $\pi_E$ in the diagram
\[
\xymatrix{
I \ar[d]_m \ar[r] \ar@{}[rd]|<<{\pullback} & E\times I \ar[d]_{\pi_E} \ar[r] \ar@{}[rd]|<<{\pullback} & I \ar[d]^m\\
0 \ar[r] & E \ar[r] & 0.}
\]
\end{proof}

\begin{examples}
\begin{enumerate}

\item
Recall from Example \ref{exproto}.\ref{exgroupoids} that any semi-abelian category appears, via the discrete equivalence relation functor $D\colon \Ac\to \Gpd(\Ac)$, as a torsion-free subcategory of the category $\Gpd(\Ac)$ of internal groupoids in $\Ac$. The corresponding torsion subcategory is the category $\Conn$ of connected groupoids (see \cite{EG}). The connected components functor $\pi_0\colon \Gpd(\Ac)\to\Ac$ is protoadditive, hence condition $(N)$ is satisfied by Lemma \ref{compositeisnormal }, since $\Gpd(\Ac)$ is semi-abelian.  

We already know from \cite{G,EG} that the normal extensions are precisely the regular epic discrete fibrations. Here we add that every regular epimorphism $f\colon A\to B$ in $\Gpd(\Ac)$ factorises, essentially uniquely, as $f=m\circ e$, where $e$ is a regular epimorphism with connected kernel and $m$ is a discrete fibration. Note that since $e\downarrow n$ for any $n\in \overline{\Mf}$, this is in particular true for any discrete fibration $n$. One says in this case that $e$ is \emph{final}, and the factorisation $f=m\circ e$ is the so-called \emph{comprehensive} factorisation of $f$ (see \cite{B}). 

\item
The torsion theory $(\TConn,\TPro)$ in the category $\TCom$ of compact $\mathbb T$-algebras, for $\mathbb T$ a semi-abelian theory, considered in Example \ref{exproto}.\ref{exdisc}, satisfies condition $(P)$, and therefore also $(N)$, since $\TCom$ is semi-abelian. Hence we obtain that the normal extensions are precisely the regular epimorphisms (=open surjective homomorphism) with a totally disconnected kernel, and any regular epimorphism $f$ of compact $\mathbb T$-algebras factorises as $f=m\circ e$, where $e$ is a regular epimorphism with a connected kernel, and $m$ a normal extension. 

If $\mathbb T$ is the theory of groups, then it is well known that as soon as the kernel of a continuous homomorphism $f\colon A\to B$ is connected (respectively, totally disconnected), then for \emph{any} element $b\in B$ the fibre $f^{-1}(b)$ over $b$ is connected (respectively, totally disconnected). This remains true for  $\mathbb T$ an arbitrary semi-abelian theory (see \cite{BC2}). Consequently, the factorisation $f=m\circ e$ of a regular epimorphism of compact $\mathbb T$-algebras obtained above is just the
classical monotone-light factorisation of the continuous map $f$.

\item  Recall from Example \ref{exproto}.\ref{exdisc} that the pair of categories $({\mathsf{Ind}}^{\mathbb T}, {\mathsf{Haus}}^{\mathbb T})$ of indiscrete semi-abelian algebras and of Hausdorff semi-abelian algebras forms an $\mathcal M$-hereditary  torsion theory in the category $\mathsf{Top}^{\mathbb T}$ of topological semi-abelian algebras, where $\mathcal M$ is the
class of \emph{regular} monomorphisms. 
By Theorem \ref{protoM} it follows that condition $(P)$ is satisfied, and then also $(N)$ is satisfied, as observed in Remark \ref{idealtopological}. 
The effective descent morphisms in ${\mathsf{Top}}^{\mathbb T}$ are the open surjective homomorphisms.
Accordingly, any open surjective homomorphism $f\colon A\to B$  factors as $f=m\circ e$, where $e$ is an open surjective homomorphism with an \emph{indiscrete} kernel, and $m$ is an open surjective homomorphism with a \emph{Hausdorff} kernel. 

\item The hereditary torsion theory ($\mathsf{NilCRng},\mathsf{RedCRng}$) in the semi-abelian category $\mathsf{CRng}$ of commutative rings (Example \ref{exproto}.\ref{rings}) satisfies both conditions $(P)$ and $(N)$. Consequently, any surjective homomorphism $f\colon A\to B$  in $\mathsf{CRng}$ factors as $f=m\circ e$, where $e$ is a surjective homomorphism with a \emph{nilpotent} kernel, and $m$ is a normal extension, namely a surjective homomorphism with a \emph{reduced} kernel.
\end{enumerate}
\end{examples}

\section{Derived torsion theories}\label{sectionderived}
In the previous section, a torsion theory $(\Tc,\Fc)$ in a homological category $\Ac$ satisfying conditions $(P)$ and $(N)$ was shown to induce a reflective subcategory $\NExt_{\Fc}(\Ac)$ of normal extensions (with respect to the corresponding absolute Galois structure $\Gamma_{\Fc}$) in the category  of effective descent morphisms $\EDM(\Ac)$.  
We shall prove in the present section that $\NExt_{\Fc}(\Ac)$ is, in fact, a torsion-free subcategory of $\EDM(\Ac)$. This implies that $\NExt_{\Fc}(\Ac)$ itself determines an admissible Galois structure $\Gamma_{\NExt_{\Fc}(\Ac)}$. As we shall see, $\Gamma_{\NExt_{\Fc}(\Ac)}$ in its turn gives rise to a torsion theory in the category of \emph{double extensions}, as defined below, whose torsion-free part consists of those double extensions that are normal with respect to it. Continuing, we shall obtain a chain of torsion theories in the categories of \emph{$n$-fold extensions} ($n\geq 1$)---and, accordingly, also a chain of admissible Galois structures---whose torsion-free part consists of the $n$-fold extensions that are normal with respect to the previous Galois structure in the chain. We shall call these induced torsion theories \emph{derived torsion theories} of $(\Tc,\Fc)$.

The thus obtained chain of Galois structures should be compared to the Galois structures of so-called ``higher central extensions''  which have played a central role in recent developments in non-abelian homological algebra (see, in particular \cite{EGV}), and which will be considered in the next sections. It will become clear in what follows that if the category $\Ac$ is semi-abelian, and the torsion-free subcategory $\Fc$ is closed in $\Ac$ under regular quotients (in this case one speaks of a \emph{cohereditary} torsion theory), then the torsion-free parts of the derived torsion theories are exacty the categories of higher central extensions.



We begin by proving that \emph{any} torsion theory $(\Tc,\Fc)$ in a homological category $\Ac$ satisfying condition $(N)$ (but not necessarily $(P)$) induces a chain of torsion theories $(\Tc_n,\Fc_n)$ in the categories $\Arrn(\Ac)$ ($n\geq 1$). These will then be shown to restrict to the derived torsion theories in the categories of $n$-fold extensions mentioned above, when $(\Tc,\Fc)$ moreover satisfies condition $(P)$, and every regular epimorphism in $\Ac$ is an effective descent morphism. 

Since, by Proposition \ref{inducedfactorisation}, any torsion theory satisfying condition $(N)$ induces a stable factorisation system $(\overline{\Ef},\overline{\Mf})$ such that every $e\in\overline{\Ef}$ is a normal epimorphism,  it is natural to consider the following lemma:

\begin{lemma}\label{inducedtt}
Let $\Ac$ be a pointed category with kernels of normal epimorphisms. Any stable factorisation system $(\Ef,\Mf)$ on $\Ac$ for which $\Ef$ is contained in the class of normal epimorphisms induces a torsion theory $(\Tc_{\Ef},\Fc_{\Mf})$ in $\Arr(\Ac)$.  Here $\Fc_{\Mf}$ is the full subcategory of $\Arr(\Ac)$ determined by $\Mf$, and  $\Tc_{\Ef}$ the full subcategory of $\Arr(\Ac)$ consisting of all $e\in \Ef$ of the form $e\colon T\to 0$.
\end{lemma}
\begin{proof}
Let $f$ be a morphism $f\colon A\to B$ in $\Ac$ with $(\Ef,\Mf)$-factorisation $f=m\circ e$. Since, by assumption, $e$ is a normal epimorphism, the following diagram     
is a short exact sequence in $\Arr(\Ac)$:
\[
\xymatrix{
0 \ar[r] & K[e]  \ar[d] \ar[r] & A \ar[r] \ar[r]^-{e} \ar[d]_f & C\ar[d]^{m} \ar[r] & 0\\
0 \ar[r] & 0 \ar[r] & B \ar@{=}[r] & B \ar[r] & 0.}
\]
 Moreover, the morphism $K[e]\to 0$ lies in $\Ef$ since it is the pullback of $e$ along $0\to C$; and $m$ lies in $\Mf$, by assumption. 

Now consider a morphism $t\to f$ in $\Arr(\Ac)$ with $t\colon T\to 0$ in $\Ef$ and $f\colon A\to B$ in $\Mf$:
\[
\xymatrix{
  T \ar[r] \ar[d]_t & A \ar[d]^f\\
  0  \ar@{.>}[ru] \ar[r] & B.}
\]
Since $t\downarrow f$ there exists the dotted arrow making the diagram commutative, and it follows that $t\to f$ is the zero morphism, as desired. We can conclude that $(\Tc_{\Ef},\Fc_{\Mf})$ is a torsion theory in $\Arr(\Ac)$.
\end{proof}

Combining Proposition \ref{inducedfactorisation} with Lemma \ref{inducedtt}, we obtain the following:
\begin{proposition}\label{firstderivedtt}
Let $\Ac$ be a homological category. Any torsion theory $(\Tc,\Fc)$ in~$\Ac$ satisfying condition $(N)$ induces a torsion theory $(\Tc_1,\Fc_1)$ in $\Arr(\Ac)$ which again satisfies $(N)$. Here $\Fc_1$ is the full subcategory of $\Arr(\Ac)$ consisting of all morphisms $f$ with $K[f]\in\Fc$, and $\Tc_1$ is the full subcategory of $\Arr(\Ac)$ of all morphisms $T\to 0$ with $T\in\Tc$. 

If $(\Tc,\Fc)$ satisfies, besides condition $(N)$, also condition $(P)$, then $(\Tc_1,\Fc_1)$ satisfies $(P)$ as well.
\end{proposition}
\begin{proof}
By Proposition \ref{inducedfactorisation}, $(\Tc,\Fc)$ induces a factorisation system $(\overline{\Ef},\overline{\Mf})$, which in its turn gives rise to a torsion theory $(\Tc_{\overline{\Ef}},\Fc_{\overline{\Mf}})$, by Lemma \ref{inducedtt}. It follows immediately from the definitions that $(\Tc_1,\Fc_1)=(\Tc_{\overline{\Ef}},\Fc_{\overline{\Mf}})$.

To see that the torsion theory $(\Tc_1,\Fc_1)$ satisfies condition $(N)$, consider a morphism
\[
\xymatrix{
 A \ar[r]^f\ar[d]_a & B \ar[d]^b \\ 
 A' \ar[r]_{f'} & B'.}
\]
in $\Arr(\Ac)$  with kernel $k\colon K[f]\to K[f']$. Then 
\[
K[k]=K[a]\cap K[f]=K\big[(a,f)\colon A\to A'\times B\big]
\]
so that $\ker (k)\circ t_K \colon T(K[k])\to A$ is a normal monomorphism by condition $(N)$ of $(\Tc,\Fc)$. Hence, the induced morphism $T_1(k)\to a$ is a normal monomorphism in $\Arr(\Ac)$, and it follows that $(\Tc_1,\Fc_1)$ satisfies condition $(N)$.

Clearly, if the torsion theory $(\Tc,\Fc)$ is $\MC$-hereditary, for $\MC$ the class of protosplit monomorphisms (i.e.~if it satisfies condition $(P)$), then so is $(\Tc_1,\Fc_1)$. 
\end{proof}

By repeatedly applying the above proposition, one obtains, for every $n\geq 1$, a torsion theory $(\Tc_n,\Fc_n)$ in the (homological) category $\Arrn(\Ac)$ of $n$-fold morphisms in $\Ac$. We shall write $F_n$ for the reflection $\Arrn(\Ac)\to\Fc_n$, $T_n$ for the coreflection $\Arrn(\Ac)\to\Tc_n$, $\eta^{n}$ for the unit $1_{\Arrn(\Ac)}\Rightarrow F_n$ and $t^n$ for the counit  $1_{\Arrn(\Ac)}\Rightarrow T_n$.


Note that an $n$-fold morphism $A$ in $\Ac$ (for $n\geq 1$) determines a commutative $n$-dimensional cube in $\Ac$. We shall sometimes write $a_i$ ($1\leq i\leq n$) for the ``initial'' ribs of this cube. We denote by $\iota$ the functor $\Ac\to \Arr(\Ac)$ that sends an object $A\in\Ac$ to the unique morphism $A\to 0$. For $n\geq 1$, we write $\iota^n$ for the composite functor $\iota\circ \dots \circ\iota\colon \Ac\to\Arrn(\Ac)$. 




As before, we denote by $\EDM(\Ac)$ the category of effective descent morphisms in $\Ac$. If $(\Tc,\Fc)$ satisfies conditions $(P)$ and $(N)$, then by Proposition \ref{protocentral} (and by the strong right cancellation property of effective descent morphisms) the reflection $F_1\colon \Arr(\Ac)\to\Fc_1$ restricts to a reflection $\EDM(\Ac)\to\NExt_{\Fc}(\Ac)$, where $\NExt_{\Fc}(\Ac)$ is the category of normal extensions with respect to $\Gamma_{\Fc}$. We shall prove below that this is still a torsion-free reflection and, moreover, that also for $n\geq 2$, the categories $\Fc_n$ restrict to suitably defined torsion-free subcategories $\NExt^n_{\Fc}(\Ac)$ of the categories of \emph{$n$-fold extensions}, of which we now recall the definition. 

If $\Ec$ is a class of morphisms in $\Ac$, then we shall write $\Ec^1$ for the class of morphisms in $\Arr(\Ac)$ defined as follows: a morphism $(f,f')\colon a\to b$ in $\Arr(\Ac)$ lies in $\Ec^1$ if every morphism in the commutative diagram 
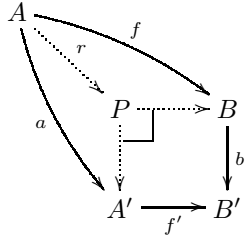
\[
\vcenter{\xymatrix{ 
A  \ar@/^/@{->}[drr]^{f} \ar@/_/@{->}[drd]_{a} \ar@{.>}[rd]^r \\
& P\ar@{}[rd]|<{\pullback} \ar@{.>}[r] \ar@{.>}[d] & B \ar@{.>}@{->}[d]^{b} \\
& A' \ar@{->}[r]_{f'} & B'}}
\]
lies in $\Ec$. Here,  $r$ is the unique factorisation to the pullback $P=A'\times_{B'} B$.

\begin{remark}\label{exactmaltsev}
Note that if $\Ec$ is a class of regular epimorphisms in a regular category $\Ac$, then any commutative square in $\Ac$ of morphisms in $\Ec$ is a pushout as soon as it is a pullback. Consequently, any element of $\Ec^1$ is a pushout square.  If we choose $\Ec$ to be the class of \emph{all} regular epimorphisms in the regular category $\Ac$, then also the converse holds---every pushout square of morphisms in $\Ec$ lies in $\Ec^1$---if and only if $\Ac$ is an exact Mal'tsev category \cite{CKP} (recall that a Mal'tsev category is one where every (internal) reflexive relation is an (internal) equivalence relation). Hence, the converse holds in particular if $\Ac$ is a semi-abelian category.
\end{remark}

Let $\Ec$ be a class of morphisms in $\Ac$. Call \emph{$\Ec$-extensions} the elements $f\in\Ec$, and write $\Ext_{\Ec}(\Ac)$ for the full subcategory of $\Arr(\Ac)$ determined by $\Ec$.  Then $\Ec$ induces a class $\Ec^1$ of double morphisms defined as above, whose elements will be called \emph{double $\Ec$-extensions}. The corresponding full subcategory of $\Arr^2(\Ac)$ will be denoted by $\Ext^2_{\Ec}(\Ac)$. Inductively, $\Ec$ determines, for \emph{any} $n\geq 1$, a class of morphisms $\Ec^n=(\Ec^{n-1})^1$ in $\Arr^n(\Ac)$, the elements of which we call \emph{$(n+1)$-fold $\Ec$-extensions}. We write $\Ext^{n+1}_{\Ec}(\Ac)$ for the corresponding full subcategory of $\Arr^{n+1}(\Ac)$.  

Our main interest is in the situation where $\Ec$ is the class of all normal epimorphisms in a homological category $\Ac$ in which every normal epimorphism is an effective descent morphism. We shall usually denote this class by $\Nc$. In this case, $\Ec=\Nc$ satisfies the list of conditions below (see \cite{Ev}). Here we write  $\Ac_{\Ec}$ for the full subcategory of $\Ac$ determined by the objects $A\in\Ac$ for which there exists in $\Ec$ at least one arrow $f:A\to B$ or one arrow $g:C\to A$. Note that if $\Ec=\Nc$, then we have that $\Ac_{\Ec} =\Ac$.

\begin{conditions}\label{extension}
 On a class $\Ec$ of morphisms in a finitely complete pointed category $\Ac$ we consider the following conditions: 
  \begin{enumerate}
  \item
  every $f\in\Ec$ is a normal epimorphism;
  \item 
  $\Ec$ contains all isomorphisms in $\Ac_{\Ec}$, and $0\in \Ac_{\Ec}$;
  \item
  $\Ec$ is closed under pulling back (in $\Ac$) along morphisms in $\Ac_{\Ec}$;
  \item
  $\Ec$  is closed under composition, and if a composite $g\circ f$ of morphisms in $\Ac_{\Ec}$ is in $\Ec$, then also $g\in\Ec$;
  \item\label{kernelextension}
$\Ec$ is  completely determined by the class of objects $\Ac_{\Ec}$ in the following way: a normal epimorphism $f\colon A\to B$ is in $\Ec$ if and only if both its domain $A$ and its kernel $K[f]$ lie in $\Ac_{\Ec}$;
\item   
For any morphism 
\[
\xymatrix{
0\ar[r] & K \ar[r] \ar[d]_k & A \ar[r] \ar[d]_a & B \ar@{=}[d] \ar[r] & 0\\
0 \ar[r] & L \ar[r] & C \ar[r] & B\ar[r] & 0,}
\]
of short exact sequences in $\Ac$, one has: if $k\in\Ec$ and $a$ lies in $\Ac_{\Ec}$, then $a\in\Ec$.
 \end{enumerate} 
 \end{conditions} 

\begin{remark}\label{remarksplit}
An important consequence of conditions (2) and (4) above is that $\Ec$ contains all split epimorphisms in $\Ac_{\Ec}$.
\end{remark}

 We have the following lemma:
\begin{lemma}\label{up}\cite{Ev}
If $\Ec$ is a class of morphisms in a homological category $\Ac$ satisfying Conditions \ref{extension}, then the class $\Ec^1$ of morphisms in $\Arr(\Ac)$ satisfies Conditions \ref{extension} as well.
\end{lemma}
Note that we have that $(\Arr(\Ac))_{\Ec^1}=\Ext_{\Ec}(\Ac)$. 

Hence, inductively, for any $n\geq 1$, the class $\Ec^n$ of $n$-fold $\Ec$-extensions satisfies Conditions \ref{extension} as soon as this is the case for $\Ec$, and we have that $(\Arr^n(\Ac))_{\Ec^n}=\Ext^{n}_{\Ec}(\Ac)$ (where $\Ec^0=\Ec$, $\Ext_{\Ec}^1(\Ac)=\Ext_{\Ec}(\Ac)$ and $\Arr^1(\Ac)=\Arr(\Ac)$).
 
 \begin{remark}
 Condition \ref{extension}.6 is of importance for Lemma \ref{up}, but shall otherwise not be needed in what follows.
 \end{remark}
 
Let us then show that the torsion theories $(\Tc_n,\Fc_n)$ restrict to torsion theories in the categories $\Ext^n_{\Nc}(\Ac)$ (for $\Nc$ the class of all normal epimorphisms in $\Ac$), where the torsion-free parts consist of what we shall call \emph{$n$-fold normal extensions}. For this, we consider the following lemmas.

For an (internal) equivalence relation $R$ on an object $A$ in $\Ac$, we write $\DiscFib(R)$ for the category of \emph{discrete fibrations} over $R$, i.e.~of morphisms  
\[
\xymatrix{
R' \ar@<.8 ex>[r]^{\pi_1'}  \ar@<-.8 ex>[r]_{\pi_2'} \ar[d]_r & A' \ar[d]^a \\
R \ar@<.8 ex>[r]^{\pi_1}  \ar@<-.8 ex>[r]_{\pi_2} & A,}
\]
of equivalence relations in $\Ac$ into $R$, such that the commutative square $a\circ \pi_2'=\pi_2\circ r$ is a pullback.  

\begin{lemma}\label{descentlemma}
Let $\Ec$ be a class of morphisms in a homological category $\Ac$, satisfying Conditions \ref{extension}, and let $p\in\Ec$ be an effective descent morphism in $\Ac$. Then $p$ is a monadic extension (with respect to $\Ec$) in $\Ac_{\Ec}$.
\end{lemma}
\begin{proof}
Let $p\colon E\to B$ be an effective descent morphism in $\Ac$ such that $p\in\Ec$. We first prove that $p$ is then also an effective descent morphism in $\Ac_{\Ec}$. Since $\Ac_{\Ec}$ is closed under pullback along $p$ (by Condition \ref{extension}.3), we can apply Corollary $3.9$ from \cite{JST}. Thus it suffices to prove that, for any morphism $f\colon A\to B$ in $\Ac$ such that the pullback $P=E\times_B A$ lies in $\Ac_{\Ec}$, one also has that $A\in\Ac_{\Ec}$:
\[
\xymatrix{
P \ar[r]^{f^*(p)} \ar[d] \ar@{}[rd]|<<{\pullback} & A \ar[d]^f \\
E \ar[r]_p & B.}
\]
Since the category $\Ac$ is homological and $p$ is a normal epimorphism, $f^*(p)$ is a normal epimorphism as well. Hence, if we can prove that $K[f^*(p)]\in\Ac_{\Ec}$, it will follow from Condition \ref{extension}.5 that $f^*(p)\in\Ec$ and, in particular, that $A\in\Ac_{\Ec}$. Since we have that $K[f^*(p)]\cong K[p]$, it suffices for this to note that $K[p]\in\Ac_{\Ec}$ because $p\in\Ec$.

We have just proved that $p$ is an effective descent morphism in $\Ac_{\Ec}$. This means that the functor $p^*\colon (\Ac_{\Ec}\downarrow B)\to (\Ac_{\Ec}\downarrow E)$ is monadic or, equivalently (see \cite{JST}), that the functor $(\Ac_{\Ec}\downarrow B)\to \DiscFib(R[p])$ which sends a morphism $f\colon A\to B$ in \\ $\Ac_{\Ec}$ to the discrete fibration obtained by pulling back $f$ along $p$ and then taking kernel pairs, pictured as the left hand square in the diagram
\[
\xymatrix{
R[f^*(p)]  \ar@{}[rd]|<<{\pullback} \ar@<.8 ex>[r]\ar@<-.8 ex>[r] \ar[d] & P \ar[r]^{f^*(p)} \ar@{}[rd]|<<{\pullback} \ar[d] & A \ar[d]^f\\
R[p]   \ar@<.8 ex>[r]\ar@<-.8 ex>[r] & E\ar[r]_p & B}
\]
is an equivalence of categories. To see that $p$ is a monadic extension, it suffices now to note that this category equivalence restricts to an equivalence $\Ext_{\Ec}(B)\to  \DiscFib(R[p])\cap \Ec$ due to the pullback-stability of the class of extensions $\Ec$ and to the fact that $\Ec$ has the strong right cancellation property (Conditions \ref{extension}.3 and \ref{extension}.4).  
\end{proof}

\begin{lemma}\label{torsionrestricts}
Let $\Ac$ be a homological category, and $\Ec$ a class of morphisms in $\Ac$ satisfying Conditions \ref{extension}. If $(\Tc,\Fc)$ is a torsion theory in $\Ac$ such that for any object $A\in\Ac_{\Ec}$ the reflection unit $\eta_A\colon A\to F(A)$ lies in $\Ec$, then the reflection $F\colon \Ac\to\Fc$ and coreflection $T\colon\Ac\to\Tc$ restrict to functors $F\colon \Ac_{\Ec}\to \Fc\cap\Ac_{\Ec}$ and $T\colon \Ac_{\Ec}\to \Tc\cap\Ac_{\Ec}$, and $(\Tc\cap\Ac_{\Ec},\Fc\cap\Ac_{\Ec})$ is a torsion theory in $\Ac_{\Ec}$.

Furthermore, if $f\in\Ec$ is an effective descent morphism in $\Ac$, then $f$ is a trivial extension (respectively a normal extension) with respect to $\Gamma_{\Fc\cap\Ac_{\Ec}}$ if and only if $f$ is a trivial extension (respectively a normal extension) with respect to $\Gamma_{\Fc}$.      
\end{lemma}
\begin{proof}
Consider, for any $A\in\Ac_{\Ec}$, the associated short exact sequence
\[
\xymatrix{
0 \ar[r] & T(A) \ar[r] & A \ar[r]^-{\eta_A} & F(A) \ar[r] & 0.}
\]
By assumption, the unit $\eta_A$ is in $\Ec$, which implies that both $F(A)$ (by definition of $\Ac_{\Ec}$) and $T(A)$ (as the kernel of an extension---by Condition \ref{extension}.3) lie in $\Ac_{\Ec}$. Since $\Ac_{\Ec}$ is a full subcategory of $\Ac$, this implies that the sequence above is a short exact sequence in $\Ac_{\Ec}$. Furthermore, for any objects $T\in\Tc\cap\Ac_{\Ec}$ and  $F\in\Fc\cap\Ac_{\Ec}$ we have that
\[
\hom_{\Ac_{\Ec}}(T,F)=\hom_{\Ac}(T,F)=\{0\},
\]
so that $(\Tc\cap\Ac_{\Ec},\Fc\cap\Ac_{\Ec})$ is indeed a torsion theory in $\Ac_{\Ec}$.

The latter part of the statement follows readily from Lemma \ref{descentlemma} and Condition \ref{extension}.3.
\end{proof}

\begin{lemma}\label{unitlemma}
With the same assumptions as in Lemma \ref{torsionrestricts}: if $(\Tc,\Fc)$ satisfies condition $(N)$ then for any $f\in\Ec$ the unit $\eta^1_f\colon f\to F_1(f)$ lies in $\Ec^1$.
\end{lemma}
\begin{proof}
By Proposition \ref{firstderivedtt}, $(\Tc_1,\Fc_1)$ is a torsion theory in $\Arr(\Ac)$, and the unit $\eta^1_f\colon f\to F_1(f)$ for any $f$ is given by the commutative square
\[
\xymatrix{
 A \ar[d]_f \ar[r]^-{q_{T(K[f])}} & A/T(K[f]) \ar[d]^{F_1(f)} \\
 B \ar@{=}[r] & B,}
\]
where $q_{T(K[f])}$ is the cokernel of the normal monomorphism $\ker(f)\circ t_{K[f]}$. Now suppose that $f$ lies in $\Ec$. Then its kernel $K[f]$ must be in $\Ac_{\Ec}$, which implies that $T(K[f])\in\Ac_{\Ec}$ by Lemma \ref{torsionrestricts}. Consequently, $q_{T(K[f])}\in\Ec$, by Condition \ref{extension}.5, and then also $F_1(f)\in\Ec$, by Condition \ref{extension}.4. From this we conclude that $\eta^1_f\in \Ec^1$.
\end{proof}

Finally, Lemmas \ref{descentlemma}---\ref{unitlemma} together with Propositions \ref{firstderivedtt} and \ref{protocentral} give the following. As before, we write $\Nc$ for the class of normal epimorphisms in $\Ac$.

\begin{theorem}\label{higherderivedT1}
Let $\Ac$ be a homological category in which every normal epimorphism is an effective descent morphism. Then any torsion theory $(\Tc,\Fc)$ in $\Ac$ satisfying conditions $(P)$ and $(N)$ induces, for any $n\geq 1$, a torsion theory $(\Tc_n,\NExt_{\Fc}^n(\Ac))$ in the category $\Ext^n_{\Nc}(\Ac)$ of $n$-fold $\Nc$-extensions. Here $\Tc_n$ is the replete image of $\Tc$ by the functor $\iota^n\colon \Ac\to\Arrn(\Ac)$ and $\NExt_{\Fc}^n(\Ac)=\Fc_n\cap\Ext^n_{\Nc}(\Ac)$ consists of all $n$-fold $\Nc$-extensions that are normal with respect to $\Gamma_{\NExt_{\Fc}^{n-1}(\Ac)}$. Moreover, for any $n\geq 1$, and any $n$-fold $\Nc$-extension $A$, we have that 
\[
A\in\NExt_{\Fc}^n(\Ac) \Leftrightarrow \bigcap_{1\leq i\leq n}K[a_i]\in\Fc.
\]
\end{theorem}

\begin{remark}
For $n\geq 1$, and under the conditions of the theorem above, we call \emph{$n$-fold normal extensions} the objects of $\NExt_{\Fc}^n(\Ac)$ with respect to the Galois structure $\Gamma_{\Fc}$. 
\end{remark}

Theorem \ref{protofactorisation} together with Lemma \ref{torsionrestricts} also imply the following:

\begin{theorem}\label{reflectivehigher}
With the same assumptions and notations as in Theorem \ref{higherderivedT1}, for any $n\geq 0$, if $(\Ef_n,\Mf_n)$ is the factorisation system induced by the reflection $F_n\colon \Ext^n_{\Ec}(\Ac)\to\NExt_{\Fc}^n(\Ac)$, then any $(n+1)$-fold extension $f\colon A\to B$ factors uniquely (up to isomorphism) as a composite $f=m\circ e$ of $(n+1)$-fold $\Nc$-extensions, where $e$ is stably in $\Ef_n$ and $m$ is an $(n+1)$-fold normal extension. 
\end{theorem}
\begin{proof}
If $f=m\circ e$ is the factorisation in $\Arr^n(\Ac)$ given by Theorem \ref{protofactorisation}, then $m$ is an ($n+1$)-fold $\Nc$-extension by Lemma \ref{torsionrestricts}, and $e$ by Condition \ref{extension}.5, since its kernel, which lies in $\Tc_n$, is an  $n$-fold $\Nc$-extension. 
\end{proof}

\section{Birkhoff subcategories with a protoadditive reflector}\label{Birkhoffsection}
Consider a torsion-free subcategory $\Fc$ of a homological category $\Ac$, and write, as usual, $F$ for the reflector $\Ac\to\Fc$ and $T$ for the associated radical. Assume that $\Fc$ satisfies conditions $(P)$ ($F$ is protoadditive) and $(N)$ (for any morphism $f\colon A\to B$ in $\Ac$, the induced monomorphism $T(K[f])\to A$ is normal). In the previous section, we have explained how $\Fc$ induces a chain of ``derived" torsion theories $(\Tc_n,\Fc_n)$ ($n\geq 1$) in the categories $\Ext^n_{\Nc}(\Ac)$ of $n$-fold $\Nc$-extensions (for $\Nc$ the class of normal epimorphisms in $\Ac$) where, for each $n\geq 1$,  $\Fc_n$ consists of all $n$-fold $\Nc$-extensions that are normal with respect to the Galois structure $\Gamma_{\Fc_{n-1}}$. In a similar manner, in \cite{EGV}, ``higher dimensional" Galois structures had been obtained starting from any \emph{Birkhoff subcategory} $\Bc$ of a semi-abelian category $\Ac$. While for this to work there is no need for the reflector $\Ac\to\Bc$ to be protoadditive, the situation where it is so is of interest and will be studied in the present section.

Recall from \cite{JK} that a Birkhoff subcategory of an exact category $\Ac$ is a full reflective subcategory $\Bc$ of $\Ac$ closed under subobjects and regular quotients; or, equivalently, a full replete (regular epi)-reflective subcategory $\Bc$ of $\Ac$, with reflector $I\colon \Ac\to\Bc$, such that, for any regular epimorphism $f\colon A\to B$ in $\Ac$, the canonical square 
\begin{equation}\label{unitsquare}
\xymatrix{
A \ar[r]\ar[d]_f & I(A) \ar[d]^{I(f)}\\
B \ar[r] & I(B)}
\end{equation}
is a pushout. Note that this last condition translates to \eqref{unitsquare} being a double $\Nc$-extension, whenever $\Ac$ is a semi-abelian category---since any semi-abelian category is exact Mal'tsev (see Remark \ref{exactmaltsev}).

\begin{example}
By Birkhoff's theorem characterising equational classes, a full subcategory $\Bc$ of a variety of universal algebras $\Ac$ is a subvariety if and only if $\Bc$ is closed in $\Ac$ under subobjects, quotients and products. It follows that a Birkhoff subcategory of a variety is the same as a subvariety---whence its name. Note that a Birkhoff subcategory is indeed closed under products---it is, in fact, closed under arbitrary limits---since it is a reflective subcategory.   
\end{example}

We shall be needing the following important property of Birkhoff subcategories, which was first observed in \cite{Gran}:
\begin{lemma}\label{Marino}
The reflector $I\colon \Ac\to\Bc$ into a Birkhoff subcategory $\Bc$ of a semi-abelian category $\Ac$ preserves pullbacks of normal epimorphisms along split epimorphisms. In particular, $I$ preserves finite products.
\end{lemma}
\begin{proof}
Consider a commutative cube 
\[
\xymatrix{
& I(P) \ar@{}[rrdd] \ar@{.>}[dd]  \ar[rr] && I(C) \ar[dd] \\
P \ar@{}[rrdd]|<<{\pullback}   \ar[ur] \ar[rr] \ar[dd] && C \ar[ur] \ar[dd] & \\
& I(A) \ar@{.>}[rr] && I(B) \\
A \ar[ur]  \ar[rr] && B \ar[ur] &}
\]
in $\Ac$, where the front square is the pullback of a split epimorphism $A\to B$ along a normal epimorphism $C\to B$, and the skew morphisms are the reflection units. Since $\Bc$ is a Birkhoff subcategory of $\Ac$, we have that the left and right hand sides are double $\Nc$-extensions, so that the cube is a three-fold $\Nc$-extension as a split epimorphism of double $\Nc$-extensions (via Remark \ref{remarksplit} and Lemma \ref{up}). This implies that the induced square
\[
\xymatrix{
P \ar[r] \ar[d] & I(P) \ar[d]\\
A\times_B C \ar[r] & I(A)\times_{I(B)}I(C)}
\]
is a double $\Nc$-extension. In particular, it is a pushout, so that the right hand vertical map is indeed an isomorphism, because the left hand one is so by assumption.

To see that $I$ preserves binary (hence, finite-) products, it suffices to take $B=0$ in the above, and note that $I(0)=0$ as $I$ preserves the initial object.
\end{proof}

\begin{remark}
Note that for the second part of the lemma above to be true, the assumption that $\Bc$ is a Birkhoff subcategory can be weakened: it suffices that $\Bc$ is a (normal epi)-reflective subcategory of $\Ac$ (because a split epimorphism of $\Nc$-extensions is always a double $\Nc$-extension). In fact, as soon as $\Bc$ is a (normal epi)-reflective subcategory of a homological category $\Ac$, the reflector $I\colon\Ac\to\Bc$ will preserve pullbacks of split epimorphisms along split epimorphisms (this even holds more generally in a regular Mal'tsev category).
\end{remark}

Let us, from now on, assume that $\Ac$ is a semi-abelian category. We know from \cite{JK} that any Birkhoff subcategory $\Bc$ of $\Ac$  determines an admissible Galois structure $\Gamma_{(\Bc,\Nc)}=(\Ac,\Bc,I,H,\Nc)$, where $I\colon \Ac\to\Bc$ is the reflector, $H\colon \Bc\to\Ac$ the inclusion functor and $\Nc$ the class of all normal epimorphisms in $\Ac$.  We shall write $[\cdot]_{\Bc}\colon \Ac\to \Ac$ for the associated radical. For $A\in\Ac$, we denote by $\eta_A\colon A\to I(A)$ the reflection unit and by $\kappa_A\colon [A]_{\Bc}\to A$ the normal monomorphism $\ker(\eta_A)$. 

Note that the normal epimorphisms in $\Ac$ coincide with the regular epimorphisms because $\Ac$ is protomodular, and the regular epimorphisms with the effective descent morphisms because $\Ac$ is exact. Since a semi-abelian category is, in particular, homological, the class  $\Nc$ satisfies Conditions \ref{extension}. 


Recall from \cite{JK} that every central extension with respect to $\Gamma_{(\Bc,\Nc)}$ is a normal extension. The category of normal extensions with respect to $\Gamma_{(\Bc,\Nc)}$ is denoted $\NExt_{(\Bc,\Nc)}(\Ac)$. Just as in the case of a torsion theory satisfying Conditions $(P)$ and $(N)$, we have that $\NExt_{(\Bc,\Nc)}(\Ac)$ is a reflective subcategory of the category of effective descent morphisms in $\Ac$, and we know from \cite{EGV} that the reflection $I_1(f)$ in $\NExt_{(\Bc,\Nc)}(\Ac)$ of a normal epimorphism $f\colon A\to B$ can be obtained as follows: $I _1(f)$ is the normal epimorphism $A/[f]_{1,{\Bc}}\to B$ induced by $f$, where the normal monomorphism $[f]_{1,{\Bc}}\to A$ is obtained as the composite $\kappa^1_f=\kappa_A\circ [\pi_2]_{\Bc}\circ\ker [\pi_1]_{\Bc}$ (where $\pi_1$ and $\pi_2$ denote the projections from the kernel pair $R[f]$ of $f$):
\[
\xymatrix{
[f]_{1,{\Bc}}=K[[\pi_1]_{\Bc}] \ar[r]^-{\ker [\pi_{1}]_{\Bc}} \ar[d] & [R[f]]_{\Bc} \ar[d]_{\kappa_{R[f]}} \ar@<0.8 ex>[r]^-{[\pi_1]_{\Bc}} \ar@<-0.8 ex>[r]_-{[\pi_2]_{\Bc}} & [A]_{\Bc} \ar[d]^{\kappa_A}\\
K[\pi_1]  \ar[r]_-{\ker (\pi_{1})} & R[f]  \ar@<0.8 ex>[r]^-{\pi_1} \ar@<-0.8 ex>[r]_-{\pi_2} & A}
\]

In Section \ref{coveringmorphisms}, we proved that, for any torsion-free subcategory $\Fc$ of a homological category $\Ac$ with protoadditive reflector $F\colon \Ac\to\Fc$, the normal extensions with respect to $\Gamma_{\Fc}$ are exactly the effective descent morphisms $f\colon A\to B$ such that $K[f]\in\Fc$. As shown in \cite{EG} one has the same characterisation for the normal extensions with respect to $\Gamma_{(\Bc,\Nc)}$, where $\Bc$ is a Birkhoff subcategory of a semi-abelian category $\Ac$ with protoadditive reflector $I\colon \Ac\to \Bc$. It turns out that the protoadditivity of $I$ is also necessary for this to be true, as soon as $\Bc$ satisfies condition $(N)$: for any normal epimorphism $f\colon A\to B$ in $\Ac$, the induced monomorphism $\ker (f)\circ \kappa_{K[f]}\colon [K[f]]_{\Bc}\to A$ is normal.  More precisely, we have the following proposition:

\begin{proposition}\label{characterisationbyextensions}
For a Birkhoff subcategory $\Bc$ of a semi-abelian category $\Ac$, the following conditions are equivalent:
\begin{enumerate}
\item the reflector $I \colon \Ac \rightarrow \Bc$ is protoadditive;
\item the associated radical $[\cdot]_{\Bc}\colon \Ac \rightarrow \Ac$ is protoadditive;
\item 
\begin{itemize}
\item[(a)] for any normal epimorphism $f\colon A\to B$, the induced monomorphism $[K[f]]_{\Bc}\to A$ is normal; 
\item[(b)] the normal extensions are precisely the normal epimorphisms $f$ with $K[f] \in \Bc$;
\end{itemize}
\item  
\begin{itemize}
\item[(a)] for any normal epimorphism $f\colon A\to B$, the induced monomorphism $[K[f]]_{\Bc}\to A$ is normal;
\item[(b)] for any normal epimorphism $f \colon A \rightarrow B$, the reflection in $\NExt_{(\Bc,\Nc)}(\Ac)$ is given by the induced morphism
$\overline{f} \colon A/[K[f]]_{\Bc} \to B$.
\end{itemize}
\end{enumerate}
\end{proposition}
\begin{proof}
The equivalence $(1) \Leftrightarrow (2)$ follows from Proposition \ref{reflector=radical}, and the implication $(1) \Rightarrow (3a)$ from Lemma \ref{compositeisnormal }. For $(1)\Rightarrow (3b)$ it suffices to note that the proof of Proposition \ref{protocentral} remains valid.

To see that the implication $(3) \Rightarrow (4b)$ holds, consider a normal epimorphism $f\colon A\to B$. By the ``double quotient'' isomorphism theorem (see Theorem $4.3.10$ in \cite{BB}), the kernel of the induced morphism $\overline{f}\colon A/[K[f]]_{\Bc}\to B$ is $K[f]/[K[f]]_{\Bc}$, which lies in $\Bc$, hence $\overline{f}$ is a normal extension.

To see that $\overline{f}$ is the reflection of $f$ in $\NExt_{(\Bc,\Nc)}(\Ac)$, consider a normal extension $g\colon C \to D$ and a morphism $(a,b)\colon f\to g$ of normal epimorphisms. We need to show that there is a (unique) morphism $\overline{a}$ such that the diagram
\[
\xymatrix@=30pt{ 
  A \ar@{-<}`u[r]`[rr]^a[rr] \ar[r] \ar[d]_{f}  & \frac{A}{[K[f]]_{\Bc}} \ar[d]^{\overline{f}} \ar@{.>}[r]^{\overline{a}} & C \ar[d]^g \\
 B  \ar@{=}[r] & B \ar[r]_b & D.
}
\]
commutes. For this, it suffices to note that there is a commutative square
\[
\xymatrix{
[K[f]]_{\Bc} \ar[r] \ar[d]_{\ker(f)\circ \kappa_{K[f]}} & [K[g]]_{\Bc} \ar[d]^{\ker(g)\circ \kappa_{K[g]}} \\
A\ar[r]_a & C}
\]
and that  $[K[g]]_{\Bc}=0$ because $g$ is a normal extension, so that $a\circ\ker(f)\circ \kappa_{K[f]}=0$.

$(4) \Rightarrow (1)$
Consider a split short exact sequence 
\begin{equation}\label{split2}
\xymatrix{0 \ar[r]& K \ar[r]^k & A \ar@<-.8 ex> [r]_f & B \ar@<-.8ex>[l]_s \ar[r] &0 }
\end{equation}
in $\Ac$, and the induced diagram 
\[
\xymatrix@=35pt{& [K]_{\Bc} \ar[d]_{\kappa_{K[f]}}\ar[r]^-{[\ker(\pi_1)]_{\Bc}} & [R[f]]_{\Bc} \ar@<.8ex>[r]^{[\pi_1]_{\Bc}} \ar@<-.8ex>[r]_{[\pi_2]_{\Bc}}  \ar[d]_{\kappa_{R[f]}}&  [A]_{\Bc} \ar[d]^{\kappa_A} \ar@<-.8 ex> [r]_{[f]_{\Bc}} & [B]_{\Bc} \ar[d]^{\kappa_B} \ar@<-.8ex>[l]_{[s]_{\Bc}}  & \\
& K \ar[r]_-{\ker(\pi_1)} &R[f] \ar@<.8ex>[r]^{\pi_1} \ar@<-.8ex>[r]_{\pi_2} & A \ar@<-.8 ex> [r]_f & B \ar@<-.8ex>[l]_s  & }
\]
obtained by factorising $k$ through the kernel pair $R[f]$ of $f$, and applying the radical $[\cdot]_{\Bc}$. The assumption says that $[f]_{1,\Bc}= [K]_{\Bc}$ so that $[\ker(\pi_1)]_{\Bc}$ is the kernel of $[\pi_1]_{\Bc}$; it follows that $[\pi_2]_{\Bc} \circ [\ker(\pi_1)]_{\Bc}$ is the normalisation of the equivalence relation $[R[f]]_{\Bc}$ on $[A]$. Since the reflector $I \colon \Ac \rightarrow \Bc$ preserves kernel pairs of split epimorphisms by Lemma \ref{Marino}, one concludes that the functor $[\cdot]_{\Bc} \colon \Ac \rightarrow \Ac$ preserves the split short exact sequence \eqref{split2}. 
\end{proof}

\begin{remark}
Note that conditions $(3a)$ and $(4a)$ say that for any normal monomorphism $k\colon K\to A$ the composite $k\circ\kappa_K$ is a normal monomorphism as well, and we could equivalently have written ``any morphism" instead of ``any normal epimorphism" in the statement of these conditions. The reason we stated it the way we did is that, as we shall explain below, the proposition can be ``relativised"---in such a way that it depends on a choice of class $\Ec$ of morphisms in $\Ac$ satisfying Conditions \ref{extension}---and in the relative version, the morphism $[K[f]]_{\Bc}\to A$ might not be defined if $f$ is not in $\Ec$.
\end{remark}

\begin{remark}
Proposition \ref{characterisationbyextensions} shows, in particular, that it is meaningful to consider the conditions $(P)$ and $(N)$ from the previous sections beyond the context of torsion theories. 
\end{remark}

\begin{remark}\label{composingce}
Let $(\Ef,\Mf)$ be the reflective (pre)factorisation system associated with a Birkhoff subcategory $\Bc$ of a semi-abelian category $\Ac$, and let $\Ef'$ and $\Mf^*$ be the induced classes of morphisms ``stably in $\Ef$" and ``locally in $\Mf$", respectively, as considered in Section \ref{coveringmorphisms}. Then $(\Ef',\Mf^*)$ need not be a (pre)factorisation system in general. In fact, normal extensions fail to be stable under composition, even if the reflector $I\colon \Ac\to\Bc$ is protoadditive, in contrast to the normal extensions associated with a torsion theory satisfying condition $(P)$, which we discussed in Section \ref{coveringmorphisms}. For instance, let $\Ac=\Ab$ be the variety of abelian groups, and $\Bc=\mathsf{B}_2$ the Burnside variety of exponent $2$ ($\mathsf{B}_2$ consists of all abelian groups $A$ such that $a+a=0$ for every $a\in A$). Then the reflector $\Ab\to\mathsf{B}_2$ is additive, but the composite of two normal extensions need not be normal: if we denote by $C_n$ the cyclic group of order $n$, then the unique map $C_2\to 0$ is a normal extension, as is the only non-trivial morphism $C_4\to C_2$. However, the composite $C_4\to 0$ is not.

Note that the composite $g\circ f\colon A\to B\to C$ of two normal extensions is a normal extension as soon as $[g\circ f]_{1,\Bc}$ is $\Bc$-perfect, i.e. $I([g\circ f]_{1,\Bc})=0$. Indeed, if $q$ is the canonical normal epimorphism $A\to A/[g\circ f]_{1,\Bc}$, then the assumption that $I([g\circ f]_{1,\Bc})=0$ implies that $q$ lies in $\Ef$, since $I$ preserves cokernels. In fact, we have that $q$ lies in $\Ef'$, since pulling back yields isomorphic kernels, and preserves normal epimorphisms. From \cite{CJKP} we recall that $\Mf^*\subseteq (\Ef')^{\downarrow}$ which is easily seen to imply that also composites of morphisms in $\Mf^*$ lie in $(\Ef')^{\downarrow}$. In particular, we have that $q \downarrow (g\circ f)$, since, by assumption, we have that $f$ and $g$ lie in $\Mf^*$. As $g\circ f=I_1(g\circ f)\circ q$, it follows that $q$ is a split monomorphism, hence an isomorphism, and we can conclude that $g\circ f$ is a normal extension.
\end{remark}

Recall from \cite{Ev,EGV} that the notion of Birkhoff subcategory can be ``relativised" as follows. Let $\Ec$ be a class of morphisms in a semi-abelian category $\Ac$ satisfying Conditions \ref{extension}, and $\Bc$ a reflective subcategory of the category $\Ac_{\Ec}$. Denote by $I\colon \Ac_{\Ec}\to\Bc$ the reflector, by $H\colon \Bc\to \Ac_{\Ec}$ the inclusion functor, and write $\eta$ for the unit of the reflection.  Then $\Bc$ is called a \emph{strongly $\Ec$-Birkhoff subcategory} of $\Ac_{\Ec}$ if the square \eqref{unitsquare} is a double $\Ec$-extension for any $\Ec$-extension $f\colon A\to B$. This determines an admissible Galois structure $\Gamma_{(\Bc,\Ec)}=(\Ac_{\Ec},\Bc,I,H,\Ec)$ with respect to which the central and normal extensions coincide, just as in the ``absolute" case. The full subcategory of $\Ext_{\Ec}(\Ac)$ of all normal $\Ec$-extensions with respect to $\Gamma_{(\Bc,\Ec)}$---denoted $\NExt_{(\Bc,\Ec)}(\Ac)$---is reflective in $\Ext_{\Ec}(\Ac)$, and the construction of the reflector $I_1\colon \Ext_{\Ec}(\Ac)\to\NExt_{(\Bc,\Ec)}(\Ac)$ is formally the same as in the ``absolute" case. $\NExt_{(\Bc,\Ec)}(\Ac)$ is, in fact, a strongly $\Ec^1$-Birkhoff subcategory of $\Ext_{\Ec}(\Ac)=(\Arr(\Ac))_{\Ec^1}$, where $\Ec^1$ denotes, as before, the class of double $\Ec$-extensions. This fact allows us to define \emph{double normal $\Ec$-extensions} (with respect to $\Gamma_{(\Bc,\Ec)}$) as those double $\Ec$-extensions that are normal with respect to the Galois structure $\Gamma_{(\Bc_1,\Ec^1)}$, where $\Bc_1=\NExt_{(\Bc,\Ec)}(\Ac)$, and then to define \emph{three-fold normal $\Ec$-extensions}, and so on. For each $n\geq 1$, we use the notation  $\Gamma_{(\Bc,\Ec)}^{n}$ for the induced Galois structure $(\Ext_{\Ec}^n(\Ac), \Bc_n,I_n,H_n,\Ec^n)$, where 
\[
\Bc_n=\NExt_{(\Bc_{n-1},\Ec^{n-1})}(\Arr^{n-1}(\Ac))=\NExt^n_{(\Bc,\Ec)}(\Ac).
\]
Similar to the case $n=1$, for $n\geq 2$ and any $n$-fold $\Ec$-extension $A$, we write $\eta^n_A\colon A\to I_n(A)$ for the reflection unit, and $\kappa^n_A\colon [A]_{n,\Bc}\to A_{\textrm{top}}$ for the morphism in $\Ac$ which appears as the ``top" morphism in the diagram of the kernel $\ker(\eta^n_A)\colon K[\eta^n_A]\to A$. Note that we have that $\iota^n[A]_{n,\Bc}=K[\eta^n_A]$, where the functor $\iota^n\colon \Ac\to\Arr^n(\Ac)$ is as in Section \ref{sectionderived}.

We refer the reader to the articles \cite{Ev,EGV} for more details, and proofs of the statements above.

Replacing ``$\Ac$'' by ``$\Ac_{\Ec}$" and ``normal epimorphism" by ``$\Ec$-extension" in Lemma \ref{Marino} and Proposition \ref{characterisationbyextensions} provides us with relative versions of these results. One easily verifies that the proofs remain valid. We obtain, in particular, for each $n\geq 1$, a characterisation of the $n$-fold normal extensions with respect to $\Gamma_{(\Bc,\Nc)}$, if the reflector $I$ is protoadditive. Indeed, in this case also the $I_n$ are protoadditive. In fact, we have:

\begin{lemma}\label{centralisationisprotoadditive}
$I\colon \Ac_{\Ec}\to\Bc$ is protoadditive if and only if $I_1\colon \Ext_{\Ec}(\Ac)\to\NExt_{(\Bc,\Ec)}(\Ac)$ is protoadditive.
\end{lemma}

\begin{proof}
The ``only if" part of this lemma has already been considered in \cite{EG}: it essentially follows from the implications $(1)\Rightarrow (2)$ and $(1)\Rightarrow (4)$ in the ``relative version" of Proposition \ref{characterisationbyextensions}, and the $3\times 3$ lemma. Now, suppose that $I_1$ is protoadditive. Since, by the ``relative version" of Lemma \ref{Marino}, $I\colon \Ac_{\Ec}\to\Bc$ preserves, for $A\in\Ac_{\Ec}$, the product $A\times A$, it follows from the construction of $I_1$ that $I_1(A\to 0)=I(A)\to 0$. It is then immediate to conclude that also $I$ is protoadditive.
\end{proof}

We are now in a position to prove the following theorem. As before, we write $\Nc$ for the class of normal epimorphisms in $\Ac$. If $A$ is an $n$-fold $\Nc$-extension, the ``initial ribs'' in the diagram of $A$ are denoted $a_i$ ($1\leq i\leq n$), and its ``top vertex'' (the domain of the morphisms $a_i$) $A_{\textrm{top}}$.

\begin{theorem}\label{characterisationbyextensionshigher}
For a Birkhoff subcategory $\Bc$ of a semi-abelian category $\Ac$, the following conditions are equivalent:
\begin{enumerate}
\item
the reflector $I\colon \Ac\to \Bc$ is protoadditive;
\item 
the associated radical $[\cdot]_{\Bc}\colon \Ac \rightarrow \Ac$ is protoadditive;
\item 
the following conditions hold for any $n\geq 1$:
\begin{itemize}
\item[(a)]
the canonical monomorphism $[\bigcap_{1\leq i\leq n}K[a_i]]_{\Bc}\to A_{\textrm{top}}$ is normal for any $n$-fold $\Nc$-extension $A$;
\item[(b)]
the $n$-fold normal extensions are precisely the $n$-fold $\Nc$-extensions $A$ with $\bigcap_{1\leq i\leq n}K[a_i]\in\Bc$;
\end{itemize}
\item
the following conditions hold for any $n\geq 1$:
\begin{itemize}
\item[(a)]
the canonical monomorphism $[\bigcap_{1\leq i\leq n}K[a_i]]_{\Bc}\to A_{\textrm{top}}$ is normal for any $n$-fold $\Nc$-extension $A$;
\item[(b)] for any $n$-fold $\Nc$-extension $A$, the reflection in $\NExt^n_{(\Bc,\Nc)}(\Ac)$ is given by the quotient
$A/\iota^n[\bigcap_{1\leq i\leq n}K[a_i]]_{\Bc}$;
\end{itemize}
\item
either (3) or (4) holds for some $n\geq 1$.
\end{enumerate}
\end{theorem}
\begin{proof}
$(1) \Leftrightarrow (2)$ was proved in Proposition \ref{characterisationbyextensions}. 

To see that $(5)$ implies $(1)$, we note that $I_k$ preserves binary products, for any $k\geq 0$, by the ``relative version" of Lemma \ref{Marino}. Taking this into account, we see from the construction of $I_n$ that $I_n(\iota^nA)=\iota_nI(A)$, for any $n\geq 1$ and any $n$-fold $\Nc$-extension $A$, so that the validity of conditions $(a)$ and $(b)$ for some $n\geq 1$ implies that for $n=1$. Proposition \ref{characterisationbyextensions} then implies that $I$ is protoadditive.

The other implications follow easily by induction on $n$, using Proposition \ref{characterisationbyextensions} and its ``relative version", and Lemma \ref{centralisationisprotoadditive}.
\end{proof}

\section{Composition of Birkhoff and protoadditive reflections}
We have seen in Section \ref{sectionderived} that a torsion theory $(\Tc,\Fc)$ on a homological category $\Ac$ satisfying conditions $(P)$ and $(N)$ induces a chain of torsion theories $(\Tc_n,\Fc_n)$ on the categories $\Ext^n_{\Nc}(\Ac)$ of $n$-fold $\Nc$-extensions such that, for each $n\geq 1$, the torsion-free subcategory $\Fc_n$ consists of all $n$-fold $\Nc$-extensions that are normal extensions with respect to the Galois structure $\Gamma_{\Fc_{n-1}}$ associated with the torsion theory $(\Tc_{n-1},\Fc_{n-1})$. Moreover, an $n$-fold $\Nc$-extension $A$ with ``initial ribs" $a_i$ ($1\leq i\leq n$) is normal if and only if the intersection $\bigcap_{1\leq i\leq n}K[a_i]$ lies in $\Fc$. 

Similarly, a Birkhoff subcategory $\Bc$ of a semi-abelian category $\Ac$ induces a chain of ``strongly $\Nc^{n}$-Birkhoff subcategories" $\Bc_n$ of the categories $\Ext^n_{\Nc}(\Ac)$, where, for each $n\geq 1$, $\Nc^n$ denotes the class of all $(n+1)$-fold $\Nc$-extensions. Moreover, in the case where the reflector $I\colon \Ac\to\Bc$ is protoadditive, the $n$-fold normal (=central) extensions admit the same simple description as in the example of a torsion theory satisfying $(P)$ and $(N)$, as we have explained in Section \ref{Birkhoffsection}. 
 
In general, it is not always easy to characterise the $n$-fold normal extensions (for $n\geq 1$) with respect to a particular Birkhoff subcategory. However, we are going to show that the problem can sometimes be simplified by decomposing the considered adjunction into a pair of  adjunctions such that one of the reflectors is protoadditive. We shall explain this in the present section. In fact, we shall consider, more generally, composite adjunctions
  \begin{equation}\label{compositeadj}
 \xymatrix@=30pt{
{\Ac \, } \ar@<1ex>[r]_-{^{\perp}}^-{I} & {\, \Bc \, }
\ar@<1ex>[l]^H  \ar@<1ex>[r]_-{^{\perp}}^-{J} & \Cc \ar@<1ex>[l]^G   }
 \end{equation}
where $\Ac$ is semi-abelian, $\Bc$ a Birkhoff subcategory of $\Ac$, and where $\Cc$ can be \emph{any} (normal epi)-reflective subcategory of $\Bc$, admissible with respect to $\Nc$, with a protoadditive reflector (but not necessarily Birkhoff). As we shall see, such a situation induces a chain of Galois structures of higher normal extensions such that, for $n\geq 1$, an $n$-fold $\Nc$-extension $A$ in $\Ac$ is normal with respect to $\Gamma_{(\Cc,\Nc)}$ if and only if it is normal with respect to $\Gamma_{(\Bc,\Nc)}$ and the intersection $\bigcap_{1\leq i\leq n}K[a_i]$ lies in $\Cc$. Here we have written, as before, $a_i$ for the ``initial ribs" of $A$. 

First, we consider the one-dimensional case (see also \cite{DEG}):

\begin{proposition}\label{composite}
Consider the composite reflection \eqref{compositeadj} where $\Ac$ is a semi-abelian category, $\Bc$ a Birkhoff subcategory of $\Ac$ and $\Cc$ a (normal epi)-reflective subcategory of $\Bc$, admissible with respect to normal epimorphisms, with protoadditive reflector $J$. Then the composite reflector $J\circ I$ is admissible with respect to normal epimorphisms and, for any normal epimorphism $f \colon A \rightarrow B$ in $\Ac$, the following conditions are equivalent:
\begin{enumerate}
\item $f \colon A \rightarrow B$ is a normal extension with respect to $\Gamma_{(\Cc,\Nc)}$;
\item $f \colon A \rightarrow B$ is a central extension with respect to $\Gamma_{(\Cc,\Nc)}$;
\item $K[f] \in \Cc $ and $f \colon A \rightarrow B$ is a $\Gamma_{(\Bc,\Nc)}$-normal extension.
\end{enumerate}
\end{proposition}
\begin{proof}
The admissibility of $J\circ I$ is clear, while the implication $(1) \Rightarrow (2)$ holds by definition.

$(2) \Rightarrow (3)$ 
Let $p:E \rightarrow B$ be an normal epimorphism such that $p^*(f)$ is $\Gamma_{(\Cc,\Nc)}$-trivial. Then in the following commutative diagram the composite of the left hand pointing squares is a pullback (here $\eta$ and $\mu$ are the reflection units): 
\[
\xymatrix@=35pt{
JI(E\times_BA) \ar[d]_{JI(p^*(f))} & I(E\times_BA) \ar[d]_{I(p^*(f))} \ar[l]_-{\mu_{I(E\times_BA)}} & E\times_B A \ar@{}[rd]|<<{\pullback}\ar[d]_{p^*(f)} \ar[l]_-{\eta_{E\times_BA}} \ar[r] & A \ar[d]^f\\
JI(E) & I(E) \ar[l]^{\mu_{I(E)}} & E \ar[l]^{\eta_E} \ar[r]_p & B}
\]
This implies, on the one hand, that  $p^*(f)$ and $\eta_{E\times_BA}$ are jointly monomorphic, and, consequently, that the middle square is a pullback, since it is a double $\Nc$-extension, because $\Bc$ is a Birkhoff subcategory of $\Ac$. Hence, $f$ is a $\Gamma_{(\Bc,\Nc)}$-central extension, and we know that, with respect to a Birkhoff subcategory, the central and normal extensions coincide \cite{JK}. On the other hand, since also the right hand square is a pullback, we have isomorphisms
\[
K[JI(p^*(f))] \cong K[p^*(f)] \cong K[f],
\]
so that $K[f]\in\Cc$, since $\Cc$, being a reflective subcategory,  is closed under limits in $\Ac$. 

$(3) \Rightarrow (1)$ Now let $f \colon A \rightarrow B$ be a normal epimorphism satisfying $(3)$. Consider the commutative diagram
$$
\xymatrix@=40pt{R[f] \ar[d]_{\pi_1} \ar[r]^-{\eta_{R[f]}} & I(R[f]) \ar[d]^{I(\pi_1)}  \ar[r]^-{\mu_{I(R[f])}} & JI (R[f]) \ar[d]^{JI(\pi_1)}  \\
A \ar[r]_{\eta_A}& I(A) \ar[r]_{\mu_{I(A)}} &JI(A)}
$$
where $\pi_1$ is the first projection of the kernel pair of $f$. By assumption, its left hand square is a pullback. Consequently, there is an isomorphism $K[\pi_1] \cong  K[ I(\pi_1)]$, so that $K[I(\pi_1)]$ lies in $\Cc$  because $K[\pi_1]\cong K[f]$ lies in $\Cc$, by assumption. Since the reflector $J$ is protoadditive and the category $\Ac$ is protomodular, this implies that also the right hand square is a pullback.
\end{proof}

We continue with a higher dimensional version of this result. For this, let us first of all remark that Proposition \ref{composite} can be ``relativised" with respect to a class $\Ec$ of morphisms in the semi-abelian category $\Ac$ satisfying Conditions \ref{extension}. More precisely, if we have a composite adjunction
  \begin{equation}\label{relativecompositeadj}
 \xymatrix@=30pt{
{\Ac_{\Ec} \, } \ar@<1ex>[r]_-{^{\perp}}^-{I} & {\, \Bc \, }
\ar@<1ex>[l]^H  \ar@<1ex>[r]_-{^{\perp}}^-{J} & \Cc \ar@<1ex>[l]^G   }
 \end{equation}
 with $\Bc$ a strongly $\Ec$-Birkhoff subcategory of $\Ac_{\Ec}$ and $\Cc$ a full $\Ec$-reflective subcategory of $\Bc$, admissible with respect to $\Ec$, with protoadditive reflector $J$, then an $\Ec$-extension $f\colon A\to B$ is normal with respect to $\Gamma_{(\Cc,\Ec)}=(\Ac_{\Ec},\Cc,J\circ I,H\circ G,\Ec)$ if and only if it is a $\Gamma_{(\Cc,\Ec)}$-central extension if and only if $K[f]\in\Cc$ and $f$ is a $\Gamma_{(\Bc,\Ec)}$-normal extension. We leave it to the reader to verify that the proof of Proposition~\ref{composite} remains valid under our assumptions. 

Now let us consider a composite reflection \eqref{relativecompositeadj} satisfying the conditions above. Write $\eta$ and $\mu$ for the units of the reflections $I$ and $J$, respectively, and $[-]_{\Cc}\colon \Ac_{\Ec}\to\Ac_{\Ec}$ for the radical induced by the $\Ec$-reflection $J\circ I$. We have the following property:

\begin{lemma}\label{Mathieu}
  For any $\Gamma_{(\Bc,\Ec)}$-normal extension $f\colon A\to B$, the monomorphism $\ker(f)\circ \ker(\mu\circ\eta)_{K[f]}\colon [K[f]]_{\Cc}\to A$ is normal.
\end{lemma}
\begin{proof}
First of all note that the radical  $[-]_{\Cc}\colon \Ac_{\Ec}\to\Ac_{\Ec}$ is well-defined since, for any object $A$ of $\Ac_{\Ec}$, the unit $(\mu\circ \eta)_A\colon A\to JI(A)$ is an $\Ec$-extension, so that its kernel lies indeed in $\Ac_{\Ec}$.  

Now consider the short exact sequence 
\[
\xymatrix{
0\ar[r] & K[f] \ar[r]^{\ker(\pi_1)} & R[f] \ar[r]^{\pi_1} & A \ar[r] & 0}
\]
where $\pi_1$ denotes the first projection of the kernel pair of $f$. It is preserved by $I$ since $\pi_1$ is a trivial extension and $\eta_{K[f]}\colon K[f]\to I(K[f])$ an isomorphism. Hence, it is preserved by $J\circ I$ since $J$ is protoadditive. In particular, we have that $JI(\ker(\pi_1))$ is a monomorphism, so that the left hand square in the morphism 
\[
\xymatrix{
0 \ar[r] & [K[f]]_{\Cc} \ar[d] \ar[r] & K[f] \ar[d]^{\ker(\pi_1)} \ar[r] & JI(K[f]) \ar[d]^{JI(\ker(\pi_1))} \ar[r] & 0\\
0 \ar[r] & [R[f]]_{\Cc} \ar[r] & R[f] \ar[r] & JI(R[f]) \ar[r] & 0}
\]
of short exact sequences is a pullback. It follows that the monomorphism $\ker(\pi_1)\circ \ker((\mu\circ\eta)_{K[f]})$ is normal. Hence, so is its normal image along the second projection $\pi_2\colon R[f]\to A$ of the kernel pair of $f$, and this is exactly the monomorphism $\ker(f)\circ \ker(\mu\circ\eta)_{K[f]}\colon [K[f]]_{\Cc}\to A$. 
\end{proof}

The above lemma, together with the ``relative'' version of  Proposition \ref{composite}, now allows us to prove that the pair of reflections \eqref{relativecompositeadj} induces a pair of reflections ``at the level of extensions'', in the following sense:

\begin{lemma}\label{doublecentralisation}
The pair of reflections \eqref{relativecompositeadj} induces new reflections
\[
\xymatrix{
\Ext_{\Ec}(\Ac) \ar@<1 ex>[r]^-{I_1} \ar@{}[r]|-{\perp} & \NExt_{(\Bc,\Ec)}(\Ac) \ar@{}[r]|{\perp}  \ar@<1 ex>[l] \ar@<1 ex>[r]^{J_1} & \NExt_{(\Cc,\Ec)}(\Ac) \ar@<1 ex>[l]}
\]
where $\NExt_{(\Bc,\Ec)}(\Ac)$ is a strongly $\Ec^1$-Birkhoff subcategory of $\Ext_{\Ec}(\Ac)$ with reflector $I_1$ and $\NExt_{(\Cc,\Ec)}(\Ac)$ is an $\Ec^1$-reflective subcategory of  $\NExt_{(\Cc,\Ec)}(\Ac)$, admissible with respect to the class of $\Ec^1$-extensions in $\NExt_{(\Bc,\Ec)}(\Ac)$, with protoadditive reflector $J_1$. We have that $J_1$ sends a $\Gamma_{(\Bc,\Ec)}$-normal extension $f\colon A\to B$ to the induced $\Gamma_{(\Cc,\Ec)}$-normal extension $J_1(f)\colon A/[K[f]]_{\Cc}\to B$.
 \end{lemma}
\begin{proof}
We already know that $\NExt_{(\Bc,\Ec)}(\Ac)$ is a strongly $\Ec^1$-Birkhoff subcategory of $\Ext_{\Ec}(\Ac)$. Let us then prove, for any $\Gamma_{(\Bc,\Ec)}$-normal extension $f\colon A\to B$, that the induced $\Ec$-extension $J_1(f)\colon A/[K[f]]_{\Cc}\to B$ is indeed its reflection in $\NExt_{(\Cc,\Ec)}(\Ac)$. (Note that the monomorphism $[K[f]]_{\Cc}\to A$ is normal, by Lemma \ref{Mathieu}.)    

On the one hand we have that $J_1(f)$ is a $\Gamma_{(\Cc,\Ec)}$-normal extension since $K[J_1(f)]=K[f]/[K[f]_{\Cc}=J(K[f])$ by the ``double quotient'' isomorphism theorem, and because the reflector $F\colon \Bc\to\Cc$ is protoadditive, by assumption. On the other hand, if $g\colon C\to D$ is a $\Gamma_{(\Cc,\Ec)}$-normal extension as well, and $(a,b)\colon f\to g$ is a morphism of $\Ec$-extensions, there exists a (unique) morphism $\overline{a}$ such that the diagram
 $$\xymatrix@1@=30pt{ 
  A \ar@{->}`u[r]`[rr]^a[rr] \ar[r] \ar[d]_{f}  & \frac{A}{[K[f]]_{\Cc}} \ar[d]^{J_1(f)} \ar@{.>}[r]^{\overline{a}} & C \ar[d]^g \\
 B  \ar@{=}[r] & B \ar[r]_b & D
}
$$
commutes. Indeed, it suffices to note that there is a commutative square
\[
\xymatrix{
[K[f]]_{\Cc} \ar[r] \ar[d] & [K[g]]_{\Cc} \ar[d] \\
A\ar[r]_a & C}
\]
and that  $[K[g]]_{\Cc}=0$. It follows that $\NExt_{(\Cc,\Ec)}(\Ac)$ is a reflective subcategory of $\NExt_{(\Bc,\Ec)}(\Ac)$ with reflector $J_1$. Since, for any $\Cc$-normal extension $f\colon A\to B$, the reflection unit
\[
\xymatrix{
A \ar[r] \ar[d]_f & \frac{A}{[K[f]]_{\Cc}} \ar[d]\\
B \ar@{=}[r] & B}
\]
is clearly a double $\Ec$-extension, $\NExt_{(\Cc,\Ec)}(\Ac)$ is an $\Ec^1$-reflective subcategory of $\NExt_{(\Bc,\Ec)}(\Ac)$.

Next we prove that the reflector $J_1$ is protoadditive. To this end we consider a split short exact sequence
\[
\xymatrix{
0 \ar[r] & K_1 \ar[r] \ar[d]_{k} & A_1 \ar@<-.8 ex> [r] \ar[d]^a & B_1 \ar[d]^b  \ar[r]  \ar@<-.8ex>[l] & 0\\
0 \ar[r] & K_0 \ar[r] & A_0 \ar@<-.8 ex> [r] & B_0 \ar[r]  \ar@<-.8ex>[l] & 0}
\]
in $\NExt_{(\Bc,\Ec)}(\Ac)$ and we note that both rows are split short exact sequences in $\Ac$. By taking kernels vertically and then applying the radical $[-]_{\Cc}\colon \Ac_{\Ec}\to\Ac_{\Ec}$, whose restriction to $\Bc\to\Bc$ is protoadditive by Proposition \ref{reflector=radical}, we obtain a split short exact sequence which is the first row in the diagram
\[
\xymatrix{
0 \ar[r] & [K[k]]_{\Cc} \ar[r] \ar[d] & [K[a]]_{\Cc} \ar[d] \ar@<-.8 ex> [r] & [K[b]]_{\Cc} \ar[d]  \ar@<-.8ex>[l] \ar[r] & 0\\
0 \ar[r] & K_1 \ar[r] \ar[d] & A_1 \ar[d] \ar@<-.8 ex> [r] & B_1 \ar[d]  \ar@<-.8ex>[l] \ar[r] & 0\\
0 \ar[r] & \frac{K_1}{[K[k]]_{\Cc}} \ar[r]  & \frac{A_1}{[K[a]]_{\Cc}}  \ar@<-.8 ex> [r] &  \frac{B_1}{[K[b]]_{\Cc}}   \ar@<-.8ex>[l] \ar[r] & 0}
\]
Since also the second row is split exact, by assumption, the third row is a split short exact sequence as well, by the $3\times 3$ lemma. If follows that the reflector $J_1$ is protoadditive, and this completes the proof. 

Finally, we prove that the reflector $J_1$ is admissible. For this, we consider a pullback
\[
\xymatrix{
& D \ar@{}[rrdd]|<<{\pullback} \ar@{.>}[dd]  \ar[rr] && B \ar@{=}[dd] \\
P \ar@{}[rrdd]|<<{\pullback}   \ar[ur]^p \ar[rr] \ar[dd] && A \ar[ur]_{f} \ar[dd] & \\
& D \ar@{.>}[rr] && B \\
C \ar[ur]^{g}  \ar[rr] && \frac{A}{[K[f]]_{\Cc}} \ar[ur]_{J_1(f)} &}
\]
 in $\NExt_{(\Bc,\Ec)}(\Ac)$ of a reflection unit $f\to J_1(f)$ along some double $\Ec$-extension $g\to J_1(f)$, and we assume that $g\in \NExt_{(\Cc,\Ec)}(\Ac)$. Notice that it is a pointwise pullback in $\Ac$. We have to prove that its image in $\NExt_{(\Cc,\Ec)}(\Ac)$ by $J_1$ is still a pullback. Since the reflection unit $f\to J_1(f)$ is sent to an isomorphism, this amounts to proving that $J_1(p)\to J_1(g)$ is an isomorphism as well. For this it suffices to show that the morphism $P/[K[p]]_{\Cc}\to C/[K[g]]_{\Cc}$ is an isomorphism.

Now, by taking kernels in the cube above, we obtain a pullback 
\[
\xymatrix{
K[p] \ar[r] \ar@{}[rd]|<{\pullback} \ar[d] & K[f] \ar[d] \\
K[g] \ar[r] & J(K[f])}
\]
in $\Bc$. Note that the object in the right hand lower corner is indeed $J(K[f])=K[f]/[K[f]]_{\Cc}$ by the ``double quotient'' isomorphism theorem, and that $K[g]\in \Cc$ because $g\in \NExt_{(\Cc,\Ec)}(\Ac)$, by assumption. Moreover, we have that $K[g]\to J(K[f])$ is an $\Ec$-extension, by Condition \ref{extension}.5 for the class $\Ec^1$. Consequently, using the admissibility of $J$, we find that the image by $J$ of the above square is a pullback in $\Cc$, which implies that $J(K[p])\to J(K[g])$ is an isomorphism. It follows that in the diagram
\[
\xymatrix{
0 \ar[r] & [K[p]]_{\Cc} \ar[r] \ar[d] & K[p] \ar[r] \ar[d] & J(K[p]) \ar@{=}[d] \ar[r] & 0\\
0 \ar[r] & [K[g]]_{\Cc} \ar[r] & K[g] \ar[r] & J(K[g]) \ar[r] & 0}
\]
of short exact sequences in $\Bc$, the left hand square is a pullback. Since, of course, also
\[
\xymatrix{
 K[p] \ar[r] \ar[d] & P \ar[d] \\
 K[g] \ar[r] & C }
\]
is a pullback, we have that the left hand square in the diagram
\[
\xymatrix{
0 \ar[r] & [K[p]]_{\Cc} \ar[r] \ar[d] & P \ar[r] \ar[d] & \frac{P}{[K[p]]_{\Cc}} \ar[d] \ar[r] & 0\\
0 \ar[r] & [K[g]]_{\Cc} \ar[r] & C \ar[r] & \frac{C}{[K[g]]_{\Cc}} \ar[r] & 0}
\]
 of short exact sequences in $\Ac$ is a pullback, and this implies that the $\Ec$-extension $P/[K[p]]_{\Cc}\to C/[K[g]]_{\Cc}$ is a monomorphism (in $\Ac$), hence an isomorphism.
\end{proof}

Thanks to this lemma, we can repeatedly apply the ``relative" version of Proposition \ref{composite}, and we obtain: 
\begin{theorem}\label{highercomposite}
Let $\Ac$ be a semi-abelian category, $\Bc$ a Birkhoff subcategory of $\Ac$, and $\Cc$ a (normal epi)-reflective subcategory, admissible with respect to normal epimorphisms, such that the reflector $J\colon \Bc\to\Cc$ is protoadditive. Then, for any $n\geq 1$ and any $n$-fold $\Nc$-extension $A$ in $\Ac$, the following 
conditions are equivalent:
\begin{enumerate}
\item
$A$ is an $n$-fold normal extension with respect to $\Gamma_{(\Cc,\Nc)}$;
\item
$A$ is an $n$-fold central extension with respect to $\Gamma_{(\Cc,\Nc)}$;
\item
$\bigcap_{1\leq i\leq n}K[a_i]\in \Cc$ and $A$ is an $n$-fold normal extension with respect to~$\Gamma_{(\Bc,\Nc)}$.
\end{enumerate}
\end{theorem}

We are mainly interested in the situation where, in the composite reflection \eqref{compositeadj}, $\Cc$ is a Birkhoff subcategory of $\Bc$, since in this case the construction of the composite reflectors $J_n\circ I_n$ ($n\geq 1$) obtained from Lemma \ref{doublecentralisation}  can be simplified, as we shall see below. However, the case of a torsion-free $\Cc$ is of interest as well: 

\begin{proposition}\label{compositetorsion}
Let $\Ac$ be a semi-abelian category, $\Bc$ a Birkhoff subcategory of $\Ac$, and $\Cc$ a (normal epi)-reflective subcategory of $\Bc$ whose reflector $F\colon \Bc\to\Cc$ is protoadditive. If $\Cc$ is a torsion-free subcategory of $\Bc$, then, for each $n\geq 1$,  $\NExt^n_{(\Cc,\Nc)}(\Ac)$ is a torsion-free subcategory of $\NExt_{(\Bc,\Nc)}^n(\Ac)$.
\end{proposition}
\begin{proof}
It suffices to prove, for any composite adjunction \eqref{relativecompositeadj} with $\Ac$ semi-abelian, $\Ec$ a class of morphisms in $\Ac$ satisfying Conditions \ref{extension}, $\Bc$ a strongly $\Ec$-Birkhoff subcategory of $\Ac$, and $\Cc$ an $\Ec$-reflective subcategory of $\Bc$ which is torsion-free and whose reflector is protoadditive, that $\NExt_{(\Cc,\Ec)}(\Ac)$ is a torsion-free subcategory of $\NExt_{(\Bc,\Ec)}(\Ac)$.  But this follows easily from the construction of the reflector $J_1\colon \NExt_{(\Bc,\Ec)}(\Ac)\to \NExt_{(\Cc,\Ec)}(\Ac)$ given in Lemma \ref{doublecentralisation}, which shows us that the associated radical $[-]_{1,\Cc}\colon \NExt_{(\Bc,\Ec)}(\Ac)\to  \NExt_{(\Bc,\Ec)}(\Ac)$ sends a $\Gamma_{(\Bc,\Ec)}$-normal extension $f\colon A\to B$ to the unique morphism $[K[f]]_{\Cc}\to 0$, so that we clearly have that $J_1$ is idempotent. By Theorem \ref{torsiontheorem} the proof is then complete.
\end{proof}

From now on, let us assume that the category $\Cc$ in the composite adjunction \eqref{compositeadj} is a Birkhoff subcategory of $\Bc$.  Since double $\Nc$-extensions are stable under composition, $\Cc$ is then also a Birkhoff subcategory of $\Ac$. Theorem \ref{highercomposite} gives us a characterisation of the higher normal extensions with respect to $\Gamma_{(\Cc,\Nc)}$, and Lemma~\ref{doublecentralisation} shows us how the functors $(J\circ I)_n$ are constructed. Using the following lemma, we shall be able to simplify this construction, by giving a description of the functors $[-]_{n,\Cc}$ in terms of $[-]_{n,\Bc}$ and $[-]_{\Cc}$ ($n\geq1$). 

Note that, in a semi-abelian category $\Ac$, any two normal subobjects $M\to A$ and $N\to A$ admit a supremum (in the lattice of normal subobjects of $A$) which can be obtained as the kernel of the ``diagonal" $A\to P\cong A/(M\vee N)$ in the pushout diagram
 \[
 \xymatrix{
 A 
\ar[r] \ar[d] & A/N \ar[d]_>>{\pushout}
\\
 A/M \ar[r] & P.}
 \]
 Any subobject $S\to A$ admits a \emph{normal closure} $\overline{S}^A\to A$ obtained as the kernel of the cokernel of $S\to A$.
 
 \begin{lemma}
Let $\Ec$ be a class of morphisms in a semi-abelian category $\Ac$ satisfying Conditions \ref{extension}, and $\Bc$ and $\Cc$ strongly $\Ec$-Birkhoff subcategories of $\Ac_{\Ec}$ such that $\Cc\subseteq \Bc$. If the comparison reflector $\Bc\to\Cc$ is protoadditive, then we have, for any $\Ec$-extension $f\colon A\to B$, the identity 
 \[
[f]_{1,\Cc}=[f]_{1,\Bc}\vee \overline{[K[f]]}^{A}_{\Cc}.
\]
\end{lemma}
\begin{proof}
We  know from the proof of Lemma \ref{doublecentralisation} that the reflection in $\NExt_{(\Cc,\Ec)}(\Ac)$ of an $\Ec$-extension $f\colon A\to B$ is given by the induced $\Ec$-extension
\begin{equation}\label{complexreflection}
J_1\circ I_1(f)=J_1(A/[f]_{1,\Bc}\to B)=\frac{A/[f]_{1,\Bc}}{[K[f]/[f]_{1,\Bc}]_{\Cc}}\to B
\end{equation}
(Notice that $K[f]/[f]_{1,\Bc}$ is indeed the kernel of $A/[f]_{1,\Bc}\to B$ by the ``double quotient" isomorphism theorem.) Now consider the following commutative diagram:
\[
\xymatrix{
[K[f]]_{\Cc}  \ar[r] \ar[rd] & \overline{[K[f]]}^{A}_{\Cc} \ar[r] \ar[d] & A 
 \ar[r] \ar[d] & A/\overline{[K[f]]}^A_{\Cc} \ar[d]_>>{\pushout}\\
&[K[f]/[f]_{1,\Bc}]_{\Cc}\ar[r] & A/[f]_{1,\Bc} \ar[r] & A/([f]_{1,\Bc}\vee \overline{[K[f]]}^A_{\Cc})}
\]
Since the canonical morphism $K[f]\to K[f]/[f]_{1,\Bc}$ is an $\Ec$-extension, and $\Cc$ is a strongly $\Ec$-Birkhoff subcategory of $\Ac_{\Ec}$, we have that the skew morphism in the diagram is an $\Ec$-extension as well, which implies that $\overline{[K[f]]}^{A}_{\Cc}\to [K[f]/[f]_{1,\Bc}]_{\Cc}$ is an epimorphism. Since the right hand square is a pushout, this shows us that, in the bottom row, the right hand morphism is the cokernel of the left hand one, which is a normal monomorphism by Lemma \ref{Mathieu}. Together with \eqref{complexreflection}, this yields the needed identity.
\end{proof}

By repeatedly applying the previous lemma, we find, for any $n$-fold $\Nc$-extension $A$ with ``top'' object $A_{\textrm{top}}$  and ``initial ribs'' $a_i$ ($1\leq i\leq n$), that 
\begin{eqnarray*}
[A]_{n,\Cc} & = & [A]_{n,\Bc}\vee \overline{[K[A]]}^{A_{\textrm{top}}}_{n-1,\Cc}\\
&=&  [A]_{n,\Bc}\vee \overline{[K[A]]}^{A_{\textrm{top}}}_{n-1,\Bc}\vee \overline{[K[K[A]]]}^{A_{\textrm{top}}}_{n-2,\Cc}\\
&=& [A]_{n,\Bc}\vee \overline{[K[K[A]]]}^{A_{\textrm{top}}}_{n-2,\Cc}\\
&=& \cdots\\
&=& [A]_{n,\Bc} \vee \overline{\big[\bigcap_{1\leq i\leq n}K[a_i]\big]_{\Cc}}^{A_{\textrm{top}}}
\end{eqnarray*}
Here we used the fact that taking joins commutes with taking normal closures, and that $[K[A]]_{n-1,\Bc}\subseteq [A]_{n,\Bc}$ (as well as $[K[K[A]]]_{n-2,\Bc}\subseteq [K[A]]_{n-1,\Bc}$ , and so on), which follows easily, for instance from the previous lemma, by taking the two reflective subcategories $\Bc$ and $\Cc$ to be the same.

Thus we have proved the following theorem:
\begin{theorem}\label{compositecommutator}
Consider the composite adjunction \eqref{compositeadj} where $\Ac$ is a semi-abelian category, $\Bc$ a Birkhoff subcategory of $\Ac$ and $\Cc$ a Birkhoff subcategory of $\Bc$ (hence, also of $\Ac$) such that the reflector $J\colon \Bc\to\Cc$ is protoadditive.  Then, for any $n\geq 1$ and $n$-fold $\Nc$-extension $A$ in $\Ac$ with ``initial ribs $a_i$" ($1\leq i\leq n$), we have the identity
\[
[A]_{n,\Cc}= [A]_{n,\Bc} \vee \overline{\big[\bigcap_{1\leq i\leq n}K[a_i]\big]_{\Cc}}^{A_{\textrm{top}}}
\]
\end{theorem}

The formula given by the above theorem can be further simplified in the following situation.  Let $\Bc$ and $\Bc'$ be Birkhoff subcategories of a semi-abelian category $\Ac$ such that either of the reflectors  is protoadditive---say the reflector $I'\colon \Ac\to \Bc'$. In this case, the restriction of this reflector to a functor $\Bc\to \Bc\cap\Bc'$ is protoadditive as well, so that if we put $\Cc=\Bc\cap\Bc'$, we are indeed in the situation of Theorem \ref{compositecommutator}. We will obtain a simplified description of the functors $[-]_{n,\Bc\cap\Bc'}\colon \Ext^n_{\Nc}(\Ac)\to\Ext^n_{\Nc}(\Ac)$ ($n\geq 0$) from the next lemma together with Theorem \ref{characterisationbyextensionshigher}. The latter tells us that  $[A]_{n,\Bc'}=[\bigcap_{1\leq i\leq n}K[a_1]]_{\Bc'}$ for any $n\geq 1$ and any $n$-fold $\Nc$-extension $A$ in $\Ac$ with ``initial ribs $a_i$" ($1\leq i\leq n$).

\begin{lemma}\label{intersectionlemma}
Let $\Ec$ be a class of morphisms in a semi-abelian category $\Ac$ satisfying Conditions \ref{extension}, and $\Bc$ and $\Bc'$ strongly $\Ec$-Birkhoff subcategories of $\Ac_{\Ec}$. Then 
\[
\NExt_{(\Bc\cap\Bc',\Ec)}(\Ac)=\NExt_{(\Bc,\Ec)}(\Ac)\cap\NExt_{(\Bc',\Ec)}(\Ac)
\]
and 
\[
[A]_{\Bc\cap\Bc'}=[A]_{n,\Bc}\vee [A]_{n,\Bc'}.
\]
\end{lemma}
\begin{proof}
For any $\Ec$-extension $f\colon A\to B$, consider the following commutative cube, where $\pi_1$ is the first projection of the kernel pair of $f$: 
\[
\xymatrix@=20pt{
& I'(R[f]) \ar@{}[rrdd] \ar@{.>}[dd]  \ar[rr] && I'(I(R[f])) \ar[dd] \\
R[f]    \ar[ur] \ar[rr] \ar[dd]_{\pi_1} && I(R[f]) \ar[ur] \ar[dd] & \\
& I'(A) \ar@{.>}[rr] && I'(I(A)) \\
A \ar[ur]  \ar[rr] && I(A) \ar[ur] &}
\]
When $f\in\NExt_{(\Bc\cap\Bc',\Ec)}(\Ac)$, the composite of the front and the right side squares is a pullback. Since the front square is a double $\Ec$-extension by the strong $\Ec$-Birkhoff property of $\Bc$, this implies that it is a pullback, hence $f\in \NExt_{(\Bc,\Ec)}(\Ac)$. Similarly, one shows that $f\in \NExt_{({\Bc'},\Ec)}(\Ac)$.  

Conversely, if $f$ is both $\Gamma_{(\Bc,\Ec)}$-normal and $\Gamma_{({\Bc'},{\Ec})}$-normal, then the left hand and the front squares are pullbacks, and then also the right hand and back ones, since both $I$ and $I'$ preserve pullbacks of split epimorphisms along $\Ec$-extensions by the ``relative version'' of Lemma \ref{Marino}. Hence,  $f\in\NExt_{(\Bc\cap\Bc',\Ec)}(\Ac)$. 

The second part of the statement follows from the fact that, for any $A$ in $\Ac_{\Ec}$, the following square is a pushout in $\Ac$, since $\Bc$ and $\Bc'$ are Birkhoff subcategories of $\Ac_{\Ec}$, 
\[
\xymatrix{
A \ar@{}[rd]|>>{\pushout}\ar[r] \ar[d] & I'(A) \ar[d]\\
I(A) \ar[r] & I'(I(A)),}
\]
which implies that  $I'(I(A))=A/([A]_{\Bc}\vee [A]_{\Bc'})$.
\end{proof}

\begin{theorem}\label{compositeintersection}
Let $\Bc$ and $\Bc'$ be Birkhoff subcategories of a semi-abelian category $\Ac$ such that the reflector $I'\colon \Ac\to \Bc'$ is protoadditive.   Then, for any $n\geq 1$ and any $n$-fold $\Nc$-extension $A$ in $\Ac$ with ``initial ribs $a_i$" ($1\leq i\leq n$), there is an identity
\[
[A]_{n,\Bc\cap \Bc'}=[A]_{n,\Bc}\vee [\bigcap_{1\leq i\leq n}K[a_i]]_{\Bc'}.
\]
\end{theorem}
The functors $[-]_{n,\Cc}$ were used in \cite{EGV} to define the Hopf formulae for homology. Hence, the previous two theorems give us a simple description of the Hopf formulae: we now recall their definition.

As before, we consider a semi-abelian category $\Ac$. By a \emph{projective presentation} of an object $A\in\Ac$ we mean a normal epimorphism $p\colon P\to A$ such that $P$ is projective with respect to normal epimorphisms: for any normal epimorphism $f\colon B\to C$ the map $\hom_{\Ac}(P,f)\colon \hom_{\Ac}(P,B)\to \hom_{\Ac}(P,C)$ obtained by postcomposing with $f$ is surjective. We shall assume, from now on, that $\Ac$ has \emph{enough projectives}, i.e. that there exists a projective presentation of any object $A\in\Ac$. As before, we consider a Birkhoff subcategory $\Bc$ of $\Ac$ with reflector $I\colon \Ac\to\Bc$.  Then for an object $A\in\Ac$ with projective presentation $p\colon P\to A$ the \emph{Hopf formula} for the second homology was defined in \cite{EverVdL1} as the quotient  
\begin{equation}\label{Hopfformula}
\frac{[P]_{\Bc}\cap K[p]}{[p]_{1,\Bc}}
\end{equation}
As was shown in \cite{EverVdL1,EverVdL2} this object is independent, up to isomorphism, of the chosen projective presentation of $A$ and, when $\Ac$ is monadic over $\Set$, is isomorphic to the first Barr-Beck derived functor \cite{Barr-Beck} of $I$ in $A$ for the associated comonad on $\Ac$.  We shall denote the quotient \eqref{Hopfformula} by  $H_2(A,\Bc)$.

\begin{example}
When $\Ac$ is the variety of groups and $\Bc$ the subvariety of abelian groups, then the above defined ``Hopf formula" coincides with the classical Hopf formula for the second (integral) homology of a group $A$.
\end{example}

For $n\geq 1$, an \emph{$n$-fold projective presentation} of an object $A\in\Ac$ is an $n$-fold $\Nc$-extension $P$ such that the ``bottom vertex" in the diagram of $P$ (an $n$-dimensional cube in $\Ac$) is $A$ and all other ``vertices" are projective objects. It is easily seen (see \cite{Ev}) that such an $n$-fold projective presentation exists for every object $A$ as soon as $\Ac$ has enough projectives. One defines (see \cite{Ev,EGV}) the \emph{Hopf formula for the $(n+1)$st homology} of $A$ as the quotient
\[
\frac{[P_{\textrm{top}}]_{\Bc}\cap \bigcap_{1\leq i\leq n}K[p_i]}{[P]_{n,\Bc}}
\]
where $P_{\textrm{top}}$ denotes the projective object that appears as the ``top vertex" in the diagram of $P$, and the $p_i$ ($1\leq i\leq n$) denote the ``initial ribs": the $n$ morphisms starting from $P_{\textrm{top}}$. Once again, this quotient is independent, up to isomorphism, of the choice of $n$-fold presentation $P$ of $A$ (see \cite{Ev}), and when $\Ac$ is monadic over $\Set$, it is isomorphic to the first Barr-Beck derived functor \cite{Barr-Beck} of $I$ in $A$ for the associated comonad on $\Ac$ (see \cite{EGV}).  It will be denoted by $H_{n+1}(A,\Bc)$.

\begin{corollary}\label{compositehopf}
With the same notations and assumptions as in Theorem \ref{compositecommutator}, and with the extra assumption that $\Ac$ has enough projectives, we have, for any object $A\in\Ac$ and $n\geq 1$, the identity
\[
H_{n+1}(A,\Cc)=\frac{[P_{\textrm{top}}]_{\Cc}\cap \bigcap_{1\geq i\geq n}K[p_i]}{[P]_{n,\Bc}\vee \overline{[\bigcap_{1\leq i\leq n}K[p_i]]}^{P_{\textrm{top}}}_{\Cc}}
\]
where $P$ is an arbitrary $n$-fold presentation of $A$, with ``top" object $P_{\textrm{top}}$ and ``initial ribs" $p_i$ ($1\leq i\leq n$).
\end{corollary}

\begin{corollary}\label{compositehopf2}
With the same notations and assumptions as in Theorem \ref{compositeintersection}, and with the extra assumption that $\Ac$ has enough projectives, we have, for any object $A\in\Ac$ and $n\geq 1$, the identity
\[
H_{n+1}(A,\Bc\cap\Bc')=\frac{([P_{\textrm{top}}]_{\Bc}\vee [P_{\textrm{top}}]_{\Bc'}) \cap \bigcap_{1\geq i\geq n}K[p_i]}{[P]_{n,\Bc}\vee [\bigcap_{1\leq i\leq n}K[p_i]]_{\Bc'}}
\]
where $P$ is an arbitrary $n$-fold presentation of $A$, with ``top" object $P_{\textrm{top}}$ and ``initial ribs" $p_i$ ($1\leq i\leq n$).
\end{corollary}

We conclude this section with some examples of situations where Theorems \ref{compositecommutator} and \ref{compositeintersection} and Corollaries \ref{compositehopf} and  \ref{compositehopf2} apply.
\vspace{5mm}

\noindent {\bf Groups with coefficients in abelian Burnside groups.}
An example of the situation of Corollary \ref{compositehopf} is provided by any abelian Birkhoff subcategory $\Cc$ of a semi-abelian category $\Ac$, by taking for $\Bc$ the category of abelian objects in $\Ac$  \cite{BG1}. Indeed, in this case the reflector $\Bc\to \Cc$ is necessarily additive, hence protoadditive.  For instance, $\Ac$ could be the variety $\Gp$ of groups and $\Cc$ the Burnside subvariety $B_k$ of abelian  groups of exponent $k$ ($k\geq 1$), which consists of all abelian groups $A$ such that $ka=a+\dots +a=0$ for every element $a\in A$:
 \[
 \xymatrix@=30pt{
{\Gp \, } \ar@<1ex>[r]_-{^{\perp}}^-{ab} & {\, \Ab \, }
\ar@<1ex>[l]^H  \ar@<1ex>[r]_-{^{\perp}}^-{J} & B_k \ar@<1ex>[l]^G   }
 \]
Let us denote, for any group $A$, the (normal) subgroup $\{ ka | a\in A\}$ by $kA$. Then from Lemma \ref{intersectionlemma} (with $\Bc=\Ab$ and $\Bc'$ the Burnside variety of arbitrary groups of exponent $k$, not necessarily abelian) we infer that $[A]_{B_k}$ is the (internal) product $kA\cdot [A,A]$ of $kA$ with the (ordinary) commutator subgroup $[A,A]$ of $A$. Since we have, for any $n\geq 1$, a description of the radical $[-]_{n,\Ab}$ in terms of group commutators (see \cite{EGV}), Corollary \ref{compositehopf} provides us with a description of the Hopf formulae. For instance, for $n=1$, we obtain, for any group $A$ and projective presentation $p\colon P\to A$ of $A$:
 \[
 H_2(A,B_k)=\frac{(kP\cdot [P,P])\cap K[p]}{[K[p],P]\cdot kK[p]},
 \]
where the symbol $\cdot$ denotes the usual product of subgroups. Note that $kK[p]$ is a normal subgroup of $P$, and that the product of normal subgroups gives the \emph{supremum} as normal subgroups, in this situation.
\vspace{3mm}

 \noindent {\bf Semi-abelian compact algebras with coefficients in totally disconnected compact algebras.}
Let $\mathbb T$ be a semi-abelian theory. By considering the abelian objects in the semi-abelian category $\TCom$ of compact (Hausdorff) algebras we get the Birkhoff subcategory ${\Ab (\TCom) }$ of $\TCom$, called the category of \emph{abelian compact algebras} \cite{BC}.

The abelianisation functor $\ab \colon \TCom \rightarrow \Ab(\TCom)$ sends an algebra $A$ to its quotient $A/\overline{[A,A]}$ by the (topological) closure $\overline{[A,A]}$ in $A$ of the ``algebraic'' commutator $[A,A]$ computed in the semi-abelian variety $\mathsf{Set}^{\mathbb T}$. We then have the following Birkhoff reflection
\begin{equation}\label{{abel}}
\xymatrix{
{\TCom }\,\, \ar@<1ex>[r]^-{\ab} & {\Ab (\TCom) }
\ar@<1ex>[l]^-{V}_-{_{\perp}}}
\end{equation}
where $V$ is the inclusion functor.
 In general, the categorical commutator (in the sense of Huq \cite{Huq}, see also \cite{BB}) of two normal closed subalgebras is simply given by the (topological) closure of the ``algebraic'' commutator in the corresponding category $\mathsf{Set}^{\mathbb T}$ of algebras:
\begin{lemma}\label{closurecommutator} 
Let $h\colon H\to A$ and $k\colon K\to A$ be two normal closed subalgebras of a compact algebra $A$. Then the commutator of $H$ and $K$ is given by 
$$[H,K]_{\TCom}=\overline{[H,K]}_{\mathsf{Set}^{\mathbb{T}}}.$$
\end{lemma}
\begin{proof}
By using the fact that the canonical morphism $H+K\to H\times K$ is an open surjection, it is easy to see that any morphism $\varphi\colon H\times K\to A$ in the category $\mathsf{Set}^{\mathbb{T}}$ such that $\varphi\circ (1_H,0)=h$ and $\varphi\circ (0,1_K)=k$ is also a morphism in the category $\TCom$. We now show that the quotient $q\colon A\to A/\overline{[H,K]}_{\mathsf{Set}^{\mathbb{T}}}$ is universal in making $H$ and $K$ commute. On the one hand, since $[H,K]_{\mathsf{Set}^{\mathbb T}}\subseteq \overline{[H,K]}_{\mathsf{Set}^{\mathbb T}}$, one certainly has that $q(H)$ and $q(K)$ commute in $\mathsf{Set}^{\mathbb{T}}$, hence in $\TCom$. On the other hand, given any other quotient $f\colon A\to B$ in $\TCom$ such that $f(H)$ and $f(K)$ commute, we have that 
$$f(\overline{[H,K]}_{\mathsf{Set}^{\mathbb T}})  \subseteq  \overline{f[H,K]}_{\mathsf{Set}^{\mathbb T}}=  \overline{[f(H),f(K)]}_{\mathsf{Set}^{\mathbb T}} =\overline{0}=0,$$
from which it follows that there is a unique $a\colon A/\overline{[H,K]}_{\mathsf{Set}^{\mathbb T}}\to B$ such that
$a\circ q=f$.
\end{proof}


We obtain an instance of the situation of Corollary \ref{compositehopf2} by choosing 
  $\Ac$ to be the category $\TCom$, $\Bc$ the category $\Ab(\TCom)$ of abelian compact algebras, and $\Bc'$ the category $\mathsf{TotDis}^{\mathbb T}$ of compact totally disconnected algebras. The intersection $\Bc\cap\Bc'$ in this case is the category $\Ab (\TPro)$ of abelian totally disconnected algebras.
We know from Example \ref{exproto}.\ref{exdisc} that the reflector $\TCom\to\mathsf{TotDis}^{\mathbb T}$ is protoadditive,
and the category $\TCom$ has enough regular projectives, since it is monadic over the category of sets (see \cite{Man, BC}). Hence, we can indeed apply Corollary \ref{compositehopf2}: for instance, given a projective presentation $p \colon P \rightarrow A$ of a compact algebra $A$, the second homology algebra $H_2(A, \Ab (\TPro))$ of $A$ is given by:
$$H_2(A, \Ab (\TPro) )= \frac{(\overline{[P,P]} \vee \Gamma_0 (P) )\cap K[p]}{\overline{[K[p] , P]} \vee \Gamma_0 (K[p]) } ,$$
where we have used Lemma \ref{closurecommutator} to compute the denominator.

For some specific algebraic theories we can give a description of higher dimensional homology objects via Hopf formulae. For instance, let $\mathbb{T}$ be the theory of groups, so that $\Ac={\mathsf{Grp(HComp)}}$ is the category of compact groups, $\Bc={\mathsf{Ab(HComp)}}$ the category of abelian compact groups, $\Bc'$ the category of profinite groups (since a topological group is totally disconnected and compact if and only if it is profinite) and $\Bc\cap \Bc'=\mathsf{Ab(Prof)}$ the category of abelian profinite groups. We can then consider a double projective presentation
$$
\label{doublext}
\xymatrix{F \ar[r]^{} \ar[d] & F/K_1 \ar[d] \\
F/K_2  \ar[r] & G}
$$
of a semi-abelian compact group $G$, so that $K_1$ and $K_2$ are closed normal subgroups of a free compact group $F$ with the property that both $F/K_1$ and $F/K_2$ are free. Then the third homology group of $G$ with coefficients in $\mathsf{Ab(Prof)}$ is given by
$$H_3 (G, {\mathsf{Ab(Prof)} } ) = \frac{ (\overline{[F,F]}\cdot \Gamma_0(F))\cap K_1 \cap K_2}{\overline{[K_1,K_2]}\cdot \overline{[K_1 \cap K_2, F]}\cdot \Gamma_0(K_1 \cap K_2)},$$
where the symbol $\cdot$ denotes the product of normal subgroups, and the closure is the topological closure.
\vspace{3mm}

\noindent {\bf Internal groupoids with coefficients in abelian objects.}
Let $\Ac$ be a semi-abelian category with enough regular projectives and $\Gpd(\Ac)$ the category of internal groupoids in $\Ac$. We obtain another instance of Corollary \ref{compositehopf2}, by taking for Birkhoff subcategories of  $\Gpd(\Ac)$ the category $\mathsf{Ab}(\Gpd(\Ac))$ of abelian objects in the category of groupoids in $\Ac$ and $\Ac$ (via the discrete functor $D\colon \Ac\to \Gpd(\Ac)$). Their intersection is the category $\Ab(\Ac)$ of abelian objects of $\Ac$. We know from Example \ref{exproto}.\ref{exgroupoids} that the connected components functor $\pi_0\colon \Gpd(\Ac)\to \Ac$ is protoadditive, and it was shown in \cite{EG} that the category ${\Gpd(\Ac) }$ has enough regular projectives whenever $\Ac$ has enough regular projectives. Hence, we can apply Corollary \ref{compositehopf2} in this situation. For instance, let us consider a projective presentation $p=(p_0,p_1)  \colon P \rightarrow A$
$$
\xymatrix{P_1 \ar[r]^{p_1}   \ar@<-1.2 ex>[d]_{d}  \ar@<+1.2 ex>[d]^{c}& A_1 \ar@<-1.2 ex>[d]_{d}  \ar@<+1.2 ex>[d]^{c} \\
P_0 \ar[r]_{p_0} \ar[u] & A_0 \ar[u]
}
$$
of an internal groupoid $A= (A_1, A_0, m,d,c,i)$. We write $\sem{P,P}$ for the internal groupoid $([P_1, P_1], [P_0,P_0], \overline{m}, \overline{d}, \overline{c}, \overline{i})$, where the arrows $\overline{d}, \overline{c}$ and $\overline{i}$ are the restrictions of $d,c$ and $i$ to the largest commutators $[P_1,P_1]$ and $[P_0,P_0]$ of $P_1$ and $P_0$, respectively, and $\overline{m}$ the induced groupoid composition. In other words, $\sem{P,P}$ is the kernel, in the category ${\Gpd(\Ac)}$, of the quotient sending the groupoid $P$ to a reflexive graph in $\mathsf{Ab}(\Ac)$, universally (recall from \cite{Jon} that the category $\mathsf{\mathsf{Ab}(Gpd}(\Ac))$ is isomorphic to the category of reflexive graphs in $\mathsf{Ab}(\Ac)$). Similar notation is used for the groupoid $\sem{K[p],P}$.

If we write $\Gamma_0 (P)$ and $\Gamma_0 (K[p])$ for the full subgroupoids of the connected components of $0$ in $P$ and in $K[p]$, respectively, then we can express the second homology groupoid of $A$ with coefficients in $\Ab(\Ac)$ as the quotient
 \[
 H_2(A,\mathsf{Ab}(A))=\frac{K[p] \cap (\Gamma_0(P) \vee \sem{P,P}) }{\sem{K[p],P} \vee \Gamma_0(K[p])}.
 \]
Note that $\vee$ indicates the supremum, as normal subobjects, in the category $\mathsf{Gpd} (\Ac)$. 

Now, let $\mathsf{Gpd}^k (\Ac)$ denote the category of $k$-fold internal groupoids in $\Ac$, defined inductively by $\mathsf{Gpd}^k (\Ac) = \mathsf{Gpd} ( \mathsf{Gpd}^{k-1} (\Ac))$. It is clear that, also for $k\geq 2$, $\Ac$ is a Birkhoff subcategory of $\Gpd^k(\Ac)$ with protoadditive reflector $\pi_0\circ \dots \circ \pi_0^k$,
 \[
\xymatrix{
 \mathsf{Gpd}^k({\Ac }) \ar@<1ex>[r]^-{\pi_0^k} & {\mathsf{Gpd}^{k-1}({\Ac })\quad }  \ar@<1ex>[l]^-{D^k}_-{_{\perp}}   {\cdots \quad }  \mathsf{Gpd}^2({\Ac }) \ar@<1ex>[r]^-{\pi_0^2} & {\Gpd(\Ac) }\,\, \ar@<1ex>[r]^-{\pi_0}  \ar@<1ex>[l]^-{D^2}_-{_{\perp}}  & {\Ac, \, } \ar@<1ex>[l]^-{D}_-{_{\perp}}
}
\]
and that ${\Gpd^k(\Ac)}$ has enough regular projectives. Hence, Corollary \ref{compositehopf2} provides us, for any $k\geq 1$, with a description of the homology objects of $k$-fold internal groupoids with coefficients in $\mathsf{Ab}(\Ac)=\Ac\cap\Ab(\Gpd^k(\Ac))$, similar to the one above.

\end{document}